\documentclass{amsart}

\usepackage[colorlinks=true, breaklinks=true, urlcolor=webbrown, linkcolor=RoyalBlue, citecolor=webgreen, backref=page]{hyperref}
\usepackage[nochapters,beramono,eulermath,pdfspacing,dottedtoc]{classicthesis1}
\usepackage{epsfig, graphicx}
\usepackage{verbatim,setspace}
\usepackage{amssymb}

\rohead{\mbox{\color{halfgray} \rightmark\hfil \hspace{0.5em} \rlap{\small \vline \kern1em\color{black}\thepage}}}
\lehead{\mbox{\llap{\small\thepage\kern1em\color{halfgray} \vline}\color{halfgray}\hspace{0.5em}\rightmark\hfil}} % The header style
\newcommand{\R}		{\mathbb{R}}
\newcommand{\C}		{\mathbb{C}}
\newcommand{\N}		{\mathbb{N}}

\newcommand{\diag}{\mathsf{diag}}
\newcommand{\supp}{\mathsf{supp}}
\newcommand{\im}{\mathsf{Im}}
\newcommand{\re}{\mathsf{Re}}

\renewcommand{\arg}{\mathsf{arg}}
\renewcommand{\det}{\mathsf{det}}
\renewcommand{\deg}{\mathsf{deg}}
\newcommand{\RS}{\boldsymbol{\mathfrak{R}}}

\newcommand{\z}	{{\boldsymbol z}}
\newcommand{\w}	{{\boldsymbol w}}
\newcommand{\x}	{{\boldsymbol x}}
\newcommand{\e}	{{\boldsymbol e}}
\newcommand{\ugamma}{\boldsymbol\gamma}
\newcommand{\rhy}   {\textnormal{RHP}-${\boldsymbol Y}$}
\newcommand{\rhx}   {\textnormal{RHP}-${\boldsymbol X}$}
\newcommand{\rhn}   {\textnormal{RHP}-${\boldsymbol N}$}
\newcommand{\rhp}   {\textnormal{RHP}-${\boldsymbol P}_e$}
\newcommand{\rhr}   {\textnormal{RHP}-${\boldsymbol Z}$}
\newcommand{\rhpsiAB}   {\textnormal{RHP}-${\boldsymbol \Psi_{\alpha,\beta}}$}
\newcommand{\rhwpsiAB}   {\textnormal{RHP}-${\widetilde{\boldsymbol \Psi}_{\alpha,\beta}}$}
\newcommand{\rhhpsiAB}   {\textnormal{RHP}-${\widehat{\boldsymbol \Psi}_{\alpha,\beta}}$}
\newcommand{\rhhpsiA}   {\textnormal{RHP}-${\widehat{\boldsymbol \Psi}_\alpha}$}

\newcommand{\rhpsiA}   {\textnormal{RHP}-${\boldsymbol \Psi_\alpha}$}
\newcommand{\rhphiAB}   {\textnormal{RHP}-${\boldsymbol \Phi_{\alpha,\beta}}$}

\newtheorem{theorem}{Theorem}
\newtheorem{proposition}[theorem]{Proposition}
\newtheorem{corollary}[theorem]{Corollary}
\newtheorem{lemma}[theorem]{Lemma}

\begin{document}

\title[Hermite-Pad\'e Approximants for Angelesco Systems]{Strong Asymptotics of Hermite-Pad\'e Approximants for Angelesco Systems}

\author[M. Yattselev]{Maxim L. Yattselev}

\address{Department of Mathematical Sciences, Indiana University-Purdue University Indianapolis, 402~North Blackford Street, Indianapolis, IN 46202}

\email{\href{mailto:maxyatts@math.iupui.edu}{maxyatts@math.iupui.edu}}

\subjclass[2000]{42C05, 41A20, 41A21}

\keywords{Hermite-Pad\'e approximation, multiple orthogonal polynomials, non-Hermitian orthogonality, strong asymptotics, matrix Riemann-Hilbert approach}

\begin{abstract}
In this work type II Hermite-Pad\'e approximants for a vector of Cauchy transforms of smooth Jacobi-type densities are considered. It is assumed that densities are supported on mutually disjoint intervals (an Angelesco system with complex weights). The formulae of strong asymptotics are derived for any ray sequence of multi-indices.
\end{abstract}

\maketitle

\setcounter{tocdepth}{2}
\tableofcontents

\section{Introduction}
\label{sec:intro}

Let $\vec f=(f_1,\ldots,f_p)$, $p\in\N$, be a vector of germs of holomorphic functions at infinity. Given a multi-index $\vec n\in\N^p$, \emph{Hermite-Pad\'e} approximant to $\vec f$ associated with $\vec n$, is a vector of rational functions
\begin{equation}
\label{HermitePade}
[\:\vec n\:]_{\vec f} := \left(P_{\vec n}^{(1)}/Q_{\vec n},\ldots,P_{\vec n}^{(p)}/Q_{\vec n}\right)
\end{equation}
such that
\begin{equation}
\label{linsys}
\left\{
\begin{array}{l}
\displaystyle \deg\big(Q_{\vec n}\big) = |\:\vec n\:| := n_1+\cdots+n_p, \medskip \\
\displaystyle R_{\vec n}^{(i)}(z) := \left(Q_{\vec n} f_i - P_{\vec n}^{(i)}\right)(z) = \mathcal{O}\left(z^{-(n_i+1)}\right) \quad \text{as} \quad z\to\infty, \quad i\in\{1,\ldots,p\}.
\end{array}
\right.
\end{equation}
It is quite simple to see that $[\:\vec n\:]_{\vec f} $ always exists since \eqref{linsys} can be rewritten as a linear system that has more unknowns than equations with coefficients coming from the Laurent expansions of $f_i$'s at infinity. Hence, $Q_{\vec n}$ is never identically zero and, in what follows, we normalize $Q_{\vec n}$ to be monic.

The vector $\vec f$ is called an \emph{Angelesco system} if
\begin{equation}
\label{angelesco}
f_i(z) = \int\frac{\mathrm{d}\sigma_i(t)}{t-z}, \qquad i\in\{1,\ldots,p\},
\end{equation}
where $\sigma_i$'s are positive measures on the real line with mutually disjoint convex hulls of their supports, i.e., $[a_j,b_j]\cap[a_k,b_k]=\varnothing$ for $j\neq k$ where $[a_i,b_i]$ is the smallest interval containing $\supp(\sigma_i)$.  Hermite-Pad\'e approximants to such systems were initially considered by Angelesco \cite{Ang19} and later by Nikishin \cite{Nik79,Nik80}. The beauty of system \eqref{angelesco} is that $Q_{\vec n}$, the denominator of $[\:\vec n\:]_{\vec f}$, turns out to be a multiple orthogonal polynomial satisfying
\begin{equation}
\label{ortho}
\int Q_{\vec n}(x)x^k\mathrm{d}\sigma_i(x) =0, \qquad k\in\{0,\ldots,n_i-1\}, \quad i\in\{1,\ldots,p\}.
\end{equation}

When $p=1$, Hermite-Pad\'e approximant $[\:\vec n\:]_{\vec f}$ specializes to the diagonal Pad\'e approximant, quite often denoted by $[n/n]_f$. It was shown by Markov \cite{Mar95} that if $f$ is of the form \eqref{angelesco} (now called a \emph{Markov function}), then $[n/n]_f$ converge to $f$ locally uniformly outside of $[a,b]$. Moreover, see \cite[Thm.~6.1.6]{StahlTotik}, it holds that
\begin{equation}
\label{weakPade}
\left\{
\begin{array}{l}
\displaystyle \lim_{n\to\infty} n^{-1}\log\big| f - [n/n]_f\big| \leq -2\big(\ell-V^\omega\big) \medskip \\
\displaystyle  \lim_{n\to\infty} n^{-1}\log|Q_n| = -V^\omega
\end{array}
\right.
\end{equation}
locally uniformly in $\overline\C\setminus[a,b]$, where $V^\omega(z):=-\int\log|z-t|\mathrm{d}\omega(t)$ is the \emph{logarithmic potential} of $\omega$, while the measure $\omega$ and the constant $\ell$ are the unique solutions of the min/max problem:
\begin{equation}
\label{min-max}
\ell := \min_{x\in[a,b]}V^\omega(x) ~= \max_{\nu\in M_1(a,b)}\min_{x\in[a,b]}V^\nu(x),
\end{equation}
where $M_c(a,b)$ is the collection of all positive Borel measures of mass $c$ supported on $[a,b]$. In fact, it also holds that $\omega$ is the \emph{equilibrium distribution} and $\ell$ is the \emph{Robin's constant} for the interval $[a,b]$. That is, $\omega$ is the unique measure on $[a,b]$ that solves the energy minimization problem:
\begin{equation}
\label{I}
I[\omega] = \min_{\nu\in M_1(a,b)} I[\nu], \quad \quad \ell=I[\omega],
\end{equation}
where $I[\nu]:=-\int\int\log|z-t|\mathrm{d}\nu(t)\mathrm{d}\nu(z)=\int V^\nu\mathrm{d}\nu$ is the \emph{logarithmic energy} of $\nu$ (for the notions of logarithmic potential theory we use \cite{Ransford} and \cite{SaffTotik} as primary references). 

It easily follows from \eqref{min-max}--\eqref{I} and properties of the superharmonic functions that
\begin{equation}
\label{EqProp}
\left\{
\begin{array}{l}
\ell - V^\omega \equiv 0 \quad \text{on} \quad [a,b], \medskip \\
\ell- V^\omega >0 \quad \text{in} \quad \overline\C\setminus[a,b].
\end{array}
\right.
\end{equation}
Hence, the diagonal Pad\'e approximants $[n/n]_f$ do indeed converge to $f$ locally uniformly in $\overline\C\setminus[a,b]$. Moreover, if $\sigma$ is a \emph{regular} measure in the sense of Stahl and Totik \cite[Sec.~3.1]{StahlTotik} (in particular, $\sigma^\prime>0$ almost everywhere on $[a,b]$ implies regularity), then the inequality in \eqref{weakPade} can be replaced by equality.

The above results were extended by Gonchar and Rakhmanov \cite{GRakh81} to Hermite-Pad\'e  approximants for Angelesco systems when multi-indices are such that
\begin{equation}
\label{multi-indices}
n_i = c_i|\:\vec n\:|+o\left(|\:\vec n\:|\right), \qquad \vec c=(c_1,\ldots,c_p)\in\big(0,1)^p, \quad |\:\vec c\:|=1,
\end{equation}
as $|\:\vec n\:|\to\infty$, and the measures $\sigma_i$ satisfy $\sigma_i^\prime>0$ almost everywhere on $[a_i,b_i]$, $i\in\{1,\ldots,p\}$. The formulae for the errors of approximation are similar in appearance to \eqref{weakPade} with measures coming not from a scalar but from a vector minimum energy problem. To describe it, define
\[
M_{\vec c}\big(\{a_i,b_i\}_1^p\big):=\big\{\vec\nu=(\nu_1,\ldots,\nu_p):~\nu_i\in M_{c_i}(a_i,b_i), ~i\in\{1,\ldots,p\}\big\}.
\]
Then it is known that there exists the unique vector of measures $\vec\omega\in M_{\vec c}\big(\{a_i,b_i\}_1^p\big)$ such that
\begin{equation}
\label{vecI}
I[\:\vec\omega\:] = \min_{\nu\in M_{\vec c}(\{a_i,b_i\}_1^p)} I[\:\vec\nu\:], \qquad I[\:\vec\nu\:] := \sum_{i=1}^p\bigg(2I[\nu_i] + \sum_{k\neq i}I[\nu_i,\nu_k]\bigg),
\end{equation}
where $I[\nu_i,\nu_k]:=-\int\int\log|z-t|\mathrm{d}\nu_i(t)\mathrm{d}\nu_k(z)$. The measures $\omega_i$ might no longer be supported on the whole intervals $[a_i,b_i]$ (the so-called \emph{pushing effect}), but in general it holds that
\begin{equation}
\label{supports}
\supp(\omega_i)=[a_{\vec c,i},b_{\vec c,i}] \subseteq [a_i,b_i], \qquad i\in\{1,\ldots,p\}.
\end{equation}
Let $W^{\vec\nu}$ be a function on $\bigcup_{i=1}^p[a_i,b_i]$ such that its restriction to $[a_i,b_i]$ is equal to $V^{\nu_i+\nu}$ where $\nu=\sum_{i=1}^p\nu_i$ is a probability measure such that $\nu_{|[a_i,b_i]}=\nu_i$. Exactly as in \eqref{min-max}, the equilibrium vector measure $\vec\omega$ can be characterized by the following property: if
\begin{equation}
\label{min-max-sys}
\min_{x\in[a_i,b_i]}W^{\vec\nu}(x) \geq \min_{x\in[a_i,b_i]}W^{\vec\omega}(x) =: \ell_i
\end{equation}
simultaneously for all $i\in\{1,\ldots,p\}$ for some $\vec\nu\in M_{\vec c}\big(\{a_i,b_i\}_1^p\big)$, then $\vec\nu=\vec\omega$.

Having all the definitions at hand, we can formulate the main result of \cite{GRakh81}, which states that
\begin{equation}
\label{weakHermitePade}
\left\{
\begin{array}{l}
\displaystyle \lim_{|\:\vec n\:|\to\infty} |\:\vec n\:|^{-1}\log\big| f_i - P_{\vec n}^{(i)}/Q_{\vec n}\big| = -\big(\ell_i-V^{\omega_i+\omega}\big), \quad i\in\{1,\ldots,p\}, \medskip \\
\displaystyle \lim_{|\:\vec n\:|\to\infty} |\:\vec n\:|^{-1}\log|Q_{\vec n}| = -V^\omega
\end{array}
\right.
\end{equation}
locally uniformly in $\overline\C\setminus\bigcup_{i=1}^p[a_i,b_i]$\footnote{\eqref{weakHermitePade} is consistent with \eqref{weakPade} when $p=1$, since in this case $I[\vec\nu]=2I[\nu_1]$, $\ell_1=2\ell$, and $V^{\omega_1+\omega}=2V^\omega$.}. Even though \eqref{weakHermitePade} looks exactly as \eqref{weakPade}, the convergence properties of the approximants are not as straightforward. Indeed, it is a direct consequence of the pushing effect ($[a_{\vec c,i},b_{\vec c,i}]\subsetneq[a_i,b_i]$), when it occurs, of course, that the first relation in \eqref{EqProp} is replaced now by 
\begin{equation}
\label{EqPropi}
\left\{
\begin{array}{l}
\ell_i - V^{\omega_i+\omega} \equiv 0 \quad \text{on} \quad [a_{\vec c,i},b_{\vec c,i}], \medskip \\
\ell_i- V^{\omega_i+\omega} < 0 \quad \text{on} \quad [a_i,b_i]\setminus[a_{\vec c,i},b_{\vec c,i}].
\end{array}
\right.
\end{equation}
Further, set
\begin{equation}
\label{Dipm}
\left\{
\begin{array}{lll}
D_i^+ &:=& \displaystyle \big\{z:~\ell_i-V^{\omega_i+\omega}(z)>0\big\}, \medskip \\
D_i^- &:=& \displaystyle \big\{z:~\ell_i-V^{\omega_i+\omega}(z)<0\big\}.
\end{array}
\right.
\end{equation}
\begin{figure}[!ht]
\centering
\includegraphics[scale=.55]{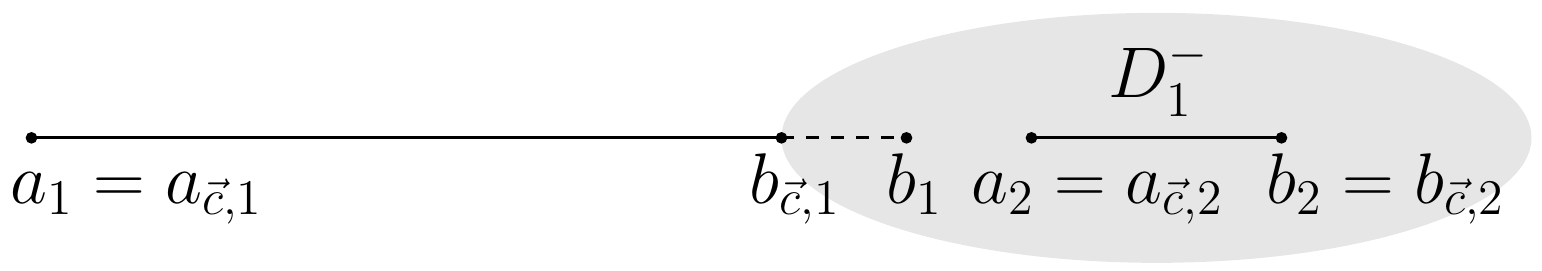}
\caption{\small Schematic representation of the pushing effect in the case of 2 intervals (in Proposition~\ref{prop:derivative} we shall show that this is the only possible situation for pushing effect in the case of 2 intervals; this is also explained in \cite{GRakh81}). The shaded region is the divergence domain $D_1^-$ while $D_2^-=\varnothing$.}
\label{fig:1}
\end{figure}
Properties of the logarithmic potentials immediately imply that $D_i^+$ is an unbounded domain. This is exactly the domain in which the approximants $P_{\vec n}^{(i)}/Q_{\vec n}$ converge to $f_i$ locally uniformly, while $D_i^-$ is a bounded open set on which the approximants diverge to infinity. This set can be empty or not. The latter situation necessarily happens when $[a_{\vec c,i},b_{\vec c,i}]\subsetneq[a_i,b_i]$ as can be clearly seen from the second line in \eqref{EqPropi}; however, the pushing effect is not necessary for the divergence set to exist.

The result of Gonchar and Rakhmanov \eqref{weakHermitePade} belongs to the realm of the so called \emph{weak asymptotics} as to distinguish from \emph{strong asymptotics} in which one establishes the existence of and identifies the limits
\begin{equation}
\label{strongHermitePade}
\left\{
\begin{array}{l}
\displaystyle \lim_{|\:\vec n\:|\to\infty}\left(\log\big| f_i - P_{\vec n}^{(i)}/Q_{\vec n}\big| + |\:\vec n\:|\big(\ell_i - V^{\omega_i+\omega}\big)\right), \medskip \\
\displaystyle \lim_{|\:\vec n\:|\to\infty}\big(\log|Q_{\vec n}|+|\:\vec n\:|V^\omega\big).
\end{array}
\right.
\end{equation}
Not surprisingly, the first result completely answering the previous question was obtained for Pad\'e approximants ($p=1$) by Szeg\H{o}. He proved that limit \eqref{strongHermitePade} takes place exactly when $\sigma^\prime$ satisfied $\int\log\sigma^\prime\mathrm{d}\omega>-\infty$, which is now known as \emph{Szeg\H{o} condition}\footnote{The word ``completely'' is slightly abused here as it was later realized by \cite{Ver36} that one can add any singular measure to $\sigma^\prime(t)\mathrm{d}t$, the absolutely continuous part, without changing \eqref{strongHermitePade}.}. The analog of the Szeg\H{o} theorem for true Hermite-Pad\'e approximants was proven by Aptekarev \cite{Ap88} when $p=2$ and the multi-indices are diagonal ($\vec n=(n,n)$) with indications how one could carry the approach to any $p>1$. A rigorous proof for any $p$ and diagonal multi-indices was completed by Aptekarev and Lysov \cite{ApLy10} for systems $\vec f$ of Markov functions generated by cyclic graphs (the so called \emph{generalized Nikishin systems}), of which Angelesco systems are a particular example. The restriction on the measures $\sigma_i$ is more stringent in \cite{ApLy10}, as it is required that
\begin{equation}
\label{rho}
\sigma_i^\prime(x) = h_i(x)(x-a_i)^{\alpha_i}(b_i-x)^{\beta_i}, \qquad \alpha_i,\beta_i>-1,
\end{equation}
 and $h_i$ is holomorphic and non-vanishing in some neighborhood of $[a_i,b_i]$. 
 
From the approximation theory point of view it is not natural to require the measures $\sigma_i$ to be positive (as well as to be supported on the real line, but we shall not dwell on this point here). In the case of Pad\'e approximants it was Nuttall \cite{Nut90} who proved the existence of and identified the limit in \eqref{strongHermitePade} for the set up \eqref{angelesco} and \eqref{rho} with $\alpha=\beta=-1/2$ and $h$ being H\"older continuous, non-vanishing, and complex-valued on $[a,b]$. The proof of Szeg\H{o} theorem for any parameters $\alpha,\beta>-1$ and $h$ complex-valued, holomorphic, and non-vanishing around $[a,b]$ follows from Aptekarev \cite{Ap02} (this result was not the main focus of \cite{Ap02}, there weighed approximation on one-arc S-contours was considered), and the condition of holomorphy of $h$ was relaxed by Baratchart and the author in \cite{BY10}, where $h$ is taken from a fractional Sobolev space that depends on the parameters $\alpha,\beta$ (again, the main focus of \cite{BY10} was weighted (multipoint) Pad\'e approximation on one-arc S-contours). The goal of this work is to extend the results of \cite{ApLy10} to Angelesco systems with complex weights and Hermite-Pad\'e approximants corresponding to multi-indices as in~\eqref{multi-indices}.

\section{Main Results}
\label{sec:MR}

From now on, we fix a system of mutually disjoint intervals $\big\{[a_i,b_i]\big\}_{i=1}^p$ and a vector $\vec c\in(0,1)^p$ such that $|\:\vec c\:|=1$. We further denote by
\[
\vec\omega=(\omega_1,\ldots,\omega_p), \quad \omega:=\sum_{i=1}^p\omega_i, \quad \supp(\omega_i)=\big[a_{\vec c,i},b_{\vec c,i}\big]\subseteq[a_i,b_i],
\]
the equilibrium vector measure minimizing the energy functional \eqref{vecI}.

To describe the forthcoming results we shall need a $(p+1)$-sheeted compact Riemann surface, say $\RS$, that we realize in the following way. Take $p+1$ copies of $\overline\C$. Cut one of them along the union $\bigcup_{i=1}^p\big[a_{\vec c,i},b_{\vec c,i}\big]$, which henceforth is denoted by $\RS^{(0)}$. Each of the remaining copies cut along only one interval $\big[a_{\vec c,i},b_{\vec c,i}\big]$ so that no two copies have the same cut and denote them by $\RS^{(i)}$. To form $\RS$, take $\RS^{(i)}$ and glue the banks of the cut $\big[a_{\vec c,i},b_{\vec c,i}\big]$ crosswise to the banks of the corresponding cut on $\RS^{(0)}$. 

It can be easily verified that thus constructed Riemann surface has genus 0. Denote by $\pi$ the natural projection from $\RS$ to $\overline\C$. We shall denote by $\z,\w,\x,\e$ generic points on $\RS$ with natural projections $z,w,x,e$. We also shall employ the notation $z^{(i)}$ for a point on $\RS^{(i)}$ with $\pi(z^{(i)})=z$. This notation is well defined everywhere outside of the cycles $\boldsymbol\Delta_i:=\RS^{(0)}\cap\RS^{(i)}$. Clearly,  $\pi(\boldsymbol\Delta_i)=\big[a_{\vec c,i},b_{\vec c,i}\big]$.  It also will be convenient to denote by $\boldsymbol a_{\vec c,i}$ and $\boldsymbol b_{\vec c,i}$ the branch points of $\RS$ with respective projections $a_{\vec c,i}$ and $b_{\vec c,i}$, $i\in\{1,\ldots,p\}$.

Unfortunately, to be able to handle general multi-indices of form \eqref{multi-indices}, one Riemann surface is not sufficient. Let $\vec n\in\N^p$. Denote by
\[
\vec\omega_{\vec n} = (\omega_{\vec n,1},\ldots,\omega_{\vec n,p}), \quad \omega_{\vec n}:=\sum_{i=1}^p\omega_{\vec n,i}, \quad \supp(\omega_{\vec n,i})=\big[a_{\vec n,i},b_{\vec n,i}\big]\subseteq[a_i,b_i],
\]
the equilibrium vector measure minimizing the energy functional \eqref{vecI} where $\vec c$ is replaced by the vector $\big(n_1/|\:\vec n\:|,\ldots,n_p/|\:\vec n\:|\big)$. The surface $\RS_{\vec n}$ is defined absolutely analogously to $\RS$. Notation $\boldsymbol\Delta_{\vec n,i}$, $\boldsymbol a_{\vec n,i}$, and $\boldsymbol b_{\vec n,i}$, $i\in\{1,\ldots,p\}$, is self-evident now.

Since each $\RS_{\vec n}$ has genus zero, one can arbitrarily prescribe zero/pole multisets of rational functions on $\RS_{\vec n}$ as long as the multisets have the same cardinality. Thus, given a multi-index $\vec n$, we shall denote $\Phi_{\vec n}$ a rational function on $\RS_{\vec n}$ which is non-zero and finite everywhere on $\RS_{\vec n}\setminus\bigcup_{k=0}^p\big\{\infty^{(k)}\big\}$, has a pole of order $|\:\vec n\:|$ at $\infty^{(0)}$, a zero of multiplicity $n_i$ at each $\infty^{(i)}$, and satisfies
\begin{equation}
\label{normalization}
\prod_{k=0}^p\Phi_{\vec n}\big(z^{(k)}\big) \equiv 1.
\end{equation}
Normalization \eqref{normalization} is possible since the function $\log\prod_{k=0}^p\big|\Phi_{\vec n}\big(z^{(k)}\big)\big|$ extends to a harmonic function on $\C$ which has a well defined limit at infinity. Hence, it is constant. Therefore, if \eqref{normalization} holds at one point, it holds throughout $\overline\C$. The importance of the function $\Phi_{\vec n}$ to our analysis lies in the following proposition.

\begin{proposition}
\label{prop:phi-vecn}
With the above notation, it holds that
\[
\frac1{|\:\vec n\:|}\log\big|\Phi_{\vec n}(\z)\big| = 
\left\{
\begin{array}{ll}
-V^{\omega_{\vec n}}(z) + \frac1{p+1}\sum_{k=1}^p\ell_{\vec n,k}, & \z\in\RS_{\vec n}^{(0)}, \medskip \\
V^{\omega_{\vec n,i}}(z) - \ell_{\vec n,i} + \frac1{p+1}\sum_{k=1}^p\ell_{\vec n,k}, & \z\in\RS_{\vec n}^{(i)}, \quad i\in\{1,\ldots,p\}.
\end{array}
\right.
\]
If a sequence $\big\{\vec n\big\}$ satisfies \eqref{multi-indices}, then the measures $\omega_{\vec n}$ converge to $\omega$ in the weak$^*$ topology of measures as $|\:\vec n\:|\to\infty$ (in particular, this implies that $\ell_{\vec n,i}\to\ell_i$, $a_{\vec n,i}\to a_{\vec c,i}$, and $b_{\vec n,i}\to b_{\vec c,i}$). Moreover, it holds that $V^{\omega_{\vec n,i}}\to V^{\omega_i}$ uniformly on compact subsets of $\C$ for each $i\in\{1,\ldots,p\}$.
\end{proposition}

It immediately follows from Proposition~\ref{prop:phi-vecn} that
\begin{equation}
\label{Phi-ratio}
\frac1{|\:\vec n\:|}\log\left|\frac{\Phi_{\vec n}\big(z^{(i)}\big)}{\Phi_{\vec n}\big(z^{(0)}\big)}\right| = V^{\omega_{\vec n,i}+\omega_{\vec n}}(z) - \ell_{\vec n,i} = V^{\omega_i+\omega}(z) - \ell_i + o(1)
\end{equation}
uniformly on compact subsets of $\C$ as $|\:\vec n\:|\to\infty$ for each $i\in\{1,\ldots,p\}$.

The following corollary is an elementary consequence of Proposition~\ref{prop:phi-vecn}. It describes the assumption with which \eqref{multi-indices} often replaced when strong asymptotics is discussed (most often $\vec k=(1,\ldots,1)$).

\begin{corollary}
Let $\vec k\in\N^p$. If $\vec c = \big(k_1/|\:\vec k\:|, \ldots, k_p/|\:\vec k\:|\big)$ and $\vec n= n\vec k$, $n\in\N$, then $\vec\omega_{\vec n}=\vec\omega$ and $\Phi_{\vec n}=\Phi_{\vec k}^n$.
\end{corollary}

Proposition~\ref{prop:phi-vecn} allows to recover $|\Phi_{\vec n}|$ via the vector equilibrium measure $\vec\omega_{\vec n}$. In order to do it for the function $\Phi_{\vec n}$ itself, let us define $h_{\vec n}$ on $\RS_{\vec n}$ by the rule
\begin{equation}
\label{h}
\left\{
\begin{array}{ll}
h_{\vec n}\big(z^{(0)}\big) := \displaystyle \int\frac{\mathrm{d}\omega_{\vec n}(x)}{z-x}, & z\in \C\setminus\bigcup_{i=1}^p\big[a_{\vec n,i},b_{\vec n,i}\big], \medskip \\
h_{\vec n}\big(z^{(i)}\big) := \displaystyle \int\frac{\mathrm{d}\omega_{\vec n,i}(x)}{x-z}, & z\in \C\setminus\big[a_{\vec n,i},b_{\vec n,i}\big], \quad i\in\{1,\ldots,p\}.
\end{array}
\right.
\end{equation}
We further define the function $h$ on $\RS$ exactly as in \eqref{h} with $\vec\omega_{\vec n}$ replaced by $\vec \omega$. For brevity, we also denote by $\boldsymbol\gamma_{\vec n,i}$ (resp. $\ugamma_i$) the Jordan arc belonging to $\RS^{(0)}_{\vec n}$ (resp. $\RS^{(0)}$) such that $\pi(\ugamma_{\vec n,i})=\big[b_{\vec n,i},a_{\vec n,i+1}\big]$ (resp. $\pi(\ugamma_i)=\big[b_{\vec c,i},a_{\vec c,i+1}\big]$), $i\in\{1,\ldots,p-1\}$.

\begin{proposition}
\label{prop:derivative}
The function $h_{\vec n}$ is a rational function on $\RS_{\vec n}$ that has a simple zero at each $\infty^{(k)}$, $k\in\{0,\ldots,p\}$, a single simple zero, say $\z_{\vec n,i}$, on each $\ugamma_{\vec n,i}$, $i\in\{1,\ldots,p-1\}$, a simple pole\footnote{Of course, if $\z_{\vec n,i}$ coincides with either $\boldsymbol b_{\vec n,i}$ or $\boldsymbol a_{\vec n,i+1}$, then it cancels the corresponding pole.} at each $\big\{\boldsymbol a_{\vec n,i},\boldsymbol b_{\vec n,i}\big\}_{i=1}^p$, and is otherwise non-vanishing and finite. Moreover,
\[
\z_{\vec n,i} = \boldsymbol b_{\vec n,i} \quad \Leftrightarrow \quad b_{\vec n,i}\in\partial D_{\vec n,i}^- \quad \text{and} \quad \z_{\vec n,i} = \boldsymbol a_{\vec n,i+1} \quad \Leftrightarrow \quad a_{\vec n,i+1}\in\partial D_{\vec n,i+1}^-,
\]
where the sets $D_{\vec n,i}^-$ are defined as in \eqref{Dipm}. Absolutely analogous claims hold for $h$, $\RS$, and $\ugamma_i$. Furthermore, it holds that
\begin{equation}
\label{PhiInt}
\Phi_{\vec n}(\z) = \exp\left\{|\:\vec n\:|\int^\z h_{\vec n}(\x)\mathrm{d}x\right\},
\end{equation}
where the initial bound for integration should be chosen so that \eqref{normalization} is satisfied. Finally, if we set $\RS_\delta$ to be $\RS$ with circular neighborhood of radius $\delta$ excised around each of its branch points, then $h_{\vec n}\to h$ uniformly on $\RS_\delta$ for each $\delta>0$, where $h_{\vec n}$ is carried over to $\RS_\delta$ with the help of natural projections.
\end{proposition}

Thus, knowing the logarithmic derivative of $\Phi_{\vec n}$, we can recover the vector equilibrium measure $\vec\omega_{\vec n}$ by
\[
\mathrm{d}\omega_{\vec n}(x) =\left(h_{\vec n-}^{(0)}(x) - h_{\vec n+}^{(0)}(x)\right)\frac{\mathrm{d}x}{2\pi\mathrm{i}},
\]
as follows from Privalov's Lemma \cite[Sec. III.2]{Privalov} (the above formula does not allow to recover $\vec\omega_{\vec n}$ via a purely geometric construction of $\Phi_{\vec n}$ as one needs to know the intervals $\big[a_{\vec n,i},b_{\vec n,i}\big]$ to construct $\RS_{\vec n}$). We prove Propositions~\ref{prop:phi-vecn} and \ref{prop:derivative} in Section~\ref{sec:geometry}.

The purpose of the following proposition is to identify the limits in \eqref{strongHermitePade}, which are nothing but appropriate generalizations of the classical Szeg\H{o} function. In order to do that we need to specify the conditions we placed on the considered densities. In what follows, it is assumed that
\begin{equation}
\label{rho_rs}
\rho_i(x)=\rho_{\mathsf{r},i}(x)\rho_{\mathsf{s},i}(x),
\end{equation}
where $\rho_{\mathsf{r},i}$ is the regular part, that is, it is holomorphic and non-vanishing in some neighborhood of $[a_i,b_i]$, and $\rho_{\mathsf{s},i}$ is the singular part consisting of finitely many Fisher-Hartwig singularities \cite{DIKr11}, i.e.,
\begin{equation}
\label{rho_s}
\rho_{\mathsf{s},i}(x) = \prod_{j=0}^{J_i}|x-x_{ij}|^{\alpha_{ij}}\prod_{j=1}^{J_i}\left\{ \begin{array}{ll} 1, & x<x_{ij} \\ \beta_{ij}, &x>x_{ij} \end{array} \right\}
\end{equation}
where $a_i=x_{i0}<x_{i1}<\cdots<x_{iJ_i-1}<x_{iJ_i}=b_i$, $\alpha_{ij}>-1$, $\beta_{ij}\in\C\setminus(-\infty,0]$. In what follows, we adopt the following convention: given a function $F$ on $\RS$, we denote by $F^{(k)}$ its pull-back from $\RS^{(k)}\setminus\boldsymbol\Delta_k$, $k\in\{0,\ldots,p\}$. That is, $F^{(k)}(z):=F\big(z^{(k)}\big)$, $z\in\overline\C\setminus\big[a_{\vec c,i},b_{\vec c,i}\big]$.

\begin{proposition}
\label{prop:szego}
For each $i\in\{1,\ldots,p\}$, let $\rho_i$ be of the form \eqref{rho_rs}--\eqref{rho_s}. Further, let
\begin{equation}
\label{wi}
w_i(z):=\sqrt{(z-a_{\vec c,i})(z-b_{\vec c,i})}
\end{equation}
be the branch holomorphic outside of $\big[a_{\vec c,i},b_{\vec c,i}\big]$ normalized so that $w_i(z)/z\to1$ as $z\to\infty$. Then there exists the unique function $S$ non-vanishing and holomorphic in $\RS\setminus\bigcup_{i=1}^p\boldsymbol\Delta_i$ such that
\begin{equation}
\label{SrhoJumps}
S_\pm^{(i)} = S_\mp^{(0)}\big(\rho_iw_{i+}\big) \quad \text{on} \quad \big(a_{\vec c,i},b_{\vec c,i}\big)\setminus\{x_{ij}\}_{j=0}^{J_i},
\end{equation}
$i\in\{1,\ldots,p\}$, and that satisfies
\begin{equation}
\label{SrhoCorners}
\big|S^{(0)}(z)\big| \sim \big|S^{(i)}(z)\big|^{-1} \sim |z-e|^{-(2\alpha+1)/4} \quad \text{as} \quad z\to e\in\big\{a_{\vec c,i},b_{\vec c,i}\big\},
\end{equation}
$i\in\{1,\ldots,p\}$, where $\alpha=\alpha_{ij}$ if $e=x_{ij}$ and $\alpha=0$ otherwise;
\begin{multline}
\label{SrhoMiddles}
\big|S^{(0)}(z)\big| \sim \big|S^{(i)}(z)\big|^{-1} \sim |z-x_{ij}|^{-(\alpha_{ij}\pm\arg(\beta_{ij})/\pi)/2} \\ \text{as} \quad z\to x_{ij}\in\big(a_{\vec c,i},b_{\vec c,i}\big), \quad \pm\im(z)>0,
\end{multline}
$i\in\{1,\ldots,p\}$; and $\prod_{k=0}^pS^{(k)}(z)\equiv1$.
\end{proposition}

We prove Proposition~\ref{prop:szego} in Section~\ref{sec:szego}. Finally, we are ready to formulate our main result.

\begin{theorem}
\label{thm:SA}
Let $\vec f=\big(f_1,\ldots,f_p\big)$ be a vector of functions given by 
\begin{equation}
\label{f_rho}
f_i(z) = \frac1{2\pi\mathrm{i}}\int_{[a_i,b_i]}\frac{\rho_i(x)}{x-z}\:\mathrm{d}x, \qquad z\in\overline\C\setminus[a_i,b_i],
\end{equation}
for a system of mutually disjoint intervals $\big\{[a_i,b_i]\big\}_{i=1}^p$, where the functions $\rho_i$ are of the form \eqref{rho_rs}--\eqref{rho_s}, $i\in\{1,\ldots,p\}$. Given $\vec c\in(0,1)^p$ such that $|\:\vec c\:|=1$ and a sequence of multi-indices $\{\:\vec n\:\}$ satisfying \eqref{multi-indices}, let $[\:\vec n\:]_{\vec f}$ be the corresponding Hermite-Pad\'e approximant \eqref{HermitePade}--\eqref{linsys}. Then
\[
\left\{
\begin{array}{lll}
Q_{\vec n} &=& C_{\vec n}\big[1 + o(1)\big]\big(S\Phi_{\vec n}\big)^{(0)} \medskip \\
R_{\vec n}^{(i)} &=& C_{\vec n}\big[1 + o(1)\big]\big(S\Phi_{\vec n}\big)^{(i)}/w_i, \quad i\in\{1,\ldots,p\},
\end{array}
\right.
\]
locally uniformly in $\overline\C\setminus\bigcup_{i=1}^p[a_i,b_i]$, where the functions $\Phi_{\vec n}$ are as in Proposition~\ref{prop:phi-vecn}, the functions $S$ and $w_i$ are as in Proposition~\ref{prop:szego}, and $\lim_{z\to\infty}C_{\vec n}\big(S\Phi_{\vec n}\big)^{(0)}(z)z^{-|\:\vec n\:|}=1$. In particular, $\deg(Q_{\vec n})=|\:\vec n\:|$ for all $|\:\vec n\:|$ large enough.
\end{theorem}

Theorem~\ref{thm:SA} is proved in Section~\ref{sec:RHA}. It follows immediately from \eqref{linsys}, \eqref{Dipm}, and \eqref{Phi-ratio} that
\[
f_i - \frac{P_{\vec n}^{(i)}}{Q_{\vec n}} = \frac{1+o(1)}{w_i}\frac{\big(S\Phi_{\vec n}\big)^{(i)}}{\big(S\Phi_{\vec n}\big)^{(0)}}
\]
is geometrically small locally uniformly in $D_i^+$ and is geometrically big locally uniformly in $D_i^-$ whenever the latter is non-empty.

\section{Riemann-Hilbert Approach}
\label{sec:RHP}

To prove Theorem~\ref{thm:SA} we use the extension to multiple orthogonal polynomials \cite{GerKVA01} of by now classical approach of Fokas, Its, and Kitaev \cite{FIK91,FIK92} connecting orthogonal polynomials to matrix Riemann-Hilbert problems. The RH problem is then analyzed via the non-linear steepest descent method of Deift and Zhou \cite{DZ93}.

The Riemann-Hilbert approach of Fokas, Its, and Kitaev lies in the following. Assume that the multi-index $\vec n=(n_1,\ldots,n_p)$ is such that
\begin{equation}
\label{assumption}
\deg(Q_{\vec n})=|\:\vec n\:| \quad \text{and} \quad R_{\vec n-\vec e_i}^{(i)}(z)\sim z^{-n_i} \quad \text{as} \quad z\to\infty, \quad i\in\{1,\ldots,p\},
\end{equation}
where all the entries of the vector $\vec e_i$ are zero except for the $i$-th one, which is 1. Set
\begin{equation}
\label{eq:y}
{\boldsymbol Y} := \left(\begin{array}{cccc}
Q_{\vec n} & R_{\vec n}^{(1)} & \cdots & R_{\vec n}^{(p)} \smallskip\\
m_{\vec n,1}Q_{\vec n-\vec e_1} & m_{\vec n,1}R_{\vec n-\vec e_1}^{(1)} & \cdots & m_{\vec n,1}R_{\vec n-\vec e_1}^{(p)} \\
\vdots & \vdots & \ddots & \vdots \\
m_{\vec n,p}Q_{\vec n-\vec e_p} & m_{\vec n,p}R_{\vec n-\vec e_p}^{(1)} & \cdots & m_{\vec n,p}R_{\vec n-\vec e_p}^{(p)}
\end{array}\right),
\end{equation}
where $m_{\vec n,i}$, $i\in\{1,\ldots,p\}$, is a constant such that
\[
\lim_{z\to\infty}m_{\vec n,i}R_{\vec n-\vec e_i}^{(i)}(z)z^{n_i}=1.
\]

To capture the block structure of many matrices appearing below, let us introduce transformations $\mathsf{T}_i$, $i\in\{1,\ldots,p\}$,  that act on $2\times2$ matrices:
\[
\mathsf{T}_i \left(\begin{matrix} e_{11} & e_{12} \\ e_{21} & e_{22} \end{matrix}\right) := e_{11}\boldsymbol E_{1,1} + e_{12} \boldsymbol E_{1,i+1} + e_{21}\boldsymbol E_{i+1,1} + e_{22} \boldsymbol E_{i+1,i+1} + \sum_{j\neq1, i+1} \boldsymbol E_{jj},
\] 
where $\boldsymbol E_{jk}$ is the matrix with all zero entries except for the $(j,k)$-th one which is 1. It can be easily checked that  $\mathsf{T}_i(\boldsymbol A\boldsymbol B)=\mathsf{T}_i(\boldsymbol A)\mathsf{T}_i(\boldsymbol B)$ for any $2\times2$ matrices $\boldsymbol A,\boldsymbol B$. 

The matrix-valued function $\boldsymbol{Y}$ solves the following Riemann-Hilbert problem (\rhy) :
\begin{itemize}
\label{rhy}
\item[(a)] ${\boldsymbol Y}$ is analytic in $\C\setminus\bigcup_{i=1}^p [a_i,b_i]$ and $\displaystyle \lim_{z\to\infty} {\boldsymbol Y}(z)z^{-\sigma(\vec n)} = {\boldsymbol I}$, where ${\boldsymbol I}$ is the identity matrix and $\displaystyle \sigma(\:\vec n\:) := \diag\left(|\:\vec n\:|, -n_1,\ldots,-n_p\right)$;
\item[(b)] ${\boldsymbol Y}$ has continuous traces on each $(a_i,b_i)$ that satisfy ${\boldsymbol Y}_+ = {\boldsymbol Y}_-\mathsf{T}_i\left(\begin{matrix} 1 & \rho_i \\ 0 & 1\end{matrix}\right)$;
\item[(c)] the entries of the $(i+1)$-st column of $\boldsymbol Y$ behave like $\mathcal{O}\left(\psi_{\alpha_{ij}}(z-x_{ij})\right)$ as $z\to x_{ij}$, $j\in\{0,\ldots,J_i\}$, while the remaining entries stay bounded, where
\[
\psi_\alpha(z) =
\left\{
\begin{array}{ll}
|z|^\alpha, & \mbox{if} \quad \alpha<0, \smallskip \\
\log|z|, & \mbox{if} \quad \alpha=0,\smallskip \\
1, & \mbox{if} \quad \alpha>0.
\end{array}
\right.
\]
\end{itemize}

The property \hyperref[rhy]{\rhy}(a) follows immediately from \eqref{linsys} and \eqref{assumption}. The property \hyperref[rhy]{\rhy}(b) is due to the equality
\[
R_{\vec n+}^{(i)} - R_{\vec n-}^{(i)} = Q_{\vec n}\left(f_{i+} - f_{i-}\right) = Q_{\vec n}\rho_i \quad \text{on} \quad (a_i,b_i),
\]
which in itself is a consequence of \eqref{linsys}, \eqref{f_rho}, and the Sokhotski-Plemelj formulae \cite[Section~4.2]{Gakhov}. Finally, \hyperref[rhy]{\rhy}(c) follows from the local analysis of Cauchy integrals in \cite[Section~8.1]{Gakhov}.

Conversely, if $\boldsymbol Y$ is a solution of \hyperref[rhy]{\rhy}, then it follows from \hyperref[rhy]{\rhy}(b) and the normalization at infinity in \hyperref[rhy]{\rhy}(a) that $[\boldsymbol Y]_{1,1}$ is a polynomial of degree exactly $|\:\vec n\:|$. It further follows from \hyperref[rhy]{\rhy}(b) that $[\boldsymbol Y]_{1,i+1}$, $i\in\{1,\ldots,p\}$, is holomorphic outside of $[a_i,b_i]$, vanishes at infinity with order $n_i+1$, and satisfies
\[
[\boldsymbol Y]_{1,i+1+} - [\boldsymbol Y]_{1,i+1-} = [\boldsymbol Y]_{1,1}\rho_i \quad \text{on} \quad (a_i,b_i).
\] 
Combining this with \hyperref[rhy]{\rhy}(c), we see that $[\boldsymbol Y]_{1,i+1}$ is the Cauchy integral of $[\boldsymbol Y]_{1,1}\rho_i$ on $[a_i,b_i]$. Furthermore, from the order of vanishing at infinity one can easily infer that $[\boldsymbol Y]_{1,1}(x)$ is orthogonal to $x^j$, $j\in\{0,\ldots,n_i-1\}$, with respect to $\rho_i(x)\mathrm{d}x$. Hence, $[\boldsymbol Y]_{1,1}=Q_{\vec n}$, $[\boldsymbol Y]_{1,i+1}=R_{\vec n}^{(i)}$, and \eqref{assumption} holds. Other rows of $\boldsymbol Y$ can be analyzed analogously. Altogether, the following proposition takes place.
\begin{proposition}
\label{prop:rhy}
If a solution of \hyperref[rhy]{\rhy} exists then it is unique. Moreover, in this case it is given by \eqref{eq:y} where $Q_{\vec n}$ and $R_{\vec n-\vec e_i}^{(i)}$ satisfy \eqref{assumption}. Conversely, if \eqref{assumption} is fulfilled, then \eqref{eq:y} solves \hyperref[rhy]{\rhy}.
\end{proposition}

\section{Model Riemann-Hilbert Problems}
\label{sec:panleve}

As known, to analyze \hyperref[rhy]{\rhy} via steepest descent method of Deift and Zhou, one needs to construct local solutions around each singular point of the functions $\rho_i$ and the endpoints of the support of each component of the vector equilibrium measure, see Section~\ref{sec:Local}. In this section, we present all these model RH problems. In what follows we use the notation
\[
\sigma_3:=\left(\begin{matrix} 1 & 0 \\ 0 & -1 \end{matrix}\right).
\]

\subsection{Singular Points of the Weights}

In what follows, we always assume that the real line as well as its subintervals is oriented from left to right. Further, we set
\begin{equation}
\label{rays}
I_\pm:=\big\{z:~\arg(z)=\pm2\pi/3\big\}, \quad J_\pm:=\big\{z:~\arg(z)=\pm\pi/3\big\},
\end{equation}
where the rays $I_\pm$ are oriented towards the origin and the rays $J_\pm$ are oriented away from the origin. Put
\[
\Sigma(\boldsymbol\Phi_{\alpha,\beta}) := I_+\cup I_-\cup J_+\cup J_-\cup (-\infty,\infty)
\]
and consider the following Riemann-Hilbert problem: given $\alpha>-1$ and $\beta\in\C\setminus(-\infty,0]$, find a matrix-valued function $\boldsymbol\Phi_{\alpha,\beta}$ such that
\begin{itemize}
\label{rhphiAB}
\item[(a)] $\boldsymbol\Phi_{\alpha,\beta}$ is holomorphic in $\C\setminus\Sigma(\boldsymbol\Phi_{\alpha,\beta})$;
\item[(b)] $\boldsymbol\Phi_{\alpha,\beta}$ has continuous traces on $\Sigma(\boldsymbol\Phi_{\alpha,\beta})\setminus\{0\}$ that satisfy
\[
\boldsymbol\Phi_{\alpha,\beta+} = \boldsymbol\Phi_{\alpha,\beta-}
\left\{
\begin{array}{rll}
\left(\begin{matrix} 0 & 1 \\ -1 & 0 \end{matrix}\right) & \text{on} & (-\infty,0), \medskip \\
\left(\begin{matrix} 0 & \beta \\ -\beta^{-1} & 0 \end{matrix}\right) & \text{on} & (0,\infty),
\end{array}
\right.
\]
and
\[
\boldsymbol\Phi_{\alpha,\beta+} = \boldsymbol\Phi_{\alpha,\beta-}
\left\{
\begin{array}{rll}
\left(\begin{matrix} 1 & 0 \\ e^{\pm\alpha\pi\mathrm{i}} & 1 \end{matrix}\right) & \text{on} & I_\pm,  \medskip \\
\left(\begin{matrix} 1 & 0 \\ 1/\beta & 1 \end{matrix}\right) & \text{on} & J_\pm;
\end{array}
\right.
\]
\item[(c)] as $\zeta\to0$  it holds that
\[
\boldsymbol\Phi_{\alpha,\beta}(\zeta) = \mathcal{O}\left( \begin{matrix} |\zeta|^{\alpha/2} & |\zeta|^{\alpha/2} + |\zeta|^{-\alpha/2} \\ |\zeta|^{\alpha/2} & |\zeta|^{\alpha/2} + |\zeta|^{-\alpha/2} \end{matrix} \right) \quad \text{and} \quad \boldsymbol\Phi_{\alpha,\beta}(\zeta) = \mathcal{O}\left( \begin{matrix} 1 & \log|\zeta| \\ 1 & \log|\zeta| \end{matrix} \right) 
\]
when $\alpha\neq 0$ and $\alpha=0$, respectively;
\item[(d)] $\boldsymbol\Phi_{\alpha,\beta}$ has the following behavior near $\infty$:
\[
\boldsymbol\Phi_{\alpha,\beta}(\zeta) = \left(\boldsymbol I+\mathcal{O}\left(\zeta^{-1}\right)\right) \big(\mathrm{i}\zeta\big)^{\log\beta\sigma_3/2\pi\mathrm{i}} \boldsymbol B_\pm \exp\big\{\mp\mathrm{i}\zeta\sigma_3/2\big\}, \quad \pm\im(\zeta)>0,
\]
uniformly in $\C\setminus\Sigma(\boldsymbol\Phi_{\alpha,\beta})$, where $\big(\mathrm{i}\zeta\big)^{\log\beta/2\pi\mathrm{i}}$ has a branch cut along $(0,\infty)$ (observe also that $\big(\mathrm{i}\zeta\big)_-^{\log\beta/2\pi\mathrm{i}}=\beta\big(\mathrm{i}\zeta\big)_+^{\log\beta/2\pi\mathrm{i}}$ on $(0,\infty)$) and 
\[
\boldsymbol B_+ := \left(\begin{matrix} \beta^{-1/2} & 0 \smallskip \\ 0 & e^{-\alpha\pi\mathrm{i}/2} \end{matrix}\right) \beta^{\sigma_3}e^{\alpha\pi\mathrm{i}\sigma_3}, \quad \boldsymbol B_- := \boldsymbol B_+\left(\begin{matrix} 0 & -1 \smallskip \\ 1 & 0 \end{matrix}\right).
\]
\end{itemize}

The solution of \hyperref[rhphiAB]{\rhphiAB} can be written explicitly with the help of confluent hypergeometric functions. It was done first in \cite{Van03} for the case $\beta=1$, then in \cite{FMMFSou10,FMMFSou11} for $\beta\in(0,\infty)$, and, in \cite{DIKr11} for $\alpha\pm\log \beta/\pi\mathrm{i} \not\in\{-2,-4,\ldots\}$  (of course, in all the cases $\alpha>-1$; parameters $\alpha_j$ and $\beta_j$ in \cite{DIKr11} correspond to $\alpha/2$ and $\mathrm{i}\log\beta/2\pi$ above). To be more precise, one needs to take $\boldsymbol\Phi_{\alpha,\beta}\beta^{\sigma_3/4}$ multiply it by $e^{-\alpha\pi\mathrm{i}\sigma_3/2}$ in the first quadrant, by $e^{\alpha\pi\mathrm{i}\sigma_3/2}$ in the fourth quadrant, and then rotate the whole picture by $\pi/2$ to get the corresponding problem in~\cite{DIKr11}.

\subsection{Hard Edge} 

Given $\alpha>-1$, find a matrix-valued function $\boldsymbol\Psi_\alpha$ such that 
\begin{itemize}
\label{rhpsiA}
\item[(a)] $\boldsymbol\Psi_\alpha$ is holomorphic in $\C\setminus\big(I_+\cup I_-\cup(-\infty,0]\big)$;
\item[(b)] $\boldsymbol\Psi_\alpha$ has continuous traces on $I_+\cup I_-\cup(-\infty,0)$ that satisfy
\[
\boldsymbol\Psi_{\alpha+} = \boldsymbol\Psi_{\alpha-}
\left\{
\begin{array}{rll}
\left(\begin{matrix} 0 & 1 \\ -1 & 0 \end{matrix}\right) & \text{on} & (-\infty,0), \medskip \\
\left(\begin{matrix} 1 & 0 \\ e^{\pm\pi\mathrm{i}\alpha} & 1 \end{matrix}\right) & \text{on} & I_\pm;
\end{array}
\right.
\]
\item[(c)] as $\zeta\to0$  it holds that
\[
\boldsymbol\Psi_\alpha(\zeta) = \mathcal{O}\left( \begin{matrix} |\zeta|^{\alpha/2} & |\zeta|^{\alpha/2} \\ |\zeta|^{\alpha/2} & |\zeta|^{\alpha/2} \end{matrix} \right) \quad \text{and} \quad \boldsymbol\Psi_\alpha(\zeta) = \mathcal{O}\left( \begin{matrix} \log|\zeta| & \log|\zeta| \\ \log|\zeta| & \log|\zeta| \end{matrix} \right) 
\]
when $\alpha<0$ and $\alpha=0$, respectively, and
\[
\boldsymbol\Psi_\alpha(\zeta) = \mathcal{O}\left( \begin{matrix} |\zeta|^{\alpha/2} & |\zeta|^{-\alpha/2} \\ |\zeta|^{\alpha/2} & |\zeta|^{-\alpha/2} \end{matrix} \right) \quad \text{and} \quad \boldsymbol\Psi_\alpha(\zeta) = \mathcal{O}\left( \begin{matrix} |\zeta|^{-\alpha/2} & |\zeta|^{-\alpha/2} \\ |\zeta|^{-\alpha/2} & |\zeta|^{-\alpha/2} \end{matrix} \right)
\]
when $\alpha>0$, for $|\arg(\zeta)|<2\pi/3$ and $2\pi/3<|\arg(\zeta)|<\pi$, respectively;
\item[(d)] $\boldsymbol\Psi_\alpha$ has the following behavior near $\infty$:
\[
\boldsymbol\Psi_\alpha(\zeta) = \frac{\zeta^{-\sigma_3/4}}{\sqrt2}\left(\begin{matrix} 1 & \mathrm{i} \\ \mathrm{i} & 1 \end{matrix}\right)\left(\boldsymbol I+\mathcal{O}\left(\zeta^{-1/2}\right)\right)\exp\left\{2\zeta^{1/2}\sigma_3\right\}
\]
uniformly in $\C\setminus\big(I_+\cup I_-\cup(-\infty,0]\big)$.
\end{itemize}
The solution of this Riemann-Hilbert problem was constructed explicitly in \cite{KMcLVAV04} with the help of modified Bessel and Hankel functions.

\subsection{Soft-Type Edge}

To describe the model Riemann-Hilbert problem we need, it will be convenient to denote by $\Omega_1$, $\Omega_2$, $\Omega_3$, and $\Omega_4$ consecutive sectors of $\C\setminus\big((-\infty,\infty)\cup I_-\cup I_+\big)$ starting with the one containing the first quadrant and continuing counter clockwise.  Given $\alpha\in\R$ and $\re(\beta)\geq0$, we are looking for a matrix-valued function $\boldsymbol\Psi_{\alpha,\beta}$ such that
\begin{itemize}
\label{rhpsiAB}
\item[(a)] $\boldsymbol\Psi_{\alpha,\beta}$ is holomorphic in $\C\setminus\big(I_+\cup I_-\cup(-\infty,\infty)\big)$;
\item[(b)] $\boldsymbol\Psi_{\alpha,\beta}$ has continuous traces on $I_+\cup I_-\cup(-\infty,0)\cup(0,\infty)$ that satisfy
\[
\boldsymbol\Psi_{\alpha,\beta+} = \boldsymbol\Psi_{\alpha,\beta-}
\left\{
\begin{array}{rll}
\left(\begin{matrix} 0 & 1 \\ -1 & 0 \end{matrix}\right) & \text{on} & (-\infty,0), \medskip \\
\left(\begin{matrix} 1 & 0 \\ e^{\pm\mathrm{i}\pi\alpha} & 1 \end{matrix}\right) & \text{on} & I_\pm, \medskip \\
\left(\begin{matrix} 1 & \beta \\ 0 & 1 \end{matrix}\right) & \text{on} & (0,\infty);
\end{array}
\right.
\]
\item[(c)] as $\zeta\to0$  it holds that
\[
\boldsymbol\Psi_{\alpha,\beta}(\zeta) = \boldsymbol E(\zeta) \boldsymbol S_{\alpha,\beta}(\zeta)\boldsymbol A_j, \quad \zeta\in\Omega_j,
\]
where $\boldsymbol E$ is a holomorphic matrix function, 
\[
\boldsymbol A_3 =  \boldsymbol A_4 \left(\begin{matrix} 1 & 0 \\ e^{-\alpha\pi\mathrm{i}} & 1 \end{matrix}\right), \quad \boldsymbol A_4 = \boldsymbol A_1\left(\begin{matrix} 1 & -\beta \\ 0 & 1 \end{matrix}\right), \quad \boldsymbol A_1 = \boldsymbol A_2\left(\begin{matrix} 1 & 0 \\ e^{\alpha\pi\mathrm{i}} & 1 \end{matrix}\right),
\]
and
\[
\boldsymbol A_2 = \left(\begin{matrix} \frac1{2\cos(\alpha\pi/2)}\frac{ 1-\beta e^{\alpha\pi\mathrm{i}} }{ 1-e^{\alpha\pi\mathrm{i}} } & \frac1{2\cos(\alpha\pi/2)}\frac{ \beta - e^{\alpha\pi\mathrm{i}} }{ 1-e^{\alpha\pi\mathrm{i}} } \medskip \\ -e^{\alpha\pi\mathrm{i}/2} & e^{-\alpha\pi\mathrm{i}/2} \end{matrix}\right) \quad \text{while} \quad \boldsymbol S_{\alpha,\beta}(\zeta) = \zeta^{\alpha\sigma_3/2}
\]
when $\alpha$ is not an integer, 
\[
\boldsymbol A_2 = \left(\begin{matrix} \frac12 e^{\alpha\pi\mathrm{i}/2} & \frac12 e^{-\alpha\pi\mathrm{i}/2} \medskip \\ -e^{\alpha\pi\mathrm{i}/2} & e^{-\alpha\pi\mathrm{i}/2} \end{matrix}\right) \quad \text{while} \quad \boldsymbol S_{\alpha,\beta}(\zeta) = \left(\begin{matrix} \zeta^{\alpha/2} & \frac{1-\beta}{2\pi\mathrm{i}}\zeta^{\alpha/2}\log\zeta \medskip \\ 0 & \zeta^{-\alpha/2} \end{matrix}\right)
\]
when $\alpha$ is an even integer,
\[
\boldsymbol A_2 = \left(\begin{matrix} 0 & e^{-\alpha\pi\mathrm{i}/2} \medskip \\ -e^{\alpha\pi\mathrm{i}/2} & e^{-\alpha\pi\mathrm{i}/2} \end{matrix}\right) \quad \text{while} \quad \boldsymbol S_{\alpha,\beta}(\zeta) = \left(\begin{matrix} \zeta^{\alpha/2} & \frac{1+\beta}{2\pi\mathrm{i}}\zeta^{\alpha/2}\log\zeta \medskip \\ 0 & \zeta^{-\alpha/2} \end{matrix}\right)
\]
when $\alpha$ is an odd integer;
\item[(d)] $\boldsymbol\Psi_{\alpha,\beta}$ has the following behavior near $\infty$:
\[
\boldsymbol\Psi_{\alpha,\beta}(\zeta;s) = \left(\boldsymbol I+\mathcal{O}\left(\zeta^{-1}\right)\right) \frac{\zeta^{-\sigma_3/4}}{\sqrt2} \left(\begin{matrix} 1 & \mathrm{i} \\ \mathrm{i} & 1 \end{matrix}\right) \exp\left\{-\frac23(\zeta+s)^{3/2}\sigma_3\right\}
\]
uniformly in $\C\setminus\big(I_+\cup I_-\cup(-\infty,\infty)\big)$.
\end{itemize}

Besides \hyperref[rhpsiAB]{\rhpsiAB}, we shall also need \rhwpsiAB~ obtained from \hyperref[rhpsiAB]{\rhpsiAB} by replacing \hyperref[rhpsiAB]{\rhpsiAB}(d) with
\begin{itemize}
\label{rhwpsiAB}
\item[($\tilde{\textnormal{d}}$)] $\widetilde{\boldsymbol\Psi}_{\alpha,\beta}$ has the following behavior near $\infty$:
\[
\widetilde{\boldsymbol\Psi}_{\alpha,\beta}(\zeta;s) = \left(\boldsymbol I+\mathcal{O}\left(\zeta^{-1}\right)\right) \frac{\zeta^{-\sigma_3/4}}{\sqrt2} \left(\begin{matrix} 1 & \mathrm{i} \\ \mathrm{i} & 1 \end{matrix}\right) \exp\left\{-\left(\frac23\zeta^{3/2}+s\zeta^{1/2}\right)\sigma_3\right\}.
\]
\end{itemize}
The problems \hyperref[rhpsiAB]{\rhpsiAB} and \hyperref[rhwpsiAB]{\rhwpsiAB} are simultaneously uniquely solvable and the solutions are connected by
\begin{equation}
\label{transfer}
\widetilde{\boldsymbol\Psi}_{\alpha,\beta}(\zeta;s) = \left(\begin{matrix} 1 & 0 \\ \mathrm{i}s^2/4 & 1 \end{matrix}\right)\boldsymbol\Psi_{\alpha,\beta}(\zeta;s)
\end{equation}
as follows from the estimate
\[
\frac23(\zeta+s)^{3/2} - \left(\frac23\zeta^{3/2}+s\zeta^{1/2}\right) = \left(1 + \mathcal{O}\big(s/\zeta\big)\right) \frac{s^2}{4\zeta^{1/2}} \quad \text{as} \quad \zeta\to\infty.
\]

When $\alpha=0$, $\beta=1$, and $s=0$, the above Riemann-Hilbert problem is well known \cite{DKMLVZ99b} and is solved using Airy functions. When $\beta=1$, the solvability of this problem for all $s\in\R$ was shown in \cite{IKOs08} with further properties investigated in \cite{IKOs09} (\hyperref[rhwpsiAB]{\rhwpsiAB} is associated with a solution of Painlev\'e XXXIV equation). The solvability of the case $\alpha=0$, $\beta\in\C\setminus(-\infty,0)$, and $s\in\R$ was obtained in \cite{XuZh11}.  The latter case appeared in \cite{uBogatClI} as well. More generally, the following theorem holds.

\begin{theorem}
\label{thm:localRH}
Given $\alpha\in\R$ and $\beta\in\C\setminus(-\infty,0)$, the RH-problems \hyperref[rhpsiAB]{\rhpsiAB}, and therefore \hyperref[rhwpsiAB]{\rhwpsiAB}, is uniquely solvable for all $s\in\R$. Moreover, assuming $\beta\neq0$, it holds that
\begin{equation}
\label{betterestimate1}
\boldsymbol\Psi_{\alpha,\beta}(\zeta;s) = \frac{\zeta^{-\sigma_3/4}}{\sqrt2} \left(\begin{matrix} 1 & \mathrm{i} \\ \mathrm{i} & 1 \end{matrix}\right) \left( \boldsymbol I + \mathcal{O}\left(\sqrt{\frac{|s|+1}{|\zeta|+1}}\right) \right)  \exp\left\{-\frac23(\zeta+s)^{3/2}\sigma_3\right\}
\end{equation}
uniformly for $\zeta\in\C\setminus\big(I_+\cup I_-\cup(-\infty,\infty)\big)$ and $s\in(-\infty,\infty)$, and it also holds uniformly for $s\in[0,\infty)$ when $\beta=0$; furthermore, we have that
\begin{equation}
\label{betterestimate2}
\widetilde{\boldsymbol\Psi}_{\alpha,0}(\zeta;s) = \frac{\zeta^{-\sigma_3/4}}{\sqrt2} \left(\begin{matrix} 1 & \mathrm{i} \\ \mathrm{i} & 1 \end{matrix}\right) \left( \boldsymbol I + \mathcal{O}\left(\sqrt{\frac{|s|+1}{|\zeta|+1}}\right) \right)  \exp\left\{-\left(\frac23\zeta^{3/2}+s\zeta^{1/2}\right)\sigma_3\right\}
\end{equation}
uniformly for $\zeta\in\C\setminus\big(I_+\cup I_-\cup(-\infty,0]\big)$ and $s\in(-\infty,0]$.
\end{theorem}

Theorem~\ref{thm:localRH} is proved in Section~\ref{sec:rhpsiAB}.

\section{Geometry}
\label{sec:geometry}

In this section we prove Propositions~\ref{prop:phi-vecn} and \ref{prop:derivative}.

\subsection{Proof of Proposition~\ref{prop:phi-vecn}}

Set
\[
O_i^\pm := \big\{z:~\re(z)\in\big(a_{\vec n,i},b_{\vec n,i}\big) \; \text{and} \; \pm\im(z)>0\big\}.
\]
Since the measures $\omega_{\vec n,i}$ are supported on the real line, \eqref{EqPropi} and Schwarz reflection principle yield that the function
\[
\left\{
\begin{array}{ll}
\ell_{\vec n,i} - V^{\omega_{\vec n}+\omega_{\vec n,i}}(z), & z\in O_i^+, \medskip \\
V^{\omega_{\vec n}+\omega_{\vec n,i}}(z) - \ell_{\vec n,i}, & z\in O_i^-,
\end{array}
\right.
\]
is harmonic across $(a_{\vec n,i},b_{\vec n,i})$. As the support of $\omega_{\vec n}-\omega_{\vec n,i}$ is disjoint from $\big[a_{\vec n,i},b_{\vec n,i}\big]$, the function $\ell_{\vec n,i}+V^{\omega_{\vec n}-\omega_{\vec n,i}}$ is harmonic across $\big(a_{\vec n,i},b_{\vec n,i}\big)$ as well. By taking the difference of these two functions, we see that
\[
\left\{
\begin{array}{ll}
- 2V^{\omega_{\vec n}}(z), & z\in O_i^+, \medskip \\
2V^{\omega_{\vec n,i}}(z) - 2\ell_{\vec n,i}, & z\in O_i^-,
\end{array}
\right.
\]
is harmonic in the same vertical strip. Thus, the function
\begin{equation}
\label{Hvecn}
H_{\vec n}(\z) := \left\{
\begin{array}{ll}
-V^{\omega_{\vec n}}(z) + \frac1{p+1}\sum_{k=1}^p\ell_{\vec n,k}, & \z\in\RS_{\vec n}^{(0)}, \medskip \\
V^{\omega_{\vec n,i}}(z) - \ell_{\vec n,i} + \frac1{p+1}\sum_{k=1}^p\ell_{\vec n,k}, & \z\in\RS_{\vec n}^{(i)}, \quad i\in\{1,\ldots,p\},
\end{array}
\right.
\end{equation}
is harmonic on $\RS_{\vec n}\setminus\bigcup_{k=0}^p\big\{\infty^{(k)}\big\}$. Since $V^\nu(z)=-|\nu|\log|z|+\mathcal{O}(1)$ as $z\to\infty$, we get that the difference
\[
|\:\vec n\:|^{-1}\log\left|\Phi_{\vec n}(\z)\right| - H_{\vec n}(\z)
\]
is harmonic on the whole surface $\RS_{\vec n}$ and therefore is a constant. Since $\sum_{k=0}^pH_{\vec n}\big(z^{(k)}\big) \equiv 0$
and $\Phi_{\vec n}$ is normalized so that \eqref{normalization} holds, the first claim of the proposition follows.

Let $\vec\nu$ be a weak$^*$ limit point of $\big\{\vec\omega_{\vec n}\big\}$. Since $\{\:\vec n\:\}$ satisfies \eqref{multi-indices}, it holds that $\vec\nu\in M_{\vec c}\big(\{a_i,b_i\}_{i=1}^p\big)$. Thus, if we show that $I[\:\vec\omega\:] \geq I[\:\vec\nu\:]$, then $\vec\nu=\vec\omega$ by \eqref{vecI}. To this end, let $\alpha_{\vec n,i}$ be positive constants such that $|\alpha_{\vec n,i}\omega_i|=n_i/|\:\vec n\:|$, $i\in\{1,\ldots,p\}$. By \eqref{multi-indices}, $\alpha_{\vec n,i}\to1$ as $|\:\vec n\:|\to\infty$. Set $\vec\nu_{\vec n}:=(\alpha_{\vec n,1}\omega_1,\ldots,\alpha_{\vec n,p}\omega_p)$. Then it follows from \eqref{vecI} applied for the vector $\big(n_1/|\:\vec n\:|,\ldots,n_p/|\:\vec n\:|\big)$ that
\[
I\big[\:\vec\omega\:\big] = \lim_{|\:\vec n\:|\to\infty}I\big[\:\vec\nu_{\vec n}\:\big] \geq \liminf_{|\:\vec n\:|\to\infty}I\big[\:\vec\omega_{\vec n}\:\big].
\]
Furthermore, the very definition of the weak$^*$ convergence implies that
\[
\lim_{|\:\vec n\:|\to\infty} I\big[\omega_{\vec n,j},\omega_{\vec n,k}\big] = I\big[\nu_j,\nu_k\big]
\]
for $j\neq k$ as $\supp\big(\omega_{\vec n,j}\big)\cap\supp\big(\omega_{\vec n,k}\big)=\varnothing$ in this case. It also follows from the Principle of Descent \cite[Thm.~I.6.8]{SaffTotik} that
\[
\liminf_{|\:\vec n\:|\to\infty}I\big[\omega_{\vec n,i}\big] \geq I[\nu_i].
\]
Altogether,
\[
I\big[\:\vec\omega\:\big] \geq \liminf_{|\:\vec n\:|\to\infty}I\big[\:\vec\omega_{\vec n}\:\big] \geq I\big[\:\vec\nu\:\big],
\]
which proves the claim about weak$^*$ convergence of measures. 

Weak$^*$ convergence of measures implies convergence of minima of the corresponding potentials \cite{GRakh81}. Hence,  \eqref{min-max-sys} yields that $\ell_{\vec n,i}\to\ell_i$ for all $i\in\{1,\ldots,p\}$. Moreover, weak$^*$ convergence also implies locally uniform convergence of $V^{\omega_{\vec n,i}}$ to $V^{\omega_i}$ in $\C\setminus\big[a_{\vec c,i},b_{\vec c,i}]$ (there is no convergence at infinity as, in general, $|\omega_{\vec n,i}|\neq|\omega_i|$ for given $\vec n$). Thus, it remains to show that the convergence of the potentials is uniform on compact subsets of~$\C$.

First let $K$ be a continuum such that $a_{\vec c,i},b_{\vec c,i}\notin K$ and either $\im(z)\geq0$ for all $z\in K$ or $\im(z)\leq 0$ for all $z\in K$ (it can intersect $\big(a_{\vec c,i},b_{\vec c,i}\big)$). Then there exists a unique continuum $K^{(i)}$ such that $\pi(K^{(i)})=K$ and $K^{(i)}\cap \RS^{(i)}\not=\varnothing$. Further, let $U$ be a neighborhood of $K$ such that $a_{\vec c,i},b_{\vec c,i}\notin U$. Denote $U^{(i)}$ the neighborhood of $K^{(i)}$ such that $\pi(U^{(i)})=U$. Since $a_{\vec n,i}\to a_{\vec c,i}$ and $b_{\vec n,i}\to b_{\vec c,i}$ as $|\:\vec n\:|\to\infty$, we can analogously define $K^{(i)}_{\vec n}$ and $U^{(i)}_{\vec n}$ on $\RS_{\vec n}$. By definition,
\[
\left\{
\begin{array}{l}
V^{\omega_{\vec n,i}}_{|K} = H_{\vec n|K^{(i)}_{\vec n}} + \ell_{\vec n,i} - \frac1{p+1}\sum_{j=1}^p\ell_{\vec n,j} \medskip \\
V^{\omega_i}_{|K} = H_{|K^{(i)}} + \ell_i - \frac1{p+1}\sum_{j=1}^p\ell_j,
\end{array}
\right.
\]
where $H$ is defined on $\RS$ exactly as $H_{\vec n}$ was defined on $\RS_{\vec n}$. Hence, to show that $V^{\omega_{\vec n,i}}$ converge to $V^{\omega_i}$ uniformly on $K$ it is enough to show that the pull backs of $H_{\vec n}$ from $U^{(i)}_{\vec n}$ to $U$ converge locally uniformly to the pull back of $H$. We do know that such a convergence takes place locally uniformly on $U\cap\{\im(z)>0\}$ and $U\cap\{\im(z)<0\}$. The full claim will follow from Harnack's theorem if we show that the pull backs of $H_{\vec n}$, which are harmonic in $U$, form a uniformly bounded family there. The latter is true since each $H_{\vec n}^{(k)}$ converges to $H^{(k)}$ on any Jordan curve $J$ that encloses $\bigcup_{i=1}^p[a_i,b_i]$. Hence, the moduli $|H_{\vec n}|$ are bounded on the lift of $J$ to $\RS_{\vec n}$ and the bound is independent of $\vec n$. The maximum principle propagates this estimate through the region of $\RS_{\vec n}$ containing $U^{(i)}_{\vec n}$ and bounded by the lift of $J$.

Assume now that $K$ is a continuum that contains one of the points $\big\{a_{\vec c,i},b_{\vec c,i}\big\}$, say $b_{\vec c,i}$ for definiteness. It is sufficient to assume that $K$ is contained in a disk, say $U$, centered at the $b_{\vec c,i}$ of radius small enough so that no other point from $\bigcup_{j=1}^p\big\{a_{\vec c,j},b_{\vec c,j}\big\}$ belongs to $U$. We can define $K^{(i)}$ and $K^{(i)}_{\vec n}$ analogously to the previous case. Let $U^{(i)}$ and $U^{(i)}_{\vec n}$ be the circular neighborhoods of $\boldsymbol b_{\vec c,i}$ and $\boldsymbol b_{\vec n,i}$, respectively, with the natural projection $U$ (clearly, they cover $U$ twice). Let $V$ be a disk centered at the origin of radius smaller than the one of $U$, but large enough so that the translation of $V$ to $b_{\vec c,i}$ still contains $K$. Then the functions $\phi_{\vec n}(z)=\big(z+b_{\vec n,i}\big)^2$ and $\phi(z)=\big(z+b_{\vec c,i}\big)^2$ provide one-to-one correspondents between $V$ and some subdomains of $U^{(i)}_{\vec n}$ and $U^{(i)}$, respectively. These subdomains still contain $K_{\vec n}^{(i)}$ and $K^{(i)}$. Since $b_{\vec n,i}\to b_{\vec c,i}$ as $|\:\vec n\:|\to\infty$, we can establish exactly as above that $H_{\vec n}\circ\phi_{\vec n}$ converges to $H\circ \phi$ locally uniformly in $V$, which again yields that $V^{\omega_{\vec n,i}}$ converges to $V^{\omega_i}$ uniformly on $K$. Clearly, the considered cases are sufficient to establish the uniform convergence on compact subsets of $\C$.

\subsection{Proof of Proposition~\ref{prop:derivative}}

Observe that
\begin{eqnarray}
h_{\vec n}^{(0)}(z) &=& \int\frac{\mathrm{d}\omega_{\vec n}(x)}{z-x} = -2\partial_z V^{\omega_{\vec n}}(z)  = 2|\:\vec n\:|^{-1}\partial_z \log\left|\Phi_{\vec n}^{(0)}(z)\right| \nonumber \\
&=& |\:\vec n\:|^{-1}\big(\Phi_{\vec n}^{(0)}(z)\big)^\prime/\Phi_{\vec n}^{(0)}(z) \nonumber
\end{eqnarray}
by Proposition~\ref{prop:phi-vecn} and direct computation, where $2\partial_z:=\partial_x -\mathrm{i}\partial_y$. Clearly, analogous formulae hold for $h_{\vec n}^{(i)}$. That is, $h_{\vec n}$ is the logarithmic derivative of $\Phi_{\vec n}$, in particular, \eqref{PhiInt} holds. Therefore, $h_{\vec n}$ is holomorphic around each point of $\RS_{\vec n}\setminus\big\{\boldsymbol a_{\vec n,i},\boldsymbol b_{\vec n,i}\big\}_{i=1}^p$ and clearly has a simple zero at each $\infty^{(k)}$, $k\in\{0,\ldots,p\}$. Since $\RS_{\vec n}$ has square root branching at each ramification point, $\Phi_{\vec n}^{(0)}$ has Puiseux expansion in non-negative powers of $1/2$ at each of them. Hence, $h_{\vec n}^{(0)}$ has such an expansion as well and the smallest exponent is $-1/2$. Thus, $h_{\vec n}$ has at most a simple pole at each $\{\boldsymbol a_{\vec n,i},\boldsymbol b_{\vec n,i}\big\}_{i=1}^p$ and, in particular, is a rational function on~$\RS_{\vec n}$.

The number of zeros and poles, including multiplicities, of a rational function should be the same. Therefore, $h_{\vec n}$ has at most $2p$  and at least $p+1$ poles (the lower bound comes from the number of zeros at ``infinities'') and at most $p-1$ ``finite'' zeros. Let us now show that each of $p-1$ arcs $\ugamma_{\vec n,i}$ contains exactly one of those ``finite'' zeros (we slightly abuse the notion of a zero here since a simple zero at the endpoint means cancelation of the corresponding pole). Clearly, this is equivalent to showing that $h_{\vec n}^{(0)}$ has a single simple zero in each gap $\big[b_{\vec n,i},a_{\vec n,i+1}\big]$ (again, a ``zero'' at the endpoint means that $h_{\vec n}^{(0)}$ is locally bounded there). 

 Assume to the contrary that there is at least one gap, say $\big[b_{\vec n,j},a_{\vec n,j+1}\big]$, without a zero. Then $h_{\vec n}^{(0)}$ would be infinite at both endpoints $b_{\vec n,j},a_{\vec n,j+1}$. However, since $\omega_{\vec n}$ is a positive measure, the very definition \eqref{h} yields that $h_{\vec n}^{(0)}$ is decreasing  on $\big(b_{\vec n,j},a_{\vec n,j+1}\big)$. The latter is possible only if
\begin{equation}
\label{infinities}
\lim_{x\to b_{\vec n,j}}h_{\vec n}^{(0)}(x) = -\lim_{x\to a_{\vec n,j+1}}h_{\vec n}^{(0)}(x) = \infty.
\end{equation}
As $h_{\vec n}^{(0)}$ is continuous on $\big(b_{\vec n,j},a_{\vec n,j+1}\big)$ it must vanish there. Since there are exactly $p-1$ gaps and $p-1$ ``free'' zeros, this contradiction proves the claim. 

Let us now show the correspondence between occurrence of the zeros at the endpoints of the gaps and the fact that divergence domains are touching the support. To this end, notice that \eqref{PhiInt} combined with \eqref{Phi-ratio} yields that
\begin{equation}
\label{MapInt}
\ell_{\vec n,i} - V^{\omega_{\vec n,i}+\omega_{\vec n}}(x) = \int_{b_{\vec n,i}}^x\left(h_{\vec n}^{(0)}-h_{\vec n}^{(i)}\right)(y)\mathrm{d}y.
\end{equation}
If the zero of $h_{\vec n}^{(0)}$ on $\big[b_{\vec n,i},a_{\vec n,i+1}\big]$ does not coincide with $b_{\vec n,i}$, then
\[\left\{
\begin{array}{lll}
h_{\vec n}^{(0)}(y) &=& c_{\vec n}\big(y-b_{\vec n,i}\big)^{-1/2} + \mathcal{O}(1) \medskip \\
h_{\vec n}^{(i)}(y) &=& -c_{\vec n}\big(y-b_{\vec n,i}\big)^{-1/2} + \mathcal{O}(1)
\end{array}
\right.
\]
for $y-b_{\vec n,i}>0$ and small enough, where $c_{\vec n}>0$, see \eqref{infinities}. Hence,
\begin{equation}
\label{MapInt1}
\ell_{\vec n,i} - V^{\omega_{\vec n,i}+\omega_{\vec n}}(x)  = 4c_{\vec n}\big(x-b_{\vec n,i}\big)^{1/2} + \mathcal{O}\big(\big|x-b_{\vec n,i}\big|^{3/2}\big)>0
\end{equation}
for $x-b_{\vec n,i}>0$ and small enough. On the other hand, if the zero coincides with $b_{\vec n,i}$, then
\[
\left\{
\begin{array}{lll}
h_{\vec n}^{(0)}(y) &=& \tilde c_{\vec n} - c^\prime_{\vec n} \big(y-b_{\vec n,i}\big)^{1/2} + \mathcal{O}\big(\big|y-b_{\vec n,i}\big|\big) \medskip \\
h_{\vec n}^{(i)}(y) &=& \tilde c_{\vec n} + c^\prime_{\vec n} \big(y-b_{\vec n,i}\big)^{1/2} + \mathcal{O}\big(\big|y-b_{\vec n,i}\big|\big)
\end{array}
\right.
\]
for $y-b_{\vec n,i}>0$ and small enough, where $c_{\vec n}^\prime>0$ (recall that $h_{\vec n}^{(0)}$ is a decreasing function in each gap). Therefore,
\begin{equation}
\label{MapInt2}
\ell_{\vec n,i} - V^{\omega_{\vec n,i}+\omega_{\vec n}}(x) = -(4c^\prime_{\vec n}/3)\big(x-b_{\vec n,i}\big)^{3/2} + \mathcal{O}\big(\big|x-b_{\vec n,i}\big|^{5/2}\big)<0
\end{equation}
for $x-b_{\vec n,i}>0$ and small enough. Thus, if the zero from $\big[b_{\vec n,i},a_{\vec n,i+1}\big]$ coincides with $b_{\vec n,i}$, then $b_{\vec n,i}\in\partial D_{\vec n,i}^-$ and if it does not, then $b_{\vec n,i}\notin\partial D_{\vec n,i}^-$, see \eqref{Dipm}. As the analysis near $a_{\vec n,i}$ can be completed similarly, this finishes the proof of the claim.

Let now $H_{\vec n}$ be defined by \eqref{Hvecn} and $H$ be defined analogously on $\RS$. We have shown during the course of the proof of Proposition~\ref{prop:phi-vecn} that $H_{\vec n}\to H$ uniformly on $\RS_\delta$, where $H_{\vec n}$ is carried over to $\RS_\delta$ with the help of natural projections. Since $h_{\vec n}=2\partial_z H_{\vec n}$ and $h=2\partial_z H$, we get that $h_{\vec n}\to h$ uniformly on $\RS_\delta$. This implies that $h$ is a rational function on $\RS$. The claim about zero/pole distribution of $h$ follows from the analogous statement for $h_{\vec n}$ and analysis similar to \eqref{MapInt}--\eqref{MapInt2}.

\section{Szeg\H{o} Function}
\label{sec:szego}

In this section we prove Proposition~\ref{prop:szego}. Let $\z,\w\in\RS$. Denote by $\mathrm{d}\Omega_{\z,\w}$ the unique abelian differential of the third kind which is holomorphic on $\RS\setminus\{\z,\w\}$ and has simple poles at $\z$ and $\w$ of respective residues $+1$ and $-1$. Define
\begin{equation}
\label{dC}
\mathrm{d}C_\z := p\mathrm{d}\Omega_{\z,\w}-\sum_{i=1}^p\mathrm{d}\Omega_{\z_i,\w},
\end{equation}
where $\pi^{-1}(z)=\{\z,\z_1,\ldots,\z_p\}$ for each $z$ which is not a projection of a branch point of $\RS$. The differential $\mathrm{d}C_\z$ does not depend on the choice of $\w$ as it is simply the normalized third kind differential with $p+1$ simple poles at $\z,\z_1,\ldots,\z_p$ having respective residues $p,-1,\ldots,-1$.

For each $\x\in\boldsymbol\Delta_i$, which is not a branch point of $\RS$, we shall denote by $\x^*$ a point on $\boldsymbol\Delta_i$ having the same canonical projection, i.e., $\pi(\x)=\pi(\x^*)$. When $\x\in\boldsymbol\Delta_i$ is a branch point of the surface, we simply set $\x^*=\x$. Let $\lambda$ be a H\"older continuous function on $\boldsymbol\Delta:=\bigcup_{i=1}^p\boldsymbol\Delta_i$. Define
\begin{equation}
\label{Lambda}
\Lambda(\z) := \frac1{2(p+1)\pi\mathrm{i}}\oint_{\boldsymbol\Delta}\lambda\mathrm{d}C_\z, \qquad \z\in\RS\setminus\pi^{-1}\big(\pi(\boldsymbol\Delta)\big).
\end{equation}
The function $\Lambda$ is holomorphic in the domain of its definition. Further, if $\z\to\x\in\boldsymbol\Delta^{\pm}$, then  $\z_j\to\x^*\in\boldsymbol\Delta^{\mp}$ for some $j\in\{1,\ldots,p\}$ and 
\[
\Lambda_+(\x) - \Lambda_-(\x) = \frac{p\lambda(\x) + \lambda(\x^*)}{p+1},
\]
according to \cite[Eq. (2.8)]{Zver71}. On the other hand, if $\z\to\tilde\x\not\in\boldsymbol\Delta$, while $\z_j\to\x\in\boldsymbol\Delta^\pm$ and $\z_k\to\x^*\in\boldsymbol\Delta^\mp$ for some $j,k\in\{1,\ldots,p\}$, then
\[
\Lambda_+\big(\tilde\x\big) - \Lambda_-\big(\tilde\x\big) = \frac{\lambda(\x^*) - \lambda(\x)}{p+1}.
\]
Thus, if we additionally require that $\lambda(\x)=\lambda(\x^*)$, then $\Lambda$ is a holomorphic function in $\RS\setminus\boldsymbol\Delta$ such that
\begin{equation}
\label{Lambda-traces}
\Lambda_+(\x) - \Lambda_-(\x) = \lambda(\x), \quad \x\in\boldsymbol\Delta.
\end{equation}
It also can be readily verified using \eqref{dC} and \eqref{Lambda} that
\begin{equation}
\label{Lambda-norm}
\Lambda(\z) + \sum_{i=1}^p\Lambda(\z_i)  \equiv 0 \quad \text{on} \quad \RS.
\end{equation}

The above construction works for discontinuous function as well. Moreover, it is known that the continuity of $\Lambda_\pm$, in fact, H\"older continuity, depends on H\"older continuity of $\lambda$ only locally. That is, if $\lambda$ is H\"older continuous on some open subarc of $\boldsymbol\Delta$, so are the traces $\Lambda_\pm$ on this subarc irrespectively of the smoothness of $\lambda$ on the remaining part of $\boldsymbol\Delta$. To capture the behavior of $\Lambda$ around the points where $\lambda$ is not continuous, we define a local approximation to the Cauchy differential $\mathrm{d}C_\z$. To this end, fix $i\in\{1,\ldots,p\}$ and denote by $\mathfrak{\boldsymbol U}$ a connected annular neighborhood of $\boldsymbol\Delta_i$ disjoint from other $\boldsymbol\Delta_j$ such that every point in $\pi(\mathfrak{\boldsymbol U})$ has exactly two preimages (except for the branch points, of course). Write $\mathfrak{\boldsymbol U}^+\cup\mathfrak{\boldsymbol U}^-=\mathfrak{\boldsymbol U}\setminus\boldsymbol\Delta$, where $\mathfrak{\boldsymbol U}^+\cap\mathfrak{\boldsymbol U}^-=\varnothing$, $\mathfrak{\boldsymbol U}^\pm$ are connected and partially bounded by $\boldsymbol\Delta_i^\pm$. Set $\widetilde w_i(\z):=\pm w_i(z)$, $\z\in\mathfrak{\boldsymbol U}^\pm$, where $w_i$ is given by \eqref{wi}. Then $\widetilde w_i$ is holomorphic in $\mathfrak{\boldsymbol U}$. Further, put
\[
\mathrm{d}\widetilde\Omega_\z(\x) := \frac12 \frac{\widetilde w_i(\x)+\widetilde w_i(\z)}{x-z}\frac{\mathrm{d}x}{\widetilde w_i(\x)},
\]
which is a holomorphic differential on $\mathfrak{\boldsymbol U}\setminus\{\z\}$ that has a simple pole at $\z$ with residue 1. Then the difference  $\mathrm{d}C_\z - p\mathrm{d}\widetilde\Omega_\z + \mathrm{d}\widetilde\Omega_{\z_i}$ is a holomorphic differential in $\mathfrak{\boldsymbol U}$ and therefore the function $\Lambda-\widetilde\Lambda$ is holomorphic $\mathfrak{\boldsymbol U}$, where
\[
\widetilde\Lambda(\z) := \frac1{2(p+1)\pi\mathrm{i}}\oint_{\boldsymbol \Delta_i}\lambda\mathrm{d}\big(p\widetilde\Omega_\z - \widetilde\Omega_{\z^*}\big)
\]
and $\z^*\neq \z$ is a point in $\mathfrak{\boldsymbol U}$ such that $\pi(\z)=\pi(\z^*)$. Thus, to understand the local behavior of $\Lambda$ is sufficient to study $\widetilde\Lambda$. Since $\widetilde w_i(\z^*)=-\widetilde w_i(\z)$ for $\z\in\mathfrak{\boldsymbol U}$, and $w_{i-}(x)=-w_{i+}(x)$ for $x\in\big(a_{\vec c,i},b_{\vec c,i}\big)$, it holds for $\lambda(\x)=\lambda(x)$ that
\begin{equation}
\label{LambdaApprox}
\widetilde\Lambda(\z) = \frac{\widetilde w_i(\z)}{2\pi\mathrm{i}}\int_{\Delta_i}\frac{\lambda(x)}{w_{i+}(x)}\frac{\mathrm{d}x}{x-z}, \quad \z\in\mathfrak{\boldsymbol U}\setminus\boldsymbol\Delta.
\end{equation}
The first type of singularities we are interested in is of the form
\begin{equation}
\label{final-lambda1}
\lambda(\x) = \alpha\log\big|x-x_0\big|,  \quad \x\in\boldsymbol\Delta_i,
\end{equation}
where $x_0\in\big[a_{\vec c,i},b_{\vec c,i}\big]$. Carefully tracing the implications of \cite[Sec. I.8.5--6]{Gakhov} to the integrals of the form \eqref{LambdaApprox} and \eqref{final-lambda1}, we get that
\begin{equation}
\label{L1}
\widetilde\Lambda(\z) = \pm\frac\alpha2 \log(z-x_0) + \mathcal{O}(1), \quad \mathfrak{\boldsymbol U}^\pm\ni\z\to \x_0.
\end{equation}
The second type of the singular behavior we want to consider is given by
\begin{equation}
\label{final-lambda2}
\lambda(\x) = (\log\beta)\chi_{x_0}(x),  \quad \x\in\boldsymbol\Delta_i,
\end{equation}
where $x_0\in\big(a_{\vec c,i},b_{\vec c,i}\big)$ and $\chi_{x_0}$ is the characteristic function of $\big[x_0,b_{\vec c,i}\big]$. It follows from the analysis in \cite[Sec. I.8.6]{Gakhov} that
\begin{equation}
\label{L2}
\left\{
\begin{array}{lll}
\widetilde\Lambda\big(z^{(0)}\big) & = &  \mp\frac{\log\beta}{2\pi\mathrm{i}} \log(z-x_0) + \mathcal{O}(1), \medskip \\
\widetilde\Lambda\big(z^{(i)}\big) & = &  \pm\frac{\log\beta}{2\pi\mathrm{i}} \log(z-x_0) + \mathcal{O}(1),
\end{array}
\right. \quad z\to x_0, \quad \pm\im(z)>0.
\end{equation}

Now, let the functions $\rho_i$ be of the form \eqref{rho_rs}--\eqref{rho_s}. Set
\[
\lambda_\rho(\x) := -\log\big(\rho_i(x)w_{i+}(x)\big), \quad \x\in\boldsymbol \Delta_i.
\]
By using the identity $w_{i+}(x)=\mathrm{i}|w_i(x)|$ and the explicit expressions \eqref{rho_s}, we can then write
\begin{multline*}
\lambda_\rho(\x) = - \log\big(\mathrm{i}\rho_{\mathrm{r},i}(x)\big) - \sum_{i=0}^{J_i}\left(\alpha_{ij}\log|x-x_{ij}| + \log\beta_{ij}\chi_{x_{ij}}(x)\right) \\
- (1/2)\log\big| x - a_{\vec c,i}\big| - (1/2)\log\big| x - b_{\vec c,i}\big|.
\end{multline*}
Clearly, the singular behavior of $\lambda_\rho$ is precisely of the form \eqref{final-lambda1} and \eqref{final-lambda2}. Define $\Lambda_\rho$ as in \eqref{Lambda} and set $S:=\exp\{\Lambda_\rho\}$. Then \eqref{SrhoJumps} is a consequence of \eqref{Lambda-traces} since
\[
\left(S^{(i)}_\pm/S^{(0)}_\mp\right)(x) = \exp\left\{\big(\Lambda_{\rho-}-\Lambda_{\rho+}\big)(\x)\right\}.
\]
Moreover, \eqref{SrhoCorners} and \eqref{SrhoMiddles} clearly follow from  \eqref{L1} and \eqref{L2}. Finally, the last claim of the proposition follows from \eqref{Lambda-norm}.

\section{Auxiliary Results}
\label{sec:aux}

Below we prove auxiliary estimates \eqref{local-un2} and \eqref{local-un1} that will be needed in Section~\ref{ssec:proof} to finish the proof of Theorem~\ref{thm:SA}. They are presented here in a separate section as the arguments used to prove them are disconnected from the techniques of the steepest descent method employed in Section~\ref{sec:RHA}. 

Let $\x,\w\in\RS$ be such that $\x$ is not a branch point of $\RS$. There exists a unique, up to multiplicative normalization, rational function on $\RS$, say $\Psi$, with a simple pole at $\x$, a simple zero at $\w$, and otherwise non-vanishing and finite. For uniqueness, we normalize $\Psi(z)=z+\{\text{holomorphic part}\}$ around $\x$ if $\x$ is a point above infinity, and $\Psi(\z)=(z-x)^{-1}+\{\text{holomorphic part}\}$ around $\x$ otherwise. 

Let $\x_{\vec n},\w_{\vec n}\in\RS_{\vec n}$ be such that they have the same canonical projections and belong to the sheets with the same labels as $\x,\w$, respectively, when the latter are not branch points of $\RS$ (points on $\bigcup_{i=1}^p\boldsymbol\Delta_i$ need to be identified with the sequences of points convergent to them to set up the correspondence). If $\w$ is a branch point, we set $\w_{\vec n}$ to be the branch point of $\RS_{\vec n}$ whose projection converges to or coincides with the one of $\w$. We define $\Psi_{\vec n}$ to be similarly normalized rational function on $\RS_{\vec n}$ with a pole at  $\x_{\vec n}$ and a zero at $\w_{\vec n}$.

As the statement of Proposition~\ref{prop:derivative}, let $\RS_\delta$ be the subsets of $\RS$ obtained by removing circular neighborhoods of radius $\delta$ around each branch point. We assume that $\delta$ is small enough so that $\x\in\RS_\delta$ and $\w\in\RS_\delta$ when $\w$ is not a branch point. Using natural projections we can redefine $\Psi_{\vec n}$ as a function on $\RS_\delta$. Naturally, it will have a pole at $\x$ and a zero at $\w$ if the latter belong to $\RS_\delta$. Then, regarding $\Psi_{\vec n}$ as a function on $\RS_\delta$, we have that
\begin{equation}
\label{estimate}
\Psi_{\vec n} = \big[1+o(1)\big]\Psi
\end{equation}
uniformly on $\RS_\delta$ as $|\:\vec n\:|\to\infty$. Indeed, assume first that $\w\in\RS_\delta$. Let $\mathfrak{U}_\x\subset\RS_\delta$ be a circular neighborhood of $\x$ such that $\w\notin\mathfrak{U}_\x$. Observe that $\Psi$ is a univalent function on $\RS$. Thus, by applying Koebe's 1/4 theorem to $1/\Psi$, we see that $|\Psi| < C$ on $\partial \mathfrak{U}_\x$ for some constant $C>0$ that depends only on the radius of $\mathfrak{U}_\x$. Moreover, the maximum modulus principle implies that $|\Psi|<C$ on $\RS\setminus\mathfrak{U}_\x$. Clearly, absolutely analogous considerations apply to $\Psi_{\vec n}$ on $\RS_{\vec n}$ and the constant $C$ remains the same. Hence, the ratio $\Psi_{\vec n}/\Psi$ is a holomorphic function on $\RS_\delta$ such that $|\Psi_{\vec n}/\Psi|<C/\tilde C$ by the maximum modulus principle, where $0<\tilde C\leq \min_{\RS\setminus\RS_\delta}|\Psi|$ and this constant can be chosen independently of $\delta$. Picking a discrete sequence $\delta_n\to0$ and using the diagonal argument as well as the normal family argument, we see that any subsequence of $\big\{\Psi_{\vec n}/\Psi\big\}$ contains a subsequence convergent to a function holomorphic on $\RS\setminus\bigcup_{i=1}^p\big\{\boldsymbol a_{\vec c,i},\boldsymbol b_{\vec c,i}\big\}$. Moreover, this function is necessarily bounded around the branch points and therefore holomorphically extends to the entire Riemann surface $\RS$. Thus, this function must be a constant and the normalization at $\x$ yields that this constant is 1. This finishes the proof of \eqref{estimate} in the case $\w\in\RS_\delta$. When $\w$ is a branch point, the first half of the above considerations yields that $\{\Psi-\Psi_{\vec n}\}$ is a family of holomorphic function on $\RS_\delta$ with uniformly and independently of $\delta$ bounded moduli. Therefore, the same argument yields that $\Psi_{\vec n}=\Psi+o(1)$ uniformly on $\RS_\delta$. As $\Psi$ is non-vanishing in $\RS_\delta$, this estimate implies \eqref{estimate}.

Let $\Upsilon_{\vec n,i}$ (resp. $\Upsilon_i$), $i\in\{1,\ldots,p\}$, be rational functions on $\RS_{\vec n}$ (resp. $\RS$) with a simple pole at $\infty^{(i)}$, a simple zero at $\infty^{(0)}$, otherwise non-vanishing and finite, and normalized so $\Upsilon_{\vec n,i}^{(i)}(z)/z\to 1$ as $z\to\infty$. Then \eqref{estimate} immediately yields 
\begin{equation}
\label{local-un2}
\Upsilon_{\vec n,i} = \big[1+o(1)\big]\Upsilon_i
\end{equation}
uniformly on each $\RS_\delta$ as $|\:\vec n\:|\to\infty$.

Further, let $\mathrm{d}\Omega_{\z,\w}^{\vec n}$ be the unique abelian differential of the third kind which is holomorphic on $\RS_{\vec n}\setminus\{\z,\w\}$, has simple poles at $\z$ and $\w$ with respective residues $+1$ and $-1$. It is known that such a differential can be written as $\mathrm{d}\Omega_{\z,\w}^{\vec n}(\x)=\Psi_{\z,\w}^{\vec n}(\x)\mathrm{d}x$, where $\Psi_{\z,\w}^{\vec n}$ is the unique rational function on $\RS_{\vec n}$ with double zero at each $\infty^{(k)}$, $k\in\{0,\ldots,p\}$, a simple pole at each $\bigcup_{i=1}^p\big\{\boldsymbol a_{\vec n,i},\boldsymbol b_{\vec n,i}\big\}$, simple poles at $\z$ and $\w$, otherwise non-vanishing and finite, and normalized to have residue $1$ at $\z$. Writing $1/\Psi_{\z,\w}^{\vec n}$ as a product of terms with one zero and one pole and applying \eqref{estimate} to these factors, we see that
\[
\Psi_{\z,\w}^{\vec n} = \big[1+o(1)\big]\Psi_{\z,\w}
\]
uniformly on each $\RS_\delta$ as $|\:\vec n\:|\to\infty$, where $\mathrm{d}\Omega_{\z,\w}(\x)=\Psi_{\z,\w}(\x)\mathrm{d}x$ is the corresponding differential on $\RS$. Then, defining $\Lambda_{\vec n}$ via analogs of \eqref{dC} and \eqref{Lambda} for $\RS_{\vec n}$, we get that $\Lambda_{\vec n}(\z)=\Lambda(\z)+o(1)$ uniformly in $\RS\setminus\mathfrak{N}$ for each neighborhood $\mathfrak{N}$ of $\bigcup_{i=1}^p\boldsymbol\Delta$. Therefore, if we define $S_{\vec n}$ on $\RS_{\vec n}$ exactly as $S$ was defined on $\RS$ and consider $S_{\vec n}$ as function on $\RS\setminus\mathfrak{N}$, then
\begin{equation}
\label{local-un1}
S_{\vec n} = \big[1+o(1)\big]S
\end{equation}
uniformly there. Moreover, $S_{\vec n}$ obeys all the conclusions of Proposition~\ref{prop:szego} with respect to $\RS_{\vec n}$.

\section{Non-linear Steepest Descent Analysis}
\label{sec:RHA}

In this section we prove Theorem~\ref{thm:SA} with some technical details relegated to Section~\ref{sec:Local}.

\subsection{Opening of the Lenses}
\label{subsec:OL}

Since we shall use these sets quite often, put
\begin{equation}
\label{Es}
\left\{
\begin{array}{lll}
E_{\vec c} &:=& \bigcup_{i=1}^p\big\{a_{\vec c,i},b_{\vec c,i}\big\}, \medskip \\
E_\mathsf{in} & := & \bigcup_{i=1}^p \left(\{x_{ij}\}\cap\big(a_{\vec c,i},b_{\vec c,i}\big)\right) \medskip \\
E_\mathsf{out} & := & \bigcup_{i=1}^p \left\{x_{ij}:~x_{ij}\not\in\big[a_{\vec c,i},b_{\vec c,i}\big]~\text{and}~\alpha_{ij}\leq 0\right\}.
 \end{array}
 \right.
\end{equation}
That is, $E_\mathsf{in}$ consists of the singular points $x_{ij}$ that belong to the support of $\vec\omega$, and $E_\mathsf{out}$ consists of those singular points outside of the support for which the densities $\rho_i$ are unbounded.

To proceed with the factorization of the jump matrices in \hyperref[rhy]{\rhy}(b), we need to construct the so-called ``lens'' around $\bigcup_{i=1}^p[a_i,b_i]$. To this end, given $e\in E_\mathsf{out} \cup E_\mathsf{in} \cup E_{\vec c}$, let $U_e$ be a disk centered at $e$. We assume that the radii of these disks are small enough so that $\overline U_{e_1}\cap \overline U_{e_2}=\varnothing$ for $e_1\neq e_2$. We also assume that $\overline U_e\subset D_i^-$ when $e\in E_\mathsf{out}$. 

Now, let $e_0,e_1$ be the $j$-th pair of two consecutive points from $\big(E_\mathsf{in}\cup E_{\vec c}\big)\cap\big[a_{\vec c,i},b_{\vec c,i}\big]$. We choose arcs $\Gamma_{ij}^\pm$ incident with $e_0$ and $e_1$ and lying in the upper ($+$) and lower $(-)$ half-planes in the following way: if $e_k\in E_{\vec c}$, then it should hold that
\begin{equation}
\label{select1}
\zeta_{e_k}\big(\Gamma_{ij}^\pm\cap U_{e_k}\big) \subset I_\pm,
\end{equation}
where the rays $I_\pm$ are defined in \eqref{rays} and $\zeta_{e_k}$ is a certain conformal function in $U_{e_k}$ constructed further below in \eqref{zeta-vec-1} or \eqref{zeta-vec-2} (depending on the considered case); if $e_k\in E_\mathsf{in}$, it should hold that
\begin{equation}
\label{select2}
\zeta_{e_k}\big(\Gamma_{ij+k-1}^\pm\cap U_{e_k}\big) \subset I_\pm \quad \text{and} \quad \zeta_{e_k}\big(\Gamma_{ij+k}^\pm\cap U_{e_k}\big) \subset J_\pm,
\end{equation}
where $\zeta_{e_k}$ is a conformal function in $U_{e_k}$ constructed further below in \eqref{zeta-vec-5} and the rays $J_\pm$ are also defined in \eqref{rays}. Outside $U_{e_0}\cup U_{e_1}$ we choose $\Gamma_{ij}^\pm$ to be  segments joining the corresponding points on $\partial U_{e_0}$ and $\partial U_{e_1}$, see Figure~\ref{fig:2}. We further set $\Gamma_i^\pm:=\bigcup_j\Gamma_{ij}^\pm$.

Since the geometry of the problem might depend on each particular index $\vec n$ (and not only on $\vec c$), we construct in a similar fashion arcs $\Gamma_{\vec n,ij}^\pm$ and $\Gamma_{\vec n,i}^\pm$, where this time the maps $\zeta_{e_k}$ are replaced by $\zeta_{\vec n,e_k}$, see \eqref{zeta-vecn-5}, \eqref{zeta-vecn-1}, \eqref{zeta-vecn-2}, or \eqref{zeta-vecn-4}. As we show later in \eqref{zeta-asymp}, the arcs $\Gamma_{\vec n,i}^\pm$ converge to $\Gamma_i^\pm$ in Hausdorff metric. Finally, we denote by $\Omega_{\vec n,ij}^\pm$ the domains delimited by $\Gamma_{\vec n,ij}^\pm$ and $\big[a_{\vec n,i},b_{\vec n,i}\big]$, and set $\Omega_{\vec n,i}^\pm:=\bigcup_j\Omega_{\vec n,ij}^\pm$.
\begin{figure}[!ht]
\centering
\includegraphics[scale=.5]{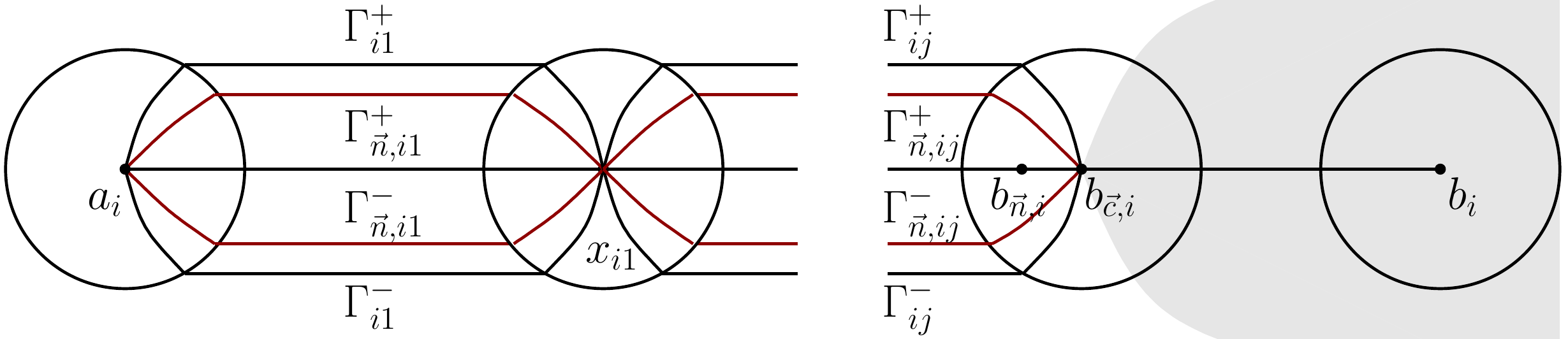}
\caption{\small The arcs $\Gamma_{ij}^\pm$ and $\Gamma_{\vec n,ij}^\pm$ in the case where there is at least one point in $E_\mathsf{in}$, $b_{\vec n,i}<b_{\vec c,i}<b_i$, and $b_i\in E_\mathsf{out}$.}
\label{fig:2}
\end{figure} 

Fix $\Gamma_{\vec n,il}^\pm$ with endpoints $e_1<e_2$. There exists an index $k$ such that $x_{ij}\leq e_1$ for $j< k$ and $x_{ij}\geq e_2$ for $j\geq k$. Then it follows from \eqref{rho_rs} and \eqref{rho_s} that the function $\rho_i$ holomorphically extends to $\Omega_{\vec n,il}^\pm$ by
\begin{equation}
\label{rhoi-ul}
\rho_i(z) := \rho_{\mathsf{r},i}(z)\prod_{j<k}\beta_{ij}\prod_{j<k}(z-x_{ij})^{\alpha_{ij}}\prod_{j\geq k}(x_{ij}-z)^{\alpha_{ij}},
\end{equation}
where $(z-x_{ij})^{\alpha_{ij}}$ is holomorphic off $(-\infty,x_{ij}]$ and $(x_{ij}-z)^{\alpha_{ij}}$ is holomorphic off $[x_{ij},\infty)$. Using these extensions, set
\begin{equation}
\label{eq:x}
\boldsymbol X:= \boldsymbol Y \left\{
\begin{array}{ll}
\mathsf{T}_i\left(\begin{matrix} 1 & 0 \\ \mp1/\rho_i & 1 \end{matrix}\right) & \mbox{in} \quad \Omega_{\vec n,i}^{\pm}, \medskip \\
\boldsymbol I & \text{otherwise},
\end{array}
\right.
\end{equation}
where $\boldsymbol Y$ is a matrix-function that solves \hyperref[rhy]{\rhy} (if it exists). It can be readily verified that ${\boldsymbol X}$ solves the following Riemann-Hilbert problem (\rhx):
\begin{itemize}
\label{rhx}
\item[(a)] ${\boldsymbol X}$ is analytic in $\C\setminus\bigcup_{i=1}^p\left([a_i,b_i]\cup\Gamma_{\vec n,i}^+\cup\Gamma_{\vec n,i}^-\right)$ and $\displaystyle \lim_{z\to\infty} {\boldsymbol X}(z)z^{-\sigma(\vec n)} = {\boldsymbol I}$;
\item[(b)] ${\boldsymbol X}$ has continuous traces on $\bigcup_{i=1}^p\left((a_i,b_i)\cup\Gamma_{\vec n,i}^+\cup\Gamma_{\vec n,i}^-\right)$ that satisfy
\[
{\boldsymbol X}_+={\boldsymbol X}_- \left\{
\begin{array}{rl}
\displaystyle \mathsf{T}_i\left(\begin{matrix} 0 & \rho_i \\ -1/\rho_i & 0 \end{matrix}\right) & \text{on} \quad \big[a_{\vec c,i},b_{\vec c,i}\big],  \medskip \\
\displaystyle \mathsf{T}_i\left(\begin{matrix} 1 & \rho_i \\ 0 & 1 \end{matrix}\right) & \text{on} \quad (a_i,b_i)\setminus\big[a_{\vec c,i},b_{\vec c,i}\big],  \medskip \\
\displaystyle \mathsf{T}_i\left(\begin{matrix} 1 & 0 \\ 1/\rho_i & 1 \end{matrix}\right) & \text{on} \quad \Gamma_{\vec n,i}^+\cup\Gamma_{\vec n,i}^-;
\end{array}
\right.
\]
\item[(c)] $\boldsymbol X$ has the following behavior near $e\in E_{\vec c}\cup E_\mathsf{in} \cup E_\mathsf{out}$: 
\begin{itemize}
\item if $e\in E_\mathsf{out}$, then $\boldsymbol X$ satisfies \hyperref[rhy]{\rhy}(c) with $\boldsymbol Y$ replaced by $\boldsymbol X$;
\item if $e\in E_{\vec c}\setminus\{x_{ij}\}$, then all the entries of $\boldsymbol X$ are bounded at $e$;
\item if $e\in E_\mathsf{in}$ or  $e\in E_{\vec c}\cap\{x_{ij}\}$, then $\boldsymbol X$ satisfies \hyperref[rhy]{\rhy}(c) with $\boldsymbol Y$ replaced by $\boldsymbol X$ outside of $\overline{\Omega_{\vec n,i}^+\cup\Omega_{\vec n,i}^-}$ while inside it behaves exactly as in \hyperref[rhy]{\rhy}(c) when $\alpha_{ij}<0$, the entries of the first and $(i+1)$-st column behave like $\mathcal{O}(\psi_0(z-x_{ij}))$ and the rest of the entries are bounded when $\alpha_{ij}=0$, and the entries of the first column behave like $\mathcal{O}(\psi_{-\alpha_{ij}}(z-x_{ij}))$ and the rest of the entries are bounded when $\alpha_{ij}>0$.
\end{itemize}
\end{itemize}

Due to the block structure of the jumps in \hyperref[rhy]{\rhy}(b), \cite[Lemma~17]{BY10} can be carried over word for word to the present case to prove:

\begin{lemma}
\label{lem:rhx}
\hyperref[rhx]{\rhx} is solvable if and only if \hyperref[rhy]{\rhy} is solvable. When solutions of \hyperref[rhx]{\rhx} and \hyperref[rhy]{\rhy} exist, they are unique and connected by \eqref{eq:x}.
\end{lemma}

\subsection{Auxiliary Parametrices}
\label{sec:bulk}

To solve \hyperref[rhx]{\rhx}, we construct parametrices that asymptotically describe the behavior of $\boldsymbol X$ away from and around each point in $E_\mathsf{in} \cup E_\mathsf{out} \cup E_{\vec c}$. To this end, we construct a matrix-valued function $\boldsymbol{N}$ that solves the following Riemann-Hilbert problem (\rhn):
\begin{itemize}
\label{rhn}
\item[(a)] ${\boldsymbol N}$ is analytic in $\C\setminus\bigcup_{i=1}^p\big[a_{\vec n,i},b_{\vec n,i}\big]$ and $\displaystyle \lim_{z\to\infty} {\boldsymbol N}(z)z^{-\sigma(\vec n)} = {\boldsymbol I}$;
\item[(b)] ${\boldsymbol N}$ has continuous traces on $\big(a_{\vec n,i},b_{\vec n,i}\big)$ that satisfy $\boldsymbol N_+ = \boldsymbol N_-\displaystyle \mathsf{T}_i\left(\begin{matrix} 0 & \rho_i \\ -1/\rho_i & 0 \end{matrix}\right)$.
\end{itemize}

Let $\Phi_{\vec n}$ be the functions from Proposition~\ref{prop:phi-vecn} while $S_{\vec n}$ and $\Upsilon_{\vec n,i}$, $i\in\{1,\ldots,p\}, $ be the functions introduced in Section~\ref{sec:aux}.  Set
\begin{equation}
\label{eq:n}
\boldsymbol N := \boldsymbol {CMD},
\end{equation}
where $\boldsymbol D:=\diag\left(\Phi_{\vec n}^{(0)},\ldots,\Phi_{\vec n}^{(p)}\right)$, $\boldsymbol C:=\diag\left(C_{\vec n,0},\ldots,C_{\vec n,p}\right)$ with the constant $C_{\vec n,k}$ defined by
\begin{equation}
\label{Cnk}
\left\{
\begin{array}{lcl}
\displaystyle \lim_{z\to\infty}C_{\vec n,0}\big(S_{\vec n}\Phi_{\vec n}\big)^{(0)}(z)z^{-|\:\vec n\:|}&=& 1 \medskip \\
\displaystyle \lim_{z\to\infty}C_{\vec n,i}\big(S_{\vec n}\Phi_{\vec n}\big)^{(i)}(z)z^{n_i}&=& 1, \quad i\in\{1,\ldots,p\},
\end{array}
\right.
\end{equation}
and the matrix $\boldsymbol M$ is given by
\begin{equation}
\label{M}
\boldsymbol M:=\left(\begin{array}{cccc}
S_{\vec n}^{(0)} & S_{\vec n}^{(1)}/w_{\vec n,1} & \cdots & S_{\vec n}^{(p)}/w_{\vec n,p} \medskip\\
\big( S_{\vec n}\Upsilon_{\vec n,1}\big)^{(0)} & \big( S_{\vec n}\Upsilon_{\vec n,1}\big)^{(1)}/w_{\vec n,1} & \cdots &\big( S_{\vec n}\Upsilon_{\vec n,1}\big)^{(p)}/w_{\vec n,p} \smallskip \\
\vdots & \vdots & \ddots & \vdots \\
\big( S_{\vec n}\Upsilon_{\vec n,p}\big)^{(0)} & \big( S_{\vec n}\Upsilon_{\vec n,p}\big)^{(1)}/w_{\vec n,1} & \cdots & \big( S_{\vec n}\Upsilon_{\vec n,p}\big)^{(p)}/w_{\vec n,p}
\end{array}\right).
\end{equation}
Then \eqref{eq:n} solves \hyperref[rhn]{\rhn}. Indeed, \hyperref[rhn]{\rhn}(a) follows immediately from the analyticity properties of $S_{\vec n}$, $\Upsilon_{\vec n,i}$, and $\Phi_{\vec n}$ as well as from \eqref{Cnk}. Observe that the multiplication by
\[
\mathsf{T}_i\left(\begin{matrix} 0 & \rho_i \\ -1/\rho_i & 0 \end{matrix}\right)
\]
on the right replaces the first column by the $(i+1)$-st one multiplied by $\rho_i$, while $(i+1)$-st column is replaced by the first one multiplied by $-1/\rho_i$. Hence, \hyperref[rhn]{\rhn}(b) follows from the analog of \eqref{SrhoJumps} for $S_{\vec n}$ and the fact that any rational function $\Psi$ on $\RS_{\vec n}$ satisfies $\Psi_\pm^{(0)}=\Psi_\mp^{(i)}$ on $\big(a_{\vec n,i},b_{\vec n,i}\big)$. 

Since the jump matrices in \hyperref[rhn]{\rhn}(b) have determinant 1, $\det(\boldsymbol{N})$ is a holomorphic function in $\overline\C\setminus\bigcup_i\big\{a_{\vec n,i},b_{\vec n,i}\big\}$ and $\det(\boldsymbol{N})(\infty)=1$. Moreover, it follows from the analogs of \eqref{SrhoCorners} and \eqref{SrhoMiddles} for $S_{\vec n}$ that each entry of the first column of $\boldsymbol N$ behaves like
\[
\mathcal{O}\left(|z-e|^{-(2\alpha+1)/4}\right) \quad \text{and} \quad \mathcal{O}\left(|z-x_{ij}|^{-(\alpha_{ij}\mp\arg(\beta_{ij})/\pi)/2}\right)
\]
for $e\in\big\{a_{\vec n,i},b_{\vec n,i}\big\}$ ($\alpha=\alpha_{ij}$ if $e=x_{ij}$ and $\alpha=0$ otherwise) and for $x_{ij} \in \big(a_{\vec n,i},b_{\vec n,i}\big)$ ($\pm\im(z)>0$), respectively, the entries of the $(i+1)$-st column behave like
\[
\mathcal{O}\left(|z-e|^{(2\alpha-1)/4}\right) \quad \text{and} \quad \mathcal{O}\left(|z-x_{ij}|^{(\alpha_{ij}\mp\arg(\beta_{ij})/\pi)/2}\right)
\]
there, and the rest of the entries are bounded. Thus, the determinant has at most square root singularities at these points and therefore is a bounded entire function. That is, $\det(\boldsymbol{N})\equiv1$ as follows from the normalization at infinity.

Further, for each $e\in E_\mathsf{in} \cup E_\mathsf{out} \cup E_{\vec c}$, we want to solve \hyperref[rhx]{\rhx} locally in $U_e$. That is, we are seeking a solution of the following \rhp:
\begin{itemize}
\label{rhp}
\item[(a,b,c)] $\boldsymbol P_e$ satisfies  \hyperref[rhx]{\rhx}(a,b,c) within $U_e$;
\item[(d)] $\boldsymbol P_e=\boldsymbol M\left(\boldsymbol I+\boldsymbol{\mathcal{O}}(\varepsilon_{e,\vec n})\right)\boldsymbol D$ uniformly on $\partial U_e\setminus\left([a_i,b_i]\cup\bigcup_{i=1}^p\Gamma_{\vec n,i}^+\cup\Gamma_{\vec n,i}^-\right)$, where $0<\varepsilon_{e,\vec n}\to0$ as $|\:\vec n\:|\to\infty$.
\end{itemize}

Since the construction of $\boldsymbol P_e$ solving \rhp~ is rather lengthy, it is carried out separately in Section~\ref{sec:Local} further below. 
 
\subsection{Final R-H Problem}

Denote by $\Omega_{\vec n,ij}$ the domain delimited by $\Gamma_{\vec n,ij}^+$ and $\Gamma_{\vec n,ij}^-$ (in particular, $\Omega_{\vec n,ij}^\pm\subset\Omega_{\vec n,ij}$). Set $\Omega_{\vec n}:=\bigcup_{ij}\Omega_{\vec n,ij}$ and $U:=\bigcup_{e\in E_\mathsf{in} \cup E_\mathsf{out} \cup E_{\vec c}} U_e$.
Define
\[
\Sigma_{\vec n} := \partial U \cup \left[\bigcup_{i=1}^p\left(\Gamma_{\vec n,i}^+\cup\Gamma_{\vec n,i}^-\right) \setminus U\right] \cup \left[\bigcup_{i=1}^p[a_i,b_i] \setminus \big(U\cup \Omega_{\vec n}\big)\right] .
\]
Moreover, we define $\Sigma$ by replacing $\Gamma_{\vec n,i}^\pm$ with $\Gamma_i^\pm$ in the definition of $\Sigma_{\vec n}$ see Figure~\ref{fig:3}. 
\begin{figure}[!ht]
\centering
\includegraphics[scale=.5]{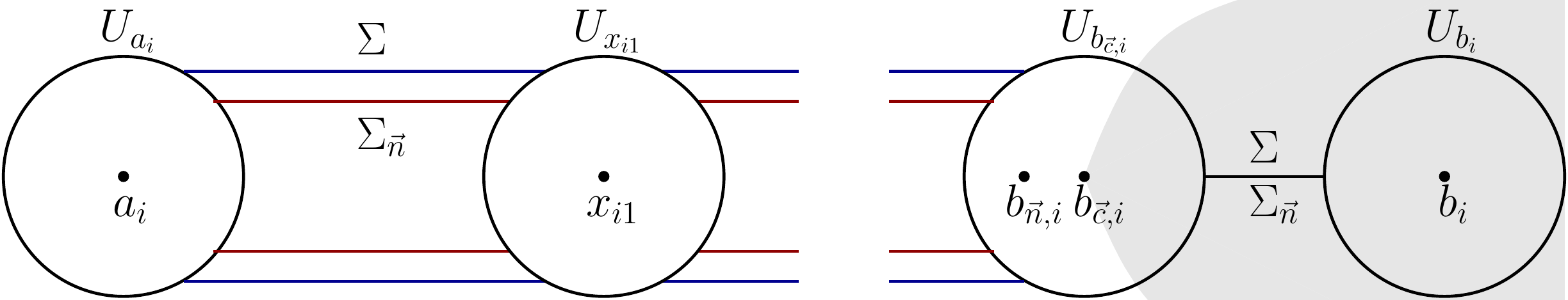}
\caption{\small Contours $\Sigma$ (black and blue lines) and $\Sigma_{\vec n}$ (black and red lines).}
\label{fig:3}
\end{figure} 
Given matrices $\boldsymbol N$ and $\boldsymbol P_e$, $e\in E_\mathsf{in} \cup E_\mathsf{out} \cup E_{\vec c}$, from the previous section, consider the following Riemann-Hilbert Problem (\rhr):
\begin{itemize}
\label{rhr}
\item[(a)] $\boldsymbol{Z}$ is a holomorphic matrix function in $\overline\C\setminus\Sigma_{\vec n}$ and $\boldsymbol{Z}(\infty)=\boldsymbol{I}$;
\item[(b)] $\boldsymbol{Z}$ has continuous traces on $\Sigma_{\vec n}$ that satisfy
\[
\boldsymbol{Z}_+  = \boldsymbol{Z}_- \left\{
\begin{array}{ll}
\boldsymbol{MD} \mathsf{T}_i\left(\begin{matrix} 1 & 0 \\ 1/\rho_i & 1 \end{matrix}\right) (\boldsymbol{MD})^{-1} & \text{on} \quad\left(\Gamma_{\vec n,i}^+\cup\Gamma_{\vec n,i}^-\right)\setminus \overline U, \medskip \\
\boldsymbol{MD} \mathsf{T}_i\left(\begin{matrix} 1 & \rho_i \\ 0 & 1 \end{matrix}\right) (\boldsymbol{MD})^{-1} & \text{on} \quad[a_i,b_i]\setminus\big(\overline{U\cup\Omega_{\vec n}}\big) \medskip \\
\boldsymbol P_e (\boldsymbol{MD})^{-1} & \text{on} \quad \partial U_e.
\end{array}
\right.
\]
\end{itemize}
Then the following lemma takes place.
\begin{lemma}
\label{lem:rhr}
The solution of \hyperref[rhr]{\rhr} exists for all $|\:\vec n\:|$ large enough and satisfies
\begin{equation}
\label{eq:r}
\boldsymbol{Z}=\boldsymbol{I} + \mathcal{O}\big( \varepsilon_{\vec n}\big)
\end{equation}
uniformly in $\overline\C$, where $\varepsilon_{\vec n}=\min_e\varepsilon_{e,\vec n}$.
\end{lemma}
\begin{proof}
Analyticity of $\rho_i$ yields that $\boldsymbol Z$ can be analytically continued to be holomorphic  outside of $\Sigma$. To do that one simply needs to multiply $\boldsymbol Z$ by the first jump matrix in \hyperref[rhr]{\rhr}(b) or its inverse (the jump matrices have determinate 1 and therefore are invertible). We shall show that the jump matrices are locally uniformly geometrically small in $D_i^+$. This would imply that the new problem is solvable if and only if the initial problem is solvable and the bound \eqref{eq:r} remains valid regardless the contour. Hence, in what follows we shall consider \rhr~ on $\Sigma$ rather than on $\Sigma_{\vec n}$.

It was shown in Section~\ref{sec:bulk} that $\det(\boldsymbol N)\equiv1$. Moreover, it follows from \eqref{normalization} that $\det(\boldsymbol D)\equiv1$  while the equality $\prod_{k=0}^pS^{(k)}_{\vec n}\equiv1$ and \eqref{Cnk} imply that $\det(\boldsymbol C)\equiv1$. Hence, $\det(\boldsymbol M)\equiv1$ and it follows from \hyperref[rhp]{\rhp}(d), \eqref{local-un1}, and \eqref{local-un2} that
\[
\boldsymbol P_e (\boldsymbol{MD})^{-1} = \boldsymbol I + \boldsymbol M\mathcal{O}\big( \varepsilon_{e,\vec n}\big)\boldsymbol M^{-1} = \boldsymbol I+\mathcal{O}(\varepsilon_{e,\vec n})
\]
holds uniformly on each $\partial U_e$. On the other hand, it holds on $\Gamma_i^\pm\setminus \overline U$ that
\[
\boldsymbol {MD} \mathsf{T_i} \left(\begin{matrix} 1 & 0 \\ 1/\rho_i & 1 \end{matrix}\right) (\boldsymbol{MD})^{-1} = \boldsymbol I + \frac1{\rho_i}\frac{\Phi_{\vec n}^{(i)}}{\Phi_{\vec n}^{(0)}}\boldsymbol M\boldsymbol E_{i+1,1}\boldsymbol M^{-1} = \boldsymbol I+\mathcal{O}\left(C_i^{-|\:\vec n\:|}\right)
\]
for some constant $C_i>1$ by \eqref{Dipm}, \eqref{Phi-ratio}, and Proposition~\ref{prop:phi-vecn}. Analogously, we get that
\[
\boldsymbol {MD} \mathsf{T_i} \left(\begin{matrix} 1 & \rho_i \\ 0 & 1 \end{matrix}\right) (\boldsymbol{MD})^{-1} = \boldsymbol I + \rho_i\frac{\Phi_{\vec n}^{(0)}}{\Phi_{\vec n}^{(i)}}\boldsymbol M\boldsymbol E_{1,i+1}\boldsymbol M^{-1} = \boldsymbol I+\mathcal{O}\left(\widetilde C_i^{-|\:\vec n\:|}\right)
\]
on $[a_i,b_i]\setminus\big(\overline{U\cup\Omega_{\vec n}}\big)$ for some $\widetilde C_i>1$ by \eqref{Phi-ratio} and \eqref{EqPropi}. That is, all the jump matrices for $\boldsymbol Z$ asymptotically behave like $\boldsymbol I + \mathcal{O}\big( \varepsilon_{\vec n}\big)$ (as will be clear in Section~\ref{sec:Local}, the decay of $\varepsilon_{\vec n}$ is of power type and not exponential). The conclusion of the lemma follows from the same argument as in \cite[Corollary~7.108]{Deift}.
\end{proof}

\subsection{Proof of Theorem~\ref{thm:SA}}
\label{ssec:proof}
Let $\boldsymbol{Z}$ be the solution of \hyperref[rhr]{\rhr} granted by Lemma~\ref{lem:rhr}, $\boldsymbol P_e$ be solutions of \hyperref[rhp]{\rhp}, and $\boldsymbol N=\boldsymbol{CMD}$ be the matrix constructed in \eqref{eq:n}. Then it can be easily checked that
\begin{equation}
\label{X}
\boldsymbol X = \boldsymbol C\boldsymbol Z
\left\{
\begin{array}{rll}
\boldsymbol M\boldsymbol D & \text{in} & \C\setminus \big(\overline U\cup\bigcup \big[a_{\vec c,i}b_{\vec c,i}\big]\big), \medskip \\
\boldsymbol P_e & \text{in} & U_e, \quad e\in E_\mathsf{out} \cup E_\mathsf{in} \cup E_{\vec c},
\end{array}
\right.
\end{equation}
solves  \hyperref[rhx]{\rhx} for all $|\:\vec n\:|$ large enough. Given a closed set $K$ in $\overline\C\setminus\bigcup_{i=1}^p[a_i,b_i]$, we can always shrink the lens so that $K\subset \C\setminus \big(\overline{U\cup \Omega_{\vec n}}\big)$. In this case $\boldsymbol Y=\boldsymbol X$ on $K$ by Lemma~\ref{lem:rhx}. Write the first row of $\boldsymbol Z$ as $\left( 1+\upsilon_{\vec n,0}, \upsilon_{\vec n,1}, \ldots, \upsilon_{\vec n,p}\right)$. Then $(1,j+1)$-st entry of $\boldsymbol{ZM}$ is equal to
\[
\left(1 + \upsilon_{\vec n,0} + \sum_{i=1}^p \upsilon_{\vec n,i}\Upsilon_{\vec n,i}^{(j)} \right)S_{\vec n}^{(j)}/w_{\vec n,j} = \big(1 + \mathcal{O}(\varepsilon_{\vec n})\big)S_{\vec n}^{(j)}/w_{\vec n,j}
\]
by Lemma~\ref{lem:rhr} and \eqref{local-un2}, where $w_{\vec n,0}\equiv1$. Therefore, it follows from Proposition~\ref{prop:rhy} that
\[
\left\{
\begin{array}{lll}
Q_{\vec n} &=& C_{\vec n,0}\big[1 + \mathcal{O}(\varepsilon_{\vec n})\big]\big(S_{\vec n}\Phi_{\vec n}\big)^{(0)} \medskip \\
R_{\vec n}^{(j)} &=& C_{\vec n,0}\big[1 + \mathcal{O}(\varepsilon_{\vec n})\big]\big(S_{\vec n}\Phi_{\vec n}\big)^{(j)}/w_{\vec n,j}.
\end{array}
\right.
\]
Theorem~\ref{thm:SA} now follows from \eqref{local-un1}, since $C_{\vec n,0}=(1+o(1))C_{\vec n}$ again by \eqref{local-un1} and  $w_{\vec n,j}\to w_j$ uniformly on $K$.

\section{Local Riemann-Hilbert Analysis}
\label{sec:Local}

The goal of this section is to construct solutions to \hyperref[rhp]{\rhp}. 

\subsection{Local Parametrices around Points in $E_\mathsf{out}$}

Let $e\in E_\mathsf{out}$, see \eqref{Es}. A  solution of \hyperref[rhp]{\rhp} is given by
\begin{equation}
\label{PeC5}
\boldsymbol P_e := \boldsymbol M\mathsf{T}_i\left( \begin{matrix} 1 & \mathcal{C}_i\Phi_{\vec n}^{(0)}/\Phi_{\vec n}^{(i)} \\ 0 & 1 \end{matrix} \right)\boldsymbol D,
\end{equation}
where $\mathcal{C}_i(z) :=\frac{1}{2\pi\mathrm{i}}\int_{[a_i,b_i]}\frac{\rho_i(x)}{x-z}\mathrm{d}x$. Indeed, since the matrices $\boldsymbol M$ and $\boldsymbol D$ are holomorphic in $U_e$, and $\mathcal{C}_i$ has a jump only across $(a_i,b_i)\cap U_e$, the matrix above satisfies \hyperref[rhp]{\rhp}(a). As  $\big(\mathcal{C}_i^+-\mathcal{C}_i^-\big)(x)=\rho_i(x)$ for $x\in(a_i,b_i)\setminus \{x_{ij}\}$, \hyperref[rhp]{\rhp}(b) follows. \hyperref[rhp]{\rhp}(c) is a consequence of the fact that $\big|\mathcal{C}_i(z)(z-x_{ij})^{-\alpha_{ij}}\big|$ is bounded in the vicinity of $x_{ij}$ for $\alpha_{ij}<0$, \cite[Sec.~8.3]{Gakhov}. Finally, \hyperref[rhp]{\rhp}(d) is easily deduced from the inclusion $\overline U_e\subset D_i^- $, \eqref{Phi-ratio} and \eqref{EqPropi}.

\subsection{Local Parametrices around Points in $E_\mathsf{in}$}

The construction below is known \cite{Van03,FMMFSou10,FMMFSou11,DIKr11}.

\subsubsection{Conformal Maps}

Since $h$ is a rational function on $\RS$, it holds that $h^{(0)}_\pm = h^{(i)}_\mp$ on $\big(a_{\vec c,i},b_{\vec c,i}\big)\cap U_e$. Then
\begin{equation}
\label{zeta-vec-5}
\zeta_e(z) := \mathsf{sgn}\big(\im(z)\big)\mathrm{i}\int_e^z\left(h^{(0)}-h^{(i)}\right)(x)\mathrm{d}x, \quad \im(z)\neq0,
\end{equation}
extends to a conformal function in $U_e$ vanishing at $e$. Define $\zeta_{\vec n,e}$ exactly as in \eqref{zeta-vec-5} with $h$ replaced by $h_{\vec n}$. Then it holds that
\begin{equation}
\label{zeta-vecn-5}
\zeta_{\vec n,e}(z) = \frac{\mathsf{sgn}\big(\im(z)\big)\mathrm{i}}{|\:\vec n\:|}\log\left(\Phi_{\vec n}^{(0)}(z)/\Phi_{\vec n}^{(i)}(z)\right), \quad \im(z)\neq0,
\end{equation}
by \eqref{PhiInt}. It follows from \eqref{Phi-ratio} and \eqref{EqPropi} that $\zeta_{\vec n,e}$ is real on $\big(a_{\vec c,i},b_{\vec c,i}\big)\cap U_e$. Moreover, since $U_e\setminus\big(a_{\vec c,i},b_{\vec c,i}\big)\subset D_i^+$, $\zeta_{\vec n,e}$ maps upper half-plane into the upper half-plane. In particular, $\zeta_{\vec n,e}(x)>0$ for $x\in\big(e,b_{\vec c,i}\big)\cap U_e$. Observe also that
\begin{equation}
\label{zeta-asymp}
\zeta_{\vec n,e} \to \zeta_e
\end{equation}
holds uniformly on $\overline U_e$ by \eqref{Phi-ratio} since \eqref{Phi-ratio} is the statement about convergence of the imaginary parts of $\zeta_{\vec n,e}$ to the imaginary part of $\zeta_e$.

\subsubsection{Matrix $\boldsymbol P_e$}

It follows from the way we extended $\rho_i$ into $\Omega_{\vec n,i}^\pm$ that we can write
\begin{equation}
\label{rho_eX}
\rho_i(z) = \rho_{\mathsf{r},e}(z)\left\{\begin{array}{rl}(e-z)^\alpha, & \re(z)<e, \smallskip \\ \beta(z-e)^\alpha, & \re(z)>e, \end{array}\right.
\end{equation}
where $\rho_{\mathsf{r},e}(x)$ is a holomorphic and non-vanishing function in $U_e$. Define $r_e$ by 
\[
r_e(z) := \sqrt{\rho_{\mathsf{r},e}(z)}(z-e)^{\alpha/2},
\]
where the square root is principal. Then $r_e$ is a holomorphic and non-vanishing function in $U_e\setminus\{x:x<e\}$ that satisfies
\begin{equation}
\label{rho_eXB}
\left\{
\begin{array}{ll}
r_{e+}(x)r_{e-}(x) = \rho_i(x), & x\in \{x:x<e\} \cap U_e, \medskip \\
r_e^2(z) = \rho_i(z)e^{\pm\pi\mathrm{i}\alpha}, & z\in\Gamma_{\vec n,ij}^\pm\cap U_e, \medskip \\
r_e^2(x) = \beta^{-1}\rho_i(x), & \left(\Gamma_{\vec n,ij+1}^+\cup \Gamma_{\vec n,ij+1}^- \cup \{x:x>e\}\right)\cap U_e.
\end{array}
\right.
\end{equation}
It is a straightforward computation using \eqref{rho_eXB} and \eqref{zeta-vecn-5} to verify that \hyperref[rhp]{\rhp} is solved by
\[
\boldsymbol P_e := \boldsymbol E_e\mathsf{T}_i\left(\boldsymbol\Phi_{\alpha,\beta}\left(|\:\vec n\:|\zeta_{\vec n,e}\right) r_e^{-\sigma_3}\left(\Phi^{(0)}_{\vec n}/\Phi^{(i)}_{\vec n}\right)^{-\sigma_3/2}\right)\boldsymbol D,
\]
where $\boldsymbol\Phi_{\alpha,\beta}$ is the solution of \hyperref[rhphiAB]{\rhphiAB} and the holomorphic prefactor $\boldsymbol E_e$ chosen below to fulfill \hyperref[rhp]{\rhp}(d). 

\subsubsection{Holomorphic Prefactor $\boldsymbol E_e$}

It follows from the properties of the branch of $\big(\mathrm{i}\zeta\big)^{\log\beta\sigma_3/2\pi\mathrm{i}}$ that
\begin{equation}
\label{HP1}
\big(\mathrm{i}\zeta\big)_+^{\log\beta\sigma_3/2\pi\mathrm{i}} \boldsymbol B_+ = \big(\mathrm{i}\zeta\big)_-^{\log\beta\sigma_3/2\pi\mathrm{i}} \boldsymbol B_-
\left\{
\begin{array}{lll}
\left(\begin{matrix} 0 & 1 \\ -1 & 0 \end{matrix}\right) & \text{on} & (-\infty,0), \medskip \\
\left(\begin{matrix} 0 & \beta \\ -1/\beta & 0 \end{matrix}\right) & \text{on} & (0,\infty),
\end{array}
\right.
\end{equation}
and it is holomorphic in $\C\setminus(-\infty,\infty)$. Therefore, it follows from \hyperref[rhn]{\rhn}(b) that
\[
\boldsymbol E_e := \boldsymbol M\mathsf{T}_i\left(\big(\mathrm{i}|\:\vec n\:|\zeta_{\vec n,e}\big)^{\log\beta\sigma_3/2\pi\mathrm{i}}\boldsymbol B_\pm r_e^{-\sigma_3}\right)^{-1}, \quad \pm\im(z)>0,
\]
is holomorphic in $U_e\setminus\{e\}$. Since $|r_e(z)|\sim|z-e|^{\alpha/2}$ and $\big|\zeta^{\log\beta/2\pi\mathrm{i}}\big|\sim|\zeta|^{\arg(\beta)/2\pi}$, $\boldsymbol E_e$ is in fact holomorphic in $U_e$ as claimed. Clearly, in this case it holds that $\varepsilon_{\vec n,e}=|\:\vec n\:|^{|\arg(\beta)|/\pi-1}$.

\subsection{Hard Edge}

In this section we assume that $e\in E_{\vec c}$ and $e\not\in \partial D_i^-$.

\subsubsection{Conformal Maps}

It follows from Proposition~\ref{prop:derivative} that $b_{\vec c,i}=b_{\vec n,i}=b_i$ or $a_{\vec c,i}=a_{\vec n,i}=a_i$ for all $|\:\vec n\:|$ large in this case. Define
\begin{equation}
\label{zeta-vec-1}
\zeta_e(z) := \left(\frac14\int_e^z\left(h^{(0)}-h^{(i)}\right)(x)\mathrm{d}x\right)^2, \quad z\in U_e.
\end{equation}
Since $h_\pm^{(0)}=h_\mp^{(i)}$ on $(a_i,b_i)\cap U_e$, $\zeta_e$ is holomorphic in $U_e$. Moreover, since $h$ has a pole at $\boldsymbol e$ (the corresponding branch point of $\RS$), $\zeta_e$ has a simple zero at $e$. Thus, we can choose $U_e$ small enough so that $\zeta_e$ is conformal in $\overline U_e$.

Define $\zeta_{\vec n,e}$ as in \eqref{zeta-vec-1} with $h$ replaced by $h_{\vec n}$. The functions $\zeta_{\vec n,e}$ form a family of holomorphic functions in $U_e$, all having a simple zero at $e$. Moreover, \eqref{PhiInt} yields that
\begin{equation}
\label{zeta-vecn-1}
\zeta_{\vec n,e}(z) = \left(\frac1{4|\:\vec n\:|}\log\left(\Phi_{\vec n}^{(0)}/\Phi_{\vec n}^{(i)}\right)\right)^2, \quad z\in U_e,
\end{equation}
which, together with \eqref{Dipm} and \eqref{Phi-ratio}, implies that $\zeta_{\vec n,e}(x)$ is positive for $x\in\big(\R\setminus[a_i,b_i])\cap U_e$ and is negative $x\in (a_i,b_i)\cap U_e$ (this also can be seen from \eqref{MapInt} and \eqref{MapInt1}). 

Considering $h_{\vec n}$ and $h$ as defined on the same doubly circular neighborhood of $\boldsymbol e$ and recalling that their ratio converges to 1 on its boundary, we see that it converges to 1 uniformly throughout the neighborhood. The latter implies that \eqref{zeta-asymp} holds uniformly on $\overline U_e$. In particular, the functions $\zeta_{\vec n,e}$ are conformal in $\overline U_e$ for all $\vec n$ large.

\subsubsection{Matrix $\boldsymbol P_e$}

In this case we can write
\begin{equation}
\label{rho_e}
\rho_i(z)=\rho_{\mathsf{r},e}(z)\left\{\begin{array}{ll}(e-z)^\alpha, & e=b_i, \smallskip \\ (z-e)^\alpha, & e=a_i, \end{array}\right.
\end{equation}
where $\rho_{\mathsf{r},e}$ is non-vanishing and holomorphic in $U_e$, $\alpha>-1$, and the $\alpha$-roots are principal. Set
\begin{equation}
\label{r_i}
r_e(z) := \sqrt{\rho_{\mathsf{r},e}(z)}\left\{\begin{array}{ll}(z-e)^{\alpha/2}, & e=b_i, \smallskip \\ (e-z)^{\alpha/2}, & e=a_i, \end{array}\right.
\end{equation}
where the branches are again principal. Then $r_e$ is a holomorphic and non-vanishing function in $U_e\setminus[a_i,b_i]$ and satisfies
\begin{equation}
\label{r_i-prop}
\left\{
\begin{array}{ll}
r_{e+}(x)r_{e-}(x) = \rho_i(x), & x\in(a_i,b_i), \medskip \\
r_e^2(z) = \rho_i(z)e^{\pm\pi\mathrm{i}\alpha}, & z\in\Gamma_{\vec n,i}^\pm\cap U_e.
\end{array}
\right.
\end{equation}
Then \eqref{zeta-vecn-1} and \eqref{r_i-prop} imply that \hyperref[rhp]{\rhp} is solved by
\begin{equation}
\label{PeC1}
\boldsymbol P_e := \boldsymbol E_e \mathsf{T}_i\left(\boldsymbol\Psi_e\left(|\:\vec n\:|^2\zeta_{\vec n,e}\right) r_e^{-\sigma_3}\left(\Phi^{(0)}_{\vec n}/\Phi^{(i)}_{\vec n}\right)^{-\sigma_3/2}\right)\boldsymbol D,
\end{equation}
where $\boldsymbol \Psi_e:=\boldsymbol\Psi_\alpha$ when $e=b_i$ and $\boldsymbol \Psi_e:=\sigma_3\boldsymbol\Psi_\alpha\sigma_3$ when $e=a_i$, and $\boldsymbol \Psi_\alpha$ solves \hyperref[rhpsiA]{\rhpsiA}, while $\boldsymbol E_e$ is a holomorphic prefactor chosen so that \hyperref[rhp]{\rhp}(d) is fulfilled.

\subsubsection{Holomorphic Prefactor $\boldsymbol E_e$}

As $\zeta_+^{1/4}=\mathrm{i}\zeta^{1/4}_-$, it can be easily checked that
\[
\frac{\zeta_+^{-\sigma_3/4}}{\sqrt2}\left(\begin{matrix} 1 & \pm\mathrm{i} \\ \pm\mathrm{i} & 1 \end{matrix}\right) = \frac{\zeta_-^{-\sigma_3/4}}{\sqrt2}\left(\begin{matrix} 1 & \pm\mathrm{i} \\ \pm\mathrm{i} & 1 \end{matrix}\right) \left(\begin{matrix} 0 & \pm1 \\ \mp1 & 0 \end{matrix}\right)
\]
on $(-\infty,0)$. Then \hyperref[rhn]{\rhn}(b) implies that
\begin{equation}
\label{HP2}
\boldsymbol E_e := \boldsymbol M\mathsf{T}_i\left(\frac{(|\:\vec n\:|^2\zeta_{\vec n,e}\big)^{-\sigma_3/4}}{\sqrt2}\left(\begin{matrix} 1 & \pm\mathrm{i} \\ \pm\mathrm{i} & 1 \end{matrix}\right) r_e^{-\sigma_3}\right)^{-1}
\end{equation}
is holomorphic around in $U_e\setminus\{e\}$, where the sign $+$ is used around $e=b_i$ while the sign $-$ is used around $e=a_i$. Since $|r_e(z)|\sim|z-e|^{\alpha/2}$, $\boldsymbol E_e$ is in fact holomorphic in $U_e$ as desired. Clearly,  $\varepsilon_{\vec n,e}=|\:\vec n\:|^{-1}$ in this case.

\subsection{Soft-Type Edge I}

Below, we assume that $e\in E_{\vec c}$ and $b_{\vec n,i}\in\partial D_{\vec n,i}^-$ or $a_{\vec n,i}\in\partial D_{\vec n,i}^-$.

\subsubsection{Conformal Maps}

By the condition of this section, it holds that $e\in\partial D_i^-$. Define
\begin{equation}
\label{zeta-vec-2}
\zeta_e(z) := \left(-\frac34\int_e^z\left(h^{(0)}-h^{(i)}\right)(x)\mathrm{d}x\right)^{2/3}, \quad z\in U_e.
\end{equation}
Further, define $\zeta_{\vec n,e}$ exactly as $\zeta_e$ only with $h$ replaced by $h_{\vec n}$ and $e$ replaced by $b_{\vec n,i}$ if $e=b_{\vec c,i}$ and by $a_{\vec n,i}$ if $e=a_{\vec c,i}$. It follows from \eqref{PhiInt} that
\begin{equation}
\label{zeta-vecn-2}
\zeta_{\vec n,e}(z) = \left(-\frac3{4|\:\vec n\:|}\log\left(\Phi_{\vec n}^{(0)}(z)/\Phi_{\vec n}^{(i)}(z)\right)\right)^{2/3}, \quad z\in U_e.
\end{equation}
Analysis in \eqref{MapInt} and \eqref{MapInt2} yields that these functions are conformal in $\overline U_e$ (make the radius smaller if necessary), are positive on $\big(\R\setminus\big[a_{\vec n,i},b_{\vec n,i}\big])\cap U_e$ and negative on $\big(a_{\vec n,i},b_{\vec n,i})\cap U_e$. Moreover, \eqref{zeta-asymp} holds as well.

\subsubsection{Matrix $\boldsymbol P_e$}

If $e=x_{ij}$ for some $j\in\{1,\ldots,J_i-1\}$, set $\alpha:=\alpha_{ij}$ and $\beta:=\beta_{ij}$ when $e=b_{\vec c,i}$ or $\beta:=1/\beta_{ij}$ when $e=a_{\vec c,i}$, see \eqref{rho_s}; if $e\not\in\{x_{ij}\}_{j=1}^{J_i-1}$ and $e\in(a_i,b_i)$, set $\alpha=0$ and $\beta=1$; if $e=a_i$, set $\alpha=\alpha_{i0}$ and $\beta=0$; if $e=b_i$, set $\alpha=\alpha_{iJ_i}$ and $\beta=0$. It follows from the way we extended $\rho_i$ into $\Omega_{\vec n,i}^\pm$ that
\begin{equation}
\label{rho_eA}
\rho_i(z) = \rho_{\mathsf{r},e}(z)\left\{\begin{array}{ll}(e-z)^\alpha, & e=b_{\vec c,i}, \smallskip \\ (z-e)^\alpha, & e=a_{\vec c,i}, \end{array}\right.
\end{equation}
for $\re(z)\in\big(a_{\vec c,i},b_{\vec c,i}\big)$ and
\begin{equation}
\label{rho_eB}
\rho_i(z) = \beta\rho_{\mathsf{r},e}(z)\left\{\begin{array}{ll}(z-e)^\alpha, & e=b_{\vec c,i}, \smallskip \\ (e-z)^\alpha, & e=a_{\vec c,i}, \end{array}\right.
\end{equation}
for $\re(z)\not\in\big[a_{\vec c,i},b_{\vec c,i}\big]$, where all the branches are principal. Define $r_e$ by \eqref{r_i} with $b_i$ and $a_i$ replaced by $b_{\vec c,i}$ and $a_{\vec c,i}$. Then $r_e$ is a holomorphic and non-vanishing function in $U_e\setminus\big[a_{\vec c,i},b_{\vec c,i}\big]$ that satisfies
\begin{equation}
\label{r_eprop2}
\left\{
\begin{array}{ll}
r_{e+}(x)r_{e-}(x) = \rho_i(x), & x\in\big(a_{\vec c,i},b_{\vec c,i}\big)\cap U_e, \medskip \\
r_e^2(z) = \rho_i(z)e^{\pm\pi\mathrm{i}\alpha}, & z\in\Gamma_{\vec n,i}^\pm\cap U_e, \medskip \\
r_e^2(x) = \beta^{-1}\rho_i(x), & \left(\R\setminus\big(a_{\vec c,i},b_{\vec c,i}\big)\right)\cap U_e.
\end{array}
\right.
\end{equation}
Then one can check using \eqref{r_eprop2} and \eqref{zeta-vecn-2} that \hyperref[rhp]{\rhp} is solved by
\begin{equation}
\label{PeC2}
\boldsymbol P_e := \boldsymbol E_e \mathsf{T}_i\left(\boldsymbol\Psi_e\left(|\:\vec n\:|^{2/3}\big(\zeta_{\vec n,e}-\zeta_{\vec n,e}(e)\big)\right)r_e^{-\sigma_3}\left(\Phi^{(0)}_{\vec n}/\Phi^{(i)}_{\vec n}\right)^{-\sigma_3/2}\right)\boldsymbol D,
\end{equation}
where $\boldsymbol \Psi_e:=\boldsymbol\Psi_{\alpha,\beta}(\cdot;s_{\vec n})$ when $e=b_{\vec c,i}$ and $\boldsymbol \Psi_e:=\sigma_3\boldsymbol\Psi_{\alpha,\beta}(\cdot;s_{\vec n})\sigma_3$ when $e=a_{\vec c,i}$, $\boldsymbol \Psi_{\alpha,\beta}(\cdot;s)$ solves \hyperref[rhpsiAB]{\rhpsiAB},
\[
s_{\vec n} := |\:\vec n\:|^{2/3}\zeta_{\vec n,e}(e),
\]
and $\boldsymbol E_e$ is a holomorphic prefactor chosen so \hyperref[rhp]{\rhp}(d) is satisfied.

\subsubsection{Holomorphic Prefactor $\boldsymbol E_e$}

If $s_{\vec n} = 0$, then $\boldsymbol E_e$ is given by \eqref{HP2} with $|\:\vec n\:|^2$ replaced by $|\:\vec n\:|^{2/3}$. In this case we have by Theorem~\ref{thm:localRH} that $\varepsilon_{\vec n,e}=|\:\vec n\:|^{-1/3}$.

If $s_{\vec n}>0$, then \eqref{HP2} is no longer applicable as the matrix $\boldsymbol M$ has the jump only across $\big(a_{\vec n,i},b_{\vec n,i})$ while $r_e^{-\sigma_3}$ is discontinuous across $\big(a_{\vec c,i},b_{\vec c,i}\big)\cap U_e$ where $b_{\vec n,i}<b_{\vec c,i}$ or $a_{\vec n,i}>a_{\vec c,i}$. Observe that 
\[
r_{e+}(x) = r_{e-}(x)e^{\alpha\pi\mathrm{i}}, \quad x\in\left( \big(a_{\vec c,i},b_{\vec c,i}\big)\setminus \big(a_{\vec n,i},b_{\vec n,i}\big) \right)\cap U_e.
\]
Therefore, define
\[
G_\alpha(\zeta) := \exp\left\{-\pi\mathrm{i}\alpha\sqrt\zeta\frac1{2\pi\mathrm{i}}\int_0^1\frac1{\sqrt x}\frac{\mathrm{d}x}{x-\zeta}\right\}, \quad \zeta\in\C\setminus(-\infty,1].
\]
It is quite easy to see that
\[
\left\{
\begin{array}{lll}
G_{\alpha+}G_{\alpha-} \equiv1 & \text{on} & (-\infty,0), \medskip \\ 
G_{\alpha-}=G_{\alpha+}\pi\mathrm{i}\alpha & \text{on} & (0,1).
\end{array}
\right.
\]
Moreover, from the theory of singular integrals \cite[Sec.~8.3]{Gakhov} we know that $G_\alpha$ is bounded around the origin and behaves like $|\zeta-1|^{-\alpha/2}$ around 1. Then it can be checked using the above properties that the matrix function
\[
\boldsymbol E_e:=\boldsymbol M\mathsf{T}_i \left( \frac{\left(|\:\vec n\:|^{2/3}\zeta_{\vec n,e}\right)^{-\sigma_3/4}}{\sqrt2} \left(\begin{matrix} 1 & \pm\mathrm{i} \\ \pm\mathrm{i} & 1 \end{matrix}\right) \left( G_\alpha \circ \left(\zeta_{\vec n,e}/\zeta_{\vec n,e}(e)\right) r_e \right)^{-\sigma_3} \right)^{-1}
\]
is holomorphic in $U_e$. With such $\boldsymbol E_e$ it holds that
\[
\boldsymbol P_e = \boldsymbol M \mathsf{T}_i\left(G_\alpha^{-\sigma_3} \circ \left(\zeta_{\vec n,e}/\zeta_{\vec n,e}(e)\right) \left(\boldsymbol I+\mathcal{O}\left(\varepsilon_{\vec n,e}\right)\right)\right) \boldsymbol D
\]
uniformly on $\partial U_e\setminus\big((a_i,b_i)\cup\Gamma_{\vec n,i}^+\cup\Gamma_{\vec n,i}^-\big)$, where
\begin{equation}
\label{kl}
\varepsilon_{\vec n,e}=\max\left\{|\zeta_{\vec n,e}(e)|^{1/2},|\:\vec n\:|^{-1/3}\right\},
\end{equation}
according to Theorem~\ref{thm:localRH}. To see that \hyperref[rhp]{\rhp}(d) is fulfilled it only remains to notice that $G_\alpha(\zeta)=1+\mathcal{O}\big(\zeta^{-1/2}\big)$ as $\zeta\to\infty$ uniformly in $\C\setminus(-\infty,1]$.

If $s_{\vec n}<0$, we need to modify \eqref{HP2} again because $\boldsymbol M$ still has its jump over $\big(a_{\vec n,i},b_{\vec n,i}\big)$ while $r_e$ over $\big(a_{\vec c,i},b_{\vec c,i})$ where $b_{\vec n,i}>b_{\vec c,i}$ or $a_{\vec n,i}<a_{\vec c,i}$. Define
\begin{equation}
\label{Fbeta}
F_\beta (\zeta) := \beta^{1/2} \left( \frac{\mathrm{i}+(\zeta-1)^{1/2}}{\mathrm{i}-(\zeta-1)^{1/2}} \right)^{\log\beta/2\pi\mathrm{i}}, \quad \zeta\in\C\setminus(-\infty,1].
\end{equation}
This function is holomorphic in the domain of its definition, tends to 1 as $\zeta\to\infty$, and satisfies
\[
F_{\beta+}(x)F_{\beta-}(x) = \left\{
\begin{array}{ll}
1, & x\in(-\infty,0), \medskip \\
\beta, & x\in(0,1).
\end{array}
\right.
\]
Indeed, the function $(\mathrm{i}+\sqrt{\zeta-1})/(\mathrm{i}-\sqrt{\zeta-1})$ maps the complement of $(-\infty,1]$ to the lower half-plane, its traces on $(-\infty,1)$ are reciprocal to each other, are positive on $(0,1)$, and are negative on $(-\infty,0)$. The stated properties now easily follow if we take the principal branch of $\log\beta/2\pi\mathrm{i}$ root of this function. Then
\[
\boldsymbol E_e:=\boldsymbol M \mathsf{T}_i\left( \frac{\left(|\:\vec n\:|^{2/3}\zeta_{\vec n,e}\right)^{-\sigma_3/4}}{\sqrt2} \left(\begin{matrix} 1 & \pm\mathrm{i} \\ \pm\mathrm{i} & 1 \end{matrix}\right) \left(F_\beta \circ \left(\frac{\zeta_{\vec n,e}(e)-\zeta_{\vec n,e}}{\zeta_{\vec n,e}(e)}\right) r_e \right)^{-\sigma_3} \right)^{-1}
\]
is holomorphic in $U_e\setminus\{e\}$. Since $|r_e(z)|\sim|z-e|^{\alpha/2}$ as $z\to e$, one can deduce as before $\boldsymbol E_e$ is holomorphic in $U_e$. Moreover, exactly as in the case $s_{\vec n}>0$, we get that \hyperref[rhp]{\rhp} holds with $\varepsilon_{\vec n,e}$ given by \eqref{kl} since $F_\beta(\zeta)=1+\mathcal{O}\big(\zeta^{-1/2}\big)$ as $\zeta\to\infty$.

\subsection{Soft-Type Edge II}

Let $e\in E_{\vec c}$, $e\in\partial D_i^-$, but $b_{\vec n,i}\not\in\partial D_{\vec n,i}^-$ or $a_{\vec n,i}\not\in\partial D_{\vec n,i}^-$. In this case it necessarily holds that $b_{\vec n,i}=b_{\vec c,i}=b_i$ or $a_{\vec n,i}=a_{\vec c,i}=a_i$.

\subsubsection{Conformal Maps}

By Proposition~\ref{prop:derivative}, $h$ is bounded at $\e$ (the corresponding branch point of $\RS$) while $h_{\vec n}$  has a simple pole at $\e$ (this time $\e$ is a branch point of $\RS_{\vec n}$, but it has the same projection $e$) and a simple zero $\z_{\vec n,i}$ or $\z_{\vec n,i-1}$ that approaches $\e$. Hence, we can write
\begin{equation}
\label{last-part}
-\frac34\int_e^z\left(h_{\vec n}^{(0)}-h_{\vec n}^{(i)}\right)(x)\mathrm{d}x = \sqrt{z-e}\left(z-e-\epsilon_{\vec n}\right)f_{\vec n}(z),
\end{equation}
where $0\leq \epsilon_{\vec n}\to 0$ as $|\:\vec n\:|\to\infty$ and $f_{\vec n}$ is non-vanishing in some neighborhood of $e$ and is positive on the real line within this neighborhood (one can factor out $\sqrt{z-e}$ as the square of the left-hand side is holomorphic exactly as in \eqref{zeta-vec-1} and \eqref{zeta-vecn-1}). Then there exist functions $\zeta_{\vec n,e}$, conformal in $U_e$, vanishing at $e$, real on $\R\cap U_e$, and positive for $x>e$ in $U_e$ such that
\begin{equation}
\label{zeta-vecn-4}
-\frac34\int_e^z\left(h_{\vec n}^{(0)}-h_{\vec n}^{(i)}\right)(x)\mathrm{d}x = \zeta_{\vec n,e}^{3/2}(z) - \zeta_{\vec n,e}(e+\epsilon_{\vec n})\zeta_{\vec n,e}^{1/2}(z).
\end{equation}
Moreover, \eqref{zeta-asymp} holds, where $\zeta_e$ is defined by \eqref{zeta-vec-2}, and the left-hand side of \eqref{zeta-vecn-4} is equal to the right-hand side of \eqref{zeta-vecn-2}. Indeed, consider the equation
\begin{equation}
\label{alg-eq}
u(z;\epsilon)\big(u(z;\epsilon)-p\big)^2 = g(z;\epsilon), \quad g(z;\epsilon) := z(z-\epsilon)^2f(z;\epsilon),
\end{equation}
where $p$ is a parameter, $f(z;\epsilon)$ is positive on the real line in some neighborhood of zero and $g^{1/3}(z;0)$ is conformal in this neighborhood. The solution of \eqref{alg-eq} is given by
\begin{equation}
\label{sol-alg-eq}
u(z;\epsilon) = 2p + v^{1/3}(z;\epsilon) + p^2v^{-1/3}(z;\epsilon),
\end{equation}
where $v(z;\epsilon)$ is the branch satisfying $v^{1/3}(0;\epsilon)=-p$ of
\begin{equation}
\label{aux-alg-eq}
v(z;\epsilon) = g(z;\epsilon) - p^3 + \sqrt{g(z;\epsilon)\big(g(z;\epsilon)-2p^3\big)}.
\end{equation}
Choose $p$ so that
\begin{equation}
\label{param-ch}
2p ^3 = \max_{x\in[0,\epsilon]}g(x;\epsilon).
\end{equation}
Conformality of $g^{1/3}(z;0)$ implies that there exists the unique $x_\epsilon>\epsilon$ such that
\[
\left\{
\begin{array}{ll}
g(x;\epsilon)\big(g(x;\epsilon)-2p^3\big) < 0, & x\in(0,x_\epsilon)\setminus\{\epsilon\}, \medskip \\
g(x;\epsilon)\big(g(x;\epsilon)-2p^3\big) > 0, & x > x_\epsilon,
\end{array}
\right.
\]
for all $\epsilon$ small enough. Then we can see from \eqref{aux-alg-eq} that
\begin{equation}
\label{circle-eq}
|v_\pm(x;\epsilon)|^2 = \big(g(x;\epsilon) - p^3)^2 - g(x;\epsilon)\big(g(x;\epsilon)-2p^3\big) = p^6
\end{equation}
for $x\in[0,x_\epsilon]$. Moreover, it holds that
\begin{equation}
\label{traces-eq}
v_+(x;\epsilon) = p^2v_-^{-1}(x;\epsilon), \quad x\in[0,x_\epsilon].
\end{equation}
Finally, observe that the conformality of $g^{1/3}(z;0)$ yields that the change of the argument of  $v_+(x;\epsilon)$ is $3\pi$ when $x$ changes between $0$ and $x_\epsilon$. Hence, $v^{1/3}(z;\epsilon)$ is holomorphic off $[0,\epsilon]$ and its traces on $[0,\epsilon]$ map this interval onto the circle centered at the origin of radius $p$ by \eqref{circle-eq}. This together with \eqref{traces-eq} implies that $u(z;\epsilon)$ given by \eqref{sol-alg-eq} is conformal in some neighborhood of the origin and $u(0;\epsilon)=0$. Thus, $\zeta_{\vec n,e}$ in \eqref{zeta-vecn-4} is given by
\[
\zeta_{\vec n,e}(z) = u(z-e;\epsilon_{\vec n}),
\]
where $u(z;\epsilon)$ is the solution given by \eqref{sol-alg-eq} of \eqref{alg-eq} with $f(z;\epsilon):=f_{\vec n}^2(z-e)$ and the parameter $p$ chosen as in \eqref{param-ch}.

\subsubsection{Matrix $\boldsymbol P_e$}

Clearly, formulae \eqref{rho_e} and \eqref{r_i} remain valid in this case. Then \eqref{r_i-prop} and \eqref{zeta-vecn-4} imply that the solution of \hyperref[rhp]{\rhp} is given by
\[
\boldsymbol P_e := \boldsymbol E_e \mathsf{T}_i\left(\boldsymbol\Psi_e\left(|\:\vec n\:|^{2/3}\zeta_{\vec n,e}\right) r_e^{-\sigma_3} \left(\Phi^{(0)}_{\vec n}/\Phi^{(i)}_{\vec n}\right)^{-\sigma_3/2}\right)\boldsymbol D,
\]
where $\boldsymbol E_e$ is given by \eqref{HP2} with $|\:\vec n\:|^2$ replaced by $|\:\vec n\:|^{2/3}$, $\boldsymbol\Psi_e=\widetilde{\boldsymbol\Psi}_{\alpha,0}(\cdot;s_{\vec n})$ when $e=b_i$ and $\boldsymbol\Psi_e=\sigma_3\widetilde{\boldsymbol\Psi}_{\alpha,0}(\cdot;s_{\vec n})\sigma_3$ when $e=a_i$,
\[
s_{\vec n} := - |\:\vec n\:|^{2/3}\zeta_{\vec n,e}(e+\epsilon_{\vec n}),
\]
and $\widetilde{\boldsymbol\Psi}_{\alpha,\beta}$ is the solution of \hyperref[rhwpsiAB]{\rhwpsiAB}. In this case, it holds by Theorem~\ref{thm:localRH} that
\[
\varepsilon_{\vec n,e} = \max\left\{\zeta_{\vec n,e}^{1/2}(e+\epsilon_{\vec n}),|\:\vec n\:|^{-1/3}\right\}.
\]

\section{Model Riemann-Hilbert Problem \hyperref[rhpsiAB]{\rhpsiAB}}
\label{sec:rhpsiAB}

In this section we prove Theorem~\ref{thm:localRH}.

\subsection{Uniqueness and Existence}

%The first claim of the theorem can be obtained by literally repeating the steps of \cite[Lemma~1]{XuZh11} with $e^{2\theta_+}$ and $e^{2\theta_-}$ in (59) and (67) replaced by $e^{\pi\mathrm i\alpha+2\theta_+}$ and $e^{-\pi\mathrm i\alpha-2\theta_-}$, respectively (the behavior in (62) changes as it depends on $\alpha$ now, but the trace of $N$ on $\R$ is still integrable and therefore (63) still holds). The fact that only the zero function solves (67) (now, with non-zero $\alpha$) was, in fact, proven in \cite[Eq. (2.27)--(2.29)]{IKOs08}.

Since all the jump matrices in \hyperref[rhpsiAB]{\rhpsiAB}(b) have unit determinant, $\det(\boldsymbol\Psi_{\alpha,\beta})$ is holomorphic in $\C\setminus\{0\}$. By  \hyperref[rhpsiAB]{\rhpsiAB}(d), it holds that $\det(\boldsymbol\Psi_{\alpha,\beta})(\infty)=1$. It also follows from \hyperref[rhpsiAB]{\rhpsiAB}(c) that $\det(\boldsymbol\Psi_{\alpha,\beta})$ cannot have a polar singularity at the origin. Hence, $\det(\boldsymbol\Psi_{\alpha,\beta})\equiv1$. In particular, any solution of \hyperref[rhpsiAB]{\rhpsiAB} is invertible. If $\boldsymbol\Psi_1$ and $\boldsymbol\Psi_2$ are two such solutions, then it is easy to verify that $\boldsymbol\Psi_1\boldsymbol\Psi_2^{-1}$ is holomorphic in $\C$. Moreover, $\boldsymbol\Psi_1\boldsymbol\Psi_2^{-1}(\zeta) = \boldsymbol I + \mathcal{O}(1/\zeta)$ as $\zeta\to\infty$. Thus, $\boldsymbol\Psi_1\boldsymbol\Psi_2^{-1}=\boldsymbol I$, which proves uniqueness.

\subsection{Local Behavior}

To proceed with the existence, we need more detailed description of the behavior of $\boldsymbol\Psi_{\alpha,\beta}$ at the origin. Denote by $\Omega_1$, $\Omega_2$, $\Omega_3$, and $\Omega_4$ consecutive sectors of $\C\setminus\big((-\infty,\infty)\cup I_-\cup I_+\big)$ staring with the one containing the first quadrant and continuing counter clockwise. Then we can write
\begin{equation}
\label{sectors}
\boldsymbol\Psi_{\alpha,\beta}(\zeta) = \boldsymbol E(\zeta) \boldsymbol S_{\alpha,\beta}(\zeta)\boldsymbol A_j, \quad \zeta\in\Omega_j,
\end{equation}
where $\boldsymbol E$ is a holomorphic matrix function, 
\begin{equation}
\label{As}
\boldsymbol A_3 =  \boldsymbol A_4 \left(\begin{matrix} 1 & 0 \\ e^{-\alpha\pi\mathrm{i}} & 1 \end{matrix}\right), \quad \boldsymbol A_4 = \boldsymbol A_1\left(\begin{matrix} 1 & -\beta \\ 0 & 1 \end{matrix}\right), \quad \boldsymbol A_1 = \boldsymbol A_2\left(\begin{matrix} 1 & 0 \\ e^{\alpha\pi\mathrm{i}} & 1 \end{matrix}\right),
\end{equation}
and
\begin{equation}
\label{S1}
\boldsymbol A_2 = \left(\begin{matrix} \frac1{2\cos(\alpha\pi/2)}\frac{ 1-\beta e^{\alpha\pi\mathrm{i}} }{ 1-e^{\alpha\pi\mathrm{i}} } & \frac1{2\cos(\alpha\pi/2)}\frac{ \beta - e^{\alpha\pi\mathrm{i}} }{ 1-e^{\alpha\pi\mathrm{i}} } \medskip \\ -e^{\alpha\pi\mathrm{i}/2} & e^{-\alpha\pi\mathrm{i}/2} \end{matrix}\right) \quad \text{while} \quad \boldsymbol S_{\alpha,\beta}(\zeta) = \zeta^{\alpha\sigma_3/2}
\end{equation}
when $\alpha$ is not an integer, 
\begin{equation}
\label{S2}
\boldsymbol A_2 = \left(\begin{matrix} \frac12 e^{\alpha\pi\mathrm{i}/2} & \frac12 e^{-\alpha\pi\mathrm{i}/2} \medskip \\ -e^{\alpha\pi\mathrm{i}/2} & e^{-\alpha\pi\mathrm{i}/2} \end{matrix}\right) \quad \text{while} \quad \boldsymbol S_{\alpha,\beta}(\zeta) = \left(\begin{matrix} \zeta^{\alpha/2} & \frac{1-\beta}{2\pi\mathrm{i}}\zeta^{\alpha/2}\log\zeta \medskip \\ 0 & \zeta^{-\alpha/2} \end{matrix}\right)
\end{equation}
when $\alpha$ is an even integer,
\begin{equation}
\label{S3}
\boldsymbol A_2 = \left(\begin{matrix} 0 & e^{-\alpha\pi\mathrm{i}/2} \medskip \\ -e^{\alpha\pi\mathrm{i}/2} & e^{-\alpha\pi\mathrm{i}/2} \end{matrix}\right) \quad \text{while} \quad \boldsymbol S_{\alpha,\beta}(\zeta) = \left(\begin{matrix} \zeta^{\alpha/2} & \frac{1+\beta}{2\pi\mathrm{i}}\zeta^{\alpha/2}\log\zeta \medskip \\ 0 & \zeta^{-\alpha/2} \end{matrix}\right)
\end{equation}
when $\alpha$ is an odd integer (observe that $\det(\boldsymbol A_j)=1$ for all $j\in\{1,2,3,4\}$). Indeed, equation \eqref{sectors} can be viewed as a definition of the matrix $\boldsymbol E$. Relations \eqref{As} are chosen so $\boldsymbol E$ is holomorphic across $I_\pm$ and $(0,\infty)$. Moreover, on $(-\infty,0)$ it holds that
\[
\boldsymbol E_-^{-1}\boldsymbol E_+ = \boldsymbol S_{\alpha,\beta-} \boldsymbol A_2 \left(\begin{matrix} 1 & 0 \\ e^{\alpha\pi\mathrm{i}} & 1 \end{matrix}\right)\left(\begin{matrix} 1 & -\beta \\ 0 & 1 \end{matrix}\right)\left(\begin{matrix} 1 & 0 \\ e^{-\alpha\pi\mathrm{i}} & 1 \end{matrix}\right)\left(\begin{matrix} 0 & 1 \\ -1 & 0 \end{matrix}\right) \boldsymbol A_2^{-1}\boldsymbol S_{\alpha,\beta+}^{-1} = \boldsymbol I,
\]
where the last equality is a tedious but straightforward computation. Hence, $\boldsymbol E$ is holomorphic in $\C\setminus\{0\}$. Using \hyperref[rhpsiAB]{\rhpsiAB}(d) and \eqref{sectors}, one can verify that $\boldsymbol E$ cannot have a polar singularity at 0 and therefore is entire as claimed.

\subsection{Vanishing Lemma}

The crucial step in showing solvability of \hyperref[rhpsiAB]{\rhpsiAB} is the following result. Assume $\boldsymbol F_{\alpha,\beta}$ satisfies  \hyperref[rhpsiAB]{\rhpsiAB}(a,b,c) and it holds that
\begin{equation}
\label{PsiABd}
\boldsymbol F_{\alpha,\beta}(\zeta) = \mathcal{O}\left(\zeta^{-1}\right) \frac{\zeta^{-\sigma_3/4}}{\sqrt2} \left(\begin{matrix} 1 & \mathrm{i} \\ \mathrm{i} & 1 \end{matrix}\right) \exp\left\{-\left(\frac23\zeta^{3/2}+s\zeta^{1/2}\right)\sigma_3\right\}
\end{equation}
as $\zeta\to\infty$ uniformly in $\C\setminus\big(I_+\cup I_-\cup(-\infty,\infty)\big)$. Then $\boldsymbol F_{\alpha,\beta}\equiv0$. To prove this claim, we follow the line of argument from \cite{IKOs08} and \cite{XuZh11}. Set, for brevity, $\theta(\zeta):=\left(\frac23\zeta^{3/2}+s\zeta^{1/2}\right)$. Assuming $\re(\beta)\geq0$, define
\[
\boldsymbol G_{\alpha,\beta}(\zeta) := \left\{
\begin{array}{ll}
\boldsymbol F_{\alpha,\beta}(\zeta)e^{\theta(\zeta)\sigma_3} \left(\begin{matrix} 0 & -1 \\ 1 & 0 \end{matrix}\right), & \zeta\in\Omega_1, \medskip \\
\boldsymbol F_{\alpha,\beta}(\zeta)e^{\theta(\zeta)\sigma_3} \left(\begin{matrix} 1 & 0 \\ e^{\alpha\pi\mathrm{i}}e^{2\theta(\zeta)} & 1 \end{matrix}\right) \left(\begin{matrix} 0 & -1 \\ 1 & 0 \end{matrix}\right), & \zeta\in\Omega_2,  \medskip \\
\boldsymbol F_{\alpha,\beta}(\zeta)e^{\theta(\zeta)\sigma_3} \left(\begin{matrix} 1 & 0 \\ -e^{-\alpha\pi\mathrm{i}}e^{2\theta(\zeta)} & 1 \end{matrix}\right), & \zeta\in\Omega_3,  \medskip \\
\boldsymbol F_{\alpha,\beta}(\zeta)e^{\theta(\zeta)\sigma_3}, & \zeta\in\Omega_4.
\end{array}
\right.
\] 
Then $\boldsymbol G_{\alpha,\beta}$ satisfies the following Riemann-Hilbert problem (RHP-$\boldsymbol G_{\alpha,\beta}$): 
\begin{itemize}
\label{rhgAB}
\item[(a)] $\boldsymbol G_{\alpha,\beta}$ is holomorphic in $\C\setminus(-\infty,\infty)$;
\item[(b)] $\boldsymbol G_{\alpha,\beta}$ has continuous traces on $(-\infty,0)\cup(0,\infty)$ that satisfy
\[
\boldsymbol G_{\alpha,\beta+} = \boldsymbol G_{\alpha,\beta-}
\left\{
\begin{array}{rll}
\left(\begin{matrix} 1 & -e^{\alpha\pi\mathrm{i}}e^{2\theta_+} \\ e^{-\alpha\pi\mathrm{i}}e^{2\theta_-} & 0 \end{matrix}\right) & \text{on} & (-\infty,0), \medskip \\
\left(\begin{matrix} \beta e^{-2\theta} & -1 \\ 1 & 0 \end{matrix}\right) & \text{on} & (0,\infty);
\end{array}
\right.
\]
\item[(c)] as $\zeta\to0$  it holds that
\[
\boldsymbol G_{\alpha,\beta}(\zeta) = \mathcal{O}\left( \begin{matrix} |\zeta|^{\alpha/2} & |\zeta|^{\alpha/2} \\ |\zeta|^{\alpha/2} & |\zeta|^{\alpha/2} \end{matrix} \right) 
\]
when $\alpha<0$,
\[
\boldsymbol G_{\alpha,\beta}(\zeta) = \mathcal{O}\left( \begin{matrix} 1 & 1+(1-\beta)\log|\zeta| \\ 1 & 1+(1-\beta)\log|\zeta| \end{matrix} \right) \quad \text{and} \quad \boldsymbol G_{\alpha,\beta}(\zeta) = \mathcal{O}\left( \begin{matrix} 1+(1-\beta)\log|\zeta| & 1 \\ 1+(1-\beta)\log|\zeta| & 1 \end{matrix} \right)
\]
when $\alpha=0$, for $\im(\zeta)<0$ and $\im(\zeta)>0$, respectively, and
\[
\boldsymbol G_{\alpha,\beta}(\zeta) = \mathcal{O}\left( \begin{matrix} |\zeta|^{\alpha/2} & |\zeta|^{-\alpha/2} \\ |\zeta|^{\alpha/2} & |\zeta|^{-\alpha/2} \end{matrix} \right) \quad \text{and} \quad \boldsymbol G_{\alpha,\beta}(\zeta) = \mathcal{O}\left( \begin{matrix} |\zeta|^{-\alpha/2} & |\zeta|^{\alpha/2} \\ |\zeta|^{-\alpha/2} & |\zeta|^{\alpha/2} \end{matrix} \right)
\]
when $\alpha>0$, for $\im(\zeta)<0$ and $\im(\zeta)>0$, respectively;
\item[(d)] $\boldsymbol G_{\alpha,\beta}=\mathcal{O}(\zeta^{-3/4})$ as $\zeta\to\infty$.
\end{itemize}
Properties \hyperref[rhgAB]{RHP-$\boldsymbol G_{\alpha,\beta}$}(a,b) can be easily verified using \hyperref[rhpsiAB]{\rhpsiAB}(a,b), the definition of $\boldsymbol G_{\alpha,\beta}$, and the fact that $\theta_++\theta_-\equiv 0$ on $(-\infty,0)$. To show \hyperref[rhgAB]{RHP-$\boldsymbol G_{\alpha,\beta}$}(c), observe that the representations \eqref{sectors}--\eqref{S3} holds for $\boldsymbol F_{\alpha,\beta}$ as well. They imply that
\[
\boldsymbol G_{\alpha,\beta} = \boldsymbol {ES}_{\alpha,\beta}\boldsymbol A_1\left(\begin{matrix} 0 & -1 \\ 1 & 0 \end{matrix}\right)e^{-\theta\sigma_3} \quad \text{and} \quad \boldsymbol G_{\alpha,\beta} = \boldsymbol {ES}_{\alpha,\beta}\boldsymbol A_4e^{\theta\sigma_3}
\]
in $\Omega_1\cup\Omega_2$ and $\Omega_3\cup\Omega_4$, respectively. Since $[\boldsymbol A_1]_{21}=[\boldsymbol A_4]_{21}=0$, \hyperref[rhgAB]{RHP-$\boldsymbol G_{\alpha,\beta}$}(c) follows. Finally, \hyperref[rhgAB]{RHP-$\boldsymbol G_{\alpha,\beta}$}(d) is the consequence of the fact that $\re(\theta)<0$ in $\Omega_2\cup\Omega_3$.

For the next step of the proof consider the matrix function
\[
\boldsymbol G_{\alpha,\beta}(\zeta)\big(\boldsymbol G_{\alpha,\beta}(\bar\zeta)\big)^*, \quad \zeta\not\in(-\infty,\infty),
\]
where $\big(\boldsymbol G_{\alpha,\beta}\big)^*$ is the hermitian conjugate of $\boldsymbol G_{\alpha,\beta}$. This matrix function is holomorphic off the real line, has integrable traces on the real line by \hyperref[rhgAB]{RHP-$\boldsymbol G_{\alpha,\beta}$}(c) (recall that $\alpha>-1$), and vanishes at infinity as $\zeta^{-3/2}$ by \hyperref[rhgAB]{RHP-$\boldsymbol G_{\alpha,\beta}$}(d). Thus, we deduce from Cauchy's theorem that the integral of its traces over the real line is zero, i.e.,
\begin{equation}
\label{last-minute} 
0 = \int_{-\infty}^\infty \boldsymbol G_{\alpha,\beta+}(x)\big(\boldsymbol G_{\alpha,\beta-}(x)\big)^*\mathrm{d}x = \int_{-\infty}^\infty \boldsymbol G_{\alpha,\beta-}(x)\big(\boldsymbol G_{\alpha,\beta+}(x)\big)^*\mathrm{d}x.
\end{equation}
Adding the last two integrals together and using \hyperref[rhgAB]{RHP-$\boldsymbol G_{\alpha,\beta}$}(b), we get
\begin{multline*}
0 = \int_{-\infty}^0 \boldsymbol G_{\alpha,\beta-}(x) \left(\begin{matrix} 2 & 0 \\ 0 & 0 \end{matrix}\right) \big(\boldsymbol G_{\alpha,\beta-}(x)\big)^* \mathrm{d}x \\
+ \int_0^\infty \boldsymbol G_{\alpha,\beta-}(x) \left(\begin{matrix} (\beta+\bar\beta)e^{-2\theta(x)} & 0 \\ 0 & 0 \end{matrix}\right) \big(\boldsymbol G_{\alpha,\beta-}(x)\big)^* \mathrm{d}x.
\end{multline*}
The above equality yields that the first column of $\boldsymbol G_{\alpha,\beta-}$ vanishes identically on $(-\infty,0)$ (on the whole real line if $\beta$ is not purely imaginary). In any case, since $\boldsymbol G_{\alpha,\beta-}$ consists of traces of holomorphic functions, its first column vanishes identically in the lower half-plane. Thus, by \hyperref[rhgAB]{RHP-$\boldsymbol G_{\alpha,\beta}$}(b), the second column of $\boldsymbol G_{\alpha,\beta}$ vanishes identically in the upper half-plane. To finish proving that $\boldsymbol G_{\alpha,\beta}\equiv0$ in the case $\re(\beta)\geq0$ (and therefore $\boldsymbol F_{\alpha,\beta}\equiv0$), set
\[
g_i(\zeta) := \left\{
\begin{array}{ll}
~[\boldsymbol G_{\alpha,\beta}]_{i1}(\zeta), & \im(\zeta)>0, \medskip \\
~[\boldsymbol G_{\alpha,\beta}]_{i2}(\zeta), & \im(\zeta)<0.
\end{array}
\right.
\]
Both functions $g_i$ are holomorphic in $\C\setminus(-\infty,0]$, satisfy $g_i(\zeta)=\mathcal{O}(\zeta^{-3/4})$ as $\zeta\to\infty$ and $g_i(\zeta)=\mathcal{O}(|\zeta|^{-|\alpha|/2})$ as $\zeta\to0$, while their traces are related by the formula
\[
g_{i+}(x) = g_{i-}(x)e^{-\alpha\pi\mathrm{i}}e^{2\theta_-(x)}, \quad x\in(-\infty,0).
\]
The latter is possible only if $g_i\equiv0$ as shown in \cite[Def. (2.26) and below]{IKOs08}.

When $\re(\beta)<0$ and $\im(\beta)\neq0$, let us redefine $\boldsymbol G_{\alpha,\beta}$ in $\Omega_1$ and $\Omega_2$ by setting
\[
\boldsymbol G_{\alpha,\beta}(\zeta) := \left\{
\begin{array}{ll}
\boldsymbol F_{\alpha,\beta}(\zeta)e^{\theta(\zeta)\sigma_3} \left(\begin{matrix} 0 & 1 \\ 1 & 0 \end{matrix}\right), & \zeta\in\Omega_1, \medskip \\
\boldsymbol F_{\alpha,\beta}(\zeta)e^{\theta(\zeta)\sigma_3} \left(\begin{matrix} 1 & 0 \\ e^{\alpha\pi\mathrm{i}}e^{2\theta(\zeta)} & 1 \end{matrix}\right) \left(\begin{matrix} 0 & 1 \\ 1 & 0 \end{matrix}\right), & \zeta\in\Omega_2.
\end{array}
\right.
\]
This newly defined function $\boldsymbol G_{\alpha,\beta}$ still satisfies \hyperref[rhgAB]{RHP-$\boldsymbol G_{\alpha,\beta}$} except for \hyperref[rhgAB]{RHP-$\boldsymbol G_{\alpha,\beta}$}(b), which now becomes
\[
\boldsymbol G_{\alpha,\beta+} = \boldsymbol G_{\alpha,\beta-}
\left\{
\begin{array}{rll}
\left(\begin{matrix} 1 & e^{\alpha\pi\mathrm{i}}e^{2\theta_+} \\ e^{-\alpha\pi\mathrm{i}}e^{2\theta_-} & 0 \end{matrix}\right) & \text{on} & (-\infty,0), \medskip \\
\left(\begin{matrix} \beta e^{-2\theta} & 1 \\ 1 & 0 \end{matrix}\right) & \text{on} & (0,\infty).
\end{array}
\right.
\]
Observe that \eqref{last-minute} remains valid. Thus, by taking the difference of the integrals in \eqref{last-minute}, we arrive at
\[
0 = \int_0^\infty \boldsymbol G_{\alpha,\beta-}(x) \left(\begin{matrix} (\beta-\bar\beta)e^{-2\theta(x)} & 0 \\ 0 & 0 \end{matrix}\right) \big(\boldsymbol G_{\alpha,\beta-}(x)\big)^* \mathrm{d}x.
\]
This again allows us to conclude that the first column of $\boldsymbol G_{\alpha,\beta}$ vanishes identically in the lower half-plane and the second column vanishes in the upper half-plane. The remaining part of the proof is now exactly the same as in the case $\re(\beta)\geq0$.
  
\subsection{Existence}

For $\boldsymbol A_j$ and $\boldsymbol S_{\alpha,\beta}$ as in \eqref{As}--\eqref{S3}, define
\[
\boldsymbol M_{\alpha,\beta}(\zeta) := \left\{
\begin{array}{ll}
\boldsymbol \Psi_{\alpha,\beta}(\zeta)\boldsymbol A_j^{-1}\boldsymbol S_{\alpha,\beta}^{-1}(\zeta), & \zeta\in\Omega_j\cap\big\{|\zeta|<1\big\}, \medskip \\
\boldsymbol \Psi_{\alpha,\beta}(\zeta)e^{\theta(\zeta)\sigma_3}\frac1{\sqrt2}\left(\begin{matrix} 1 & -\mathrm{i} \\ -\mathrm{i} & 1 \end{matrix}\right)\zeta^{\sigma_3/4}, & \zeta\in\Omega_j\cap\big\{|\zeta|\geq1\big\}.
\end{array}
\right.
\]
Further, let the contour $\Sigma_{\boldsymbol M}$ be as on Figure~\ref{fig:SigmaM} with its subarcs oriented so that
\begin{figure}[!ht]
\centering
\includegraphics[scale=.5]{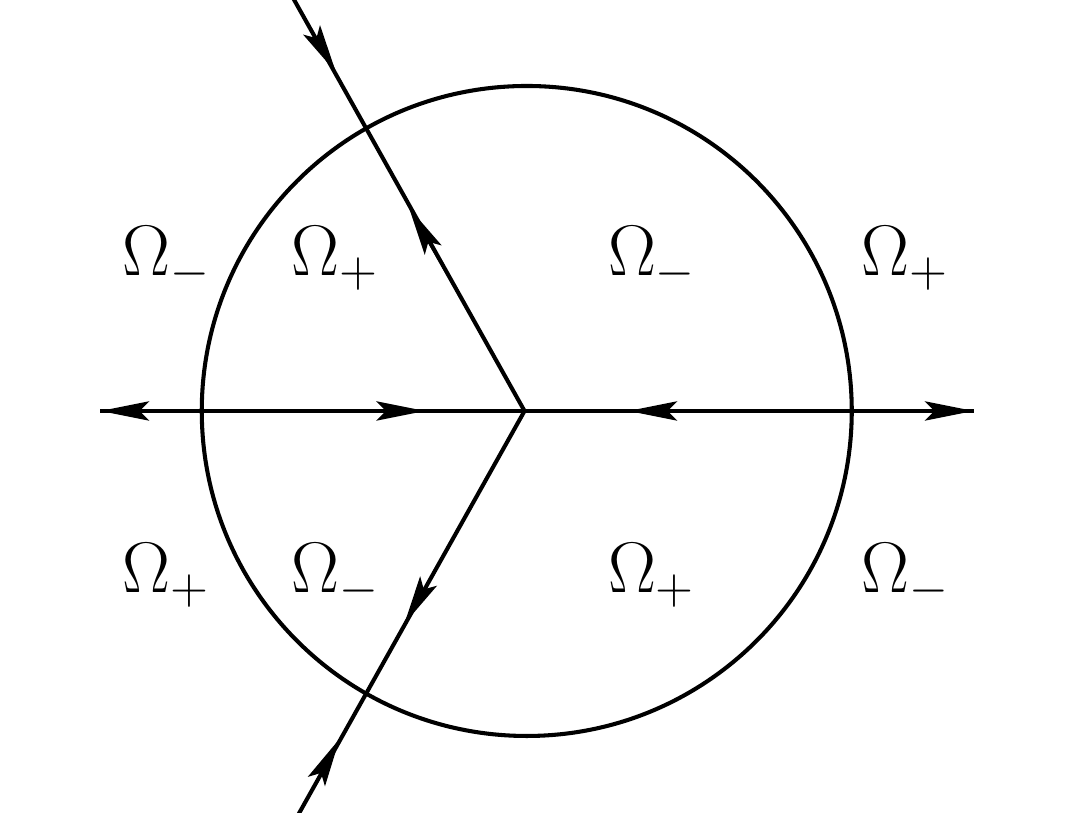}
\caption{\small Contours $\Sigma_{\boldsymbol M}$.}
\label{fig:SigmaM}
\end{figure} 
$\C\setminus\Sigma_{\boldsymbol M}=\Omega_+\cup\Omega_-$, where $\Sigma_{\boldsymbol M}$ is positively oriented boundary of $\Omega_+$ and is negatively oriented boundary of $\Omega_-$. If $\boldsymbol \Psi_{\alpha,\beta}$ uniquely solves \hyperref[rhpsiAB]{\rhpsiAB}, then $\boldsymbol M_{\alpha,\beta}$ uniquely solves  the following Riemann-Hilbert problem (RHP-$\boldsymbol M_{\alpha,\beta}$):
\begin{itemize}
\label{rhmAB}
\item[(a)] $\boldsymbol M_{\alpha,\beta}$ is holomorphic in $\C\setminus\Sigma_{\boldsymbol M}$ and $\boldsymbol M_{\alpha,\beta}(\zeta)=\boldsymbol I+\mathcal{O}(1/\zeta)$ as $\zeta\to\infty$;
\item[(b)] $\boldsymbol M_{\alpha,\beta}$ has continuous traces on $\Sigma_{\boldsymbol M}$ that satisfy $\boldsymbol M_{\alpha,\beta+} = \boldsymbol M_{\alpha,\beta-}\boldsymbol J$, where
\[
\boldsymbol J(\zeta) = \left(\boldsymbol S_{\alpha,\beta}(\zeta)\boldsymbol A_j e^{\theta(\zeta)\sigma_3}\frac1{\sqrt2}\left(\begin{matrix} 1 & -\mathrm{i} \\ -\mathrm{i} & 1 \end{matrix}\right)\zeta^{\sigma_3/4}\right)^{\pm1},
\]
on $\Omega_j\cap\big\{|\zeta|=1\big\}$ with the exponent $1$ for $j=1,3$ and the exponent $-1$ for $j=2,4$, and on the rest of the contour $\Sigma_{\boldsymbol M}$ the jump is equal to
\[
\boldsymbol J(\zeta) = \boldsymbol I + e^{-2\theta(\zeta)}\left\{
\begin{array}{rl}
\boldsymbol 0, & \zeta\in\Sigma_{\boldsymbol M}\cap\big\{|\zeta|<1\big\} \: \text{or} \; \zeta\in(-\infty,-1), \medskip \\
\frac\beta2 \left(\begin{matrix} -\mathrm{i} & \zeta^{-1/2} \\ \zeta^{1/2} & \mathrm{i}\end{matrix}\right), & \zeta\in(1,\infty), \medskip \\
\frac{e^{\pm\alpha\pi\mathrm{i}}}2 \left(\begin{matrix} \mathrm{i}\zeta^{-1/2} & 1 \\ 1 & -\mathrm{i}\zeta^{1/2} \end{matrix}\right), & \zeta\in I_\pm \cap \big\{|\zeta|>1\big\}.
\end{array}
\right.
\]
\end{itemize}

According to \cite[Appendix]{KamvissisMcLaughlinMiller}, see also \cite[Prop. 2.4]{IKOs08}, the unique solution of \hyperref[rhmAB]{RHP-$\boldsymbol M_{\alpha,\beta}$} is given by the formula
\[
\boldsymbol M_{\alpha,\beta}(\zeta) = \boldsymbol I + \mathcal{C}\big(\boldsymbol M(\boldsymbol W_++\boldsymbol W_-)\big)(\zeta), \quad \zeta\in\C\setminus\Sigma_{\boldsymbol M},
\]
where $\boldsymbol J = (\boldsymbol I-\boldsymbol W_-)^{-1}(\boldsymbol I-\boldsymbol W_+)$ is a factorization of the jump $\boldsymbol J$ for some $\boldsymbol W_\pm\in L^2(\Sigma_{\boldsymbol M})\cap L^\infty(\Sigma_{\boldsymbol M})$, $\mathcal{C}$ is a Cauchy operator
\[
\mathcal{C}\boldsymbol H(\zeta) = \frac1{2\pi\mathrm{i}}\int_{\Sigma_{\boldsymbol M}}\frac{\boldsymbol H(s)}{s-\zeta}\mathrm{d}s, \quad \boldsymbol H\in L^2(\Sigma_{\boldsymbol M}), \quad  \zeta\in\C\setminus\Sigma_{\boldsymbol M},
\]
and $\boldsymbol M\in \boldsymbol I+L^2(\Sigma_{\boldsymbol M})$ is the solution of the singular integral equation
\begin{equation}
\label{SingEq}
\big(\mathcal{I}-\mathcal{C}_{\boldsymbol W}\big)\boldsymbol M = \boldsymbol I
\end{equation}
for the singular operator $\mathcal{C}_{\boldsymbol W}:L^2(\Sigma_{\boldsymbol M})\to L^2(\Sigma_{\boldsymbol M})$ given by
\[
\mathcal{C}_{\boldsymbol W}\boldsymbol H := \mathcal{C}_+(\boldsymbol{HW}_-) + \mathcal{C}_-(\boldsymbol{HW}_+), \quad \boldsymbol H\in L^2(\Sigma_{\boldsymbol M}),
\]
provided this solution exists and is unique. Indeed, given such $\boldsymbol M$, it holds that
\[
\boldsymbol I + \mathcal{C}_\pm\big(\boldsymbol M(\boldsymbol W_++\boldsymbol W_-)\big) = \boldsymbol I + \mathcal{C}_{\boldsymbol W}\boldsymbol M + \big(\mathcal{C}_\pm-\mathcal{C}_\mp\big)(\boldsymbol{MW}_\pm) = \boldsymbol M \pm \boldsymbol {MW}_\pm
\]
by \eqref{SingEq} and Sokhotski-Pemelj formulae \cite[Section~4.2]{Gakhov}. Then
\[
\left(\boldsymbol I + \mathcal{C}_-\big(\boldsymbol M(\boldsymbol W_++\boldsymbol W_-)\big)\right)^{-1}\left(\boldsymbol I + \mathcal{C}_+\big(\boldsymbol M(\boldsymbol W_++\boldsymbol W_-)\big)\right) = (\boldsymbol I-\boldsymbol W_-)^{-1}(\boldsymbol I-\boldsymbol W_+) = \boldsymbol J
\]
as desired. Thus, only the unique solvability of \eqref{SingEq} needs to be shown. The sufficient condition for the latter is bijectivity of the operator $\mathcal{I}-\mathcal{C}_{\boldsymbol W}$, which can be established by showing that $\mathcal{I}-\mathcal{C}_{\boldsymbol W}$ is Fredholm with index zero and trivial kernel.

To this end, let us specify $\boldsymbol W_\pm$. Away from the points of self-intersection of $\Sigma_{\boldsymbol M}$, set
\[
\boldsymbol W_+ := \boldsymbol J - \boldsymbol I \quad \text{and} \quad \boldsymbol W_- = \boldsymbol 0.
\]
Around the points of self-intersection, we chose $\boldsymbol W_+$ to be continuous along the boundary of $\Omega_+$ and $\boldsymbol W_-$ to be continuous along the boundary of $\Omega_-$. The latter is possible because around each point of self-intersection of $\Sigma_{\boldsymbol M}$, the jumps satisfy the cyclic relation
\[
\boldsymbol J_1\boldsymbol J_2^{-1}\boldsymbol J_3\boldsymbol J_4^{-1} = \boldsymbol I,
\]
where we label the four arcs meeting at the point of self-intersection counter-clockwise starting with an arc oriented away from the point and denote by $\boldsymbol J_i$ the jump across the $i$-th arc. Clearly, $\boldsymbol W_\pm\in L^2(\Sigma_{\boldsymbol M})\cap L^\infty(\Sigma_{\boldsymbol M})$. To show that $\mathcal{I}-\mathcal{C}_{\boldsymbol W}$ is Fredholm, one needs to construct its pseudoinverse. The latter is given by $\mathcal{I}-\mathcal{C}_{\widetilde{\boldsymbol W}}$, where
\[
\widetilde{\boldsymbol W}_+ := \big(\boldsymbol I + \boldsymbol W_+\big)^{-1} - \boldsymbol I \quad \text{and} \quad \widetilde{\boldsymbol W}_- := \boldsymbol I- \big(\boldsymbol I - \boldsymbol W_-\big)^{-1},
\]
as explained in \cite[Eq. (2.39)--(2.42)]{IKOs08}. The index of $\mathcal{I}-\mathcal{C}_{\boldsymbol W}$ is equal to the winding number of $\det(\boldsymbol J)$, which is zero since $\det(\boldsymbol J)\equiv1$. Finally, the kernel of $\mathcal{I}-\mathcal{C}_{\boldsymbol W}$ is trivial if and only if the homogeneous Riemann-Hilbert problem corresponding to \hyperref[rhmAB]{RHP-$\boldsymbol M_{\alpha,\beta}$} has only trivial solutions. Correspondence between \hyperref[rhpsiAB]{\rhpsiAB} and \hyperref[rhmAB]{RHP-$\boldsymbol M_{\alpha,\beta}$} implies that the kernel is trivial if and only if the solution of \hyperref[rhpsiAB]{\rhpsiAB} with \hyperref[rhpsiAB]{\rhpsiAB}(d) replaced by \eqref{PsiABd} has only trivial solutions. This is precisely the content of the preceding subsection. This finishes the proof of the first claim of Theorem~\ref{thm:localRH}.

\subsection{Asymptotics of \hyperref[rhpsiAB]{\rhpsiAB} for $s>0$}

It is known that $\mathcal{O}\left(\eta^{-1}\right)$ is uniform for $s$ on compact subsets of the real line \cite{IKOs08}. Thus, we only need to prove \eqref{betterestimate1} for $s$ large.

\subsubsection{Renormalized \hyperref[rhpsiAB]{\rhpsiAB}}

Set $\widehat I_\pm:=\{\eta:\arg(\eta+1)=\pm 2\pi/3\}$ and let $\widehat\Omega_j$, $j\in\{1,2,3,4\}$, be the domains comprising $\C\setminus\big((-\infty,\infty)\cup \widehat I_+\cup\widehat I_-\big)$, numbered counter-clockwise and so that $\widehat\Omega_1$ contains the first quadrant. Define
\[
g(\eta) = \frac23(\eta+1)^{3/2}, \quad \eta\in\C\setminus(-\infty,-1],
\]
to be the principal branch and set for convenience $\tau:=s^{3/2}$. Let
\begin{equation}
\label{st1}
 \widehat{\boldsymbol\Psi}_{\alpha,\beta}(\eta;\tau) = s^{\sigma_3/4}\boldsymbol\Psi_{\alpha,\beta}(s\eta;s) \left\{
\begin{array}{lll}
\boldsymbol I & \text{in} & \Omega_1\cup\Omega_4\cup\widehat\Omega_2\cup\widehat\Omega_3, \medskip \\
\left(\begin{matrix} 1 & 0 \\ \pm e^{\pm\alpha\pi\mathrm{i}} & 1 \end{matrix}\right) & \text{in} & \Omega_2\setminus\widehat\Omega_2, \; \Omega_3\setminus\widehat\Omega_3,
\end{array}
\right.
\end{equation}
where the sign $+$ is used in $\Omega_2\setminus\widehat\Omega_2$ and the sign $-$ in $\Omega_3\setminus\widehat\Omega_3$. Then $\widehat{\boldsymbol\Psi}_{\alpha,\beta}$ solves the following Riemann-Hilbert problem (\rhhpsiAB):
\begin{itemize}
\label{rhhpsiAB}
\item[(a)] $\widehat{\boldsymbol\Psi}_{\alpha,\beta}$ is holomorphic in $\C\setminus\big(\widehat I_+\cup \widehat I_-\cup(-\infty,\infty)\big)$;
\item[(b)] $\widehat{\boldsymbol\Psi}_{\alpha,\beta}$ has continuous traces on $\widehat I_+\cup \widehat I_-\cup(-\infty,-1)\cup(-1,0)\cup(0,\infty)$ that satisfy
\[
\widehat{\boldsymbol\Psi}_{\alpha,\beta+} = \widehat{\boldsymbol\Psi}_{\alpha,\beta-}
\left\{
\begin{array}{rll}
\left(\begin{matrix} 0 & 1 \\ -1 & 0 \end{matrix}\right) & \text{on} & (-\infty,-1), \medskip \\
\left(\begin{matrix} e^{\alpha\pi\mathrm{i}} & 1 \\ 0 & e^{-\alpha\pi\mathrm{i}} \end{matrix}\right) & \text{on} & (-1,0), \medskip \\
\left(\begin{matrix} 1 & \beta \\ 0 & 1 \end{matrix}\right) & \text{on} & (0,\infty), \medskip \\
\left(\begin{matrix} 1 & 0 \\ e^{\pm\alpha\pi\mathrm{i}} & 1 \end{matrix}\right) & \text{on} & \widehat I_\pm;
\end{array}
\right.
\]
\item[(c)] as $\eta\to0$  it holds that
\[
\widehat{\boldsymbol\Psi}_{\alpha,\beta}(\eta;\tau) = \widehat{\boldsymbol E}(\eta) \boldsymbol S_{\alpha,\beta}(\eta)\boldsymbol A_j, \quad \eta\in\widehat\Omega_j, \quad j\in\{1,4\},
\]
where $\widehat{\boldsymbol E}$ is holomorphic, and $\boldsymbol S_{\alpha,\beta}$, $\boldsymbol A_1$ and $\boldsymbol A_4$ are the same as in \hyperref[rhpsiAB]{\rhpsiAB}(c);
\item[(d)] $\widehat{\boldsymbol\Psi}_{\alpha,\beta}$ has the following behavior near $\infty$:
\[
\widehat{\boldsymbol\Psi}_{\alpha,\beta}(\eta;\tau) = \left(\boldsymbol I+\mathcal{O}\left(\eta^{-1}\right)\right) \frac{\eta^{-\sigma_3/4}}{\sqrt2} \left(\begin{matrix} 1 & \mathrm{i} \\ \mathrm{i} & 1 \end{matrix}\right)e^{-\tau g(\eta)\sigma_3}
\]
uniformly in $\C\setminus\big(\widehat I_+\cup \widehat I_-\cup(-\infty,\infty)\big)$.
\end{itemize}

\subsubsection{Global Parametrix}

Let
\begin{eqnarray}
\widehat{\boldsymbol\Psi}^{(\infty)}(\eta;\tau) &:=& \left(\begin{matrix} 1 & 0 \\ \alpha\mathrm{i} & 1 \end{matrix}\right)\frac{(\eta+1)^{-\sigma_3/4}}{\sqrt2} \left(\begin{matrix} 1 & \mathrm{i} \\ \mathrm{i} & 1 \end{matrix}\right) \left( \frac{(\eta+1)^{1/2}+1}{(\eta+1)^{1/2}-1} \right)^{-\alpha\sigma_3/2}e^{-\tau g(\eta)\sigma_3} \nonumber \\
& =: & \boldsymbol F^{(\infty)}(\tau)e^{-\tau g(\eta)\sigma_3}. \nonumber
\end{eqnarray}
Then, as is explained in \cite[Section~2.4.1]{IKOs09}, this matrix-valued function solves the following Riemann-Hilbert problem:
\begin{itemize}
\item[(a)] $\widehat{\boldsymbol\Psi}^{(\infty)}$ is holomorphic in $\C\setminus(-\infty,0]$;
\item[(b)] $\widehat{\boldsymbol\Psi}^{(\infty)}$ has continuous traces on $(-\infty,-1)\cup(-1,0)$ that satisfy
\[
\widehat{\boldsymbol\Psi}_+^{(\infty)} = \widehat{\boldsymbol\Psi}_-^{(\infty)}
\left\{
\begin{array}{rll}
\left(\begin{matrix} 0 & 1 \\ -1 & 0 \end{matrix}\right) & \text{on} & (-\infty,-1), \medskip \\
e^{\alpha\pi\mathrm{i}\sigma_3} & \text{on} & (-1,0),
\end{array}
\right.
\]
\item[(c)] as $\eta\to0$  it holds that $\widehat{\boldsymbol\Psi}^{(\infty)}(\eta;\tau) = \widehat{\boldsymbol E}^{(\infty)}(\eta)\eta^{\alpha\sigma_3/2}$, where $\widehat{\boldsymbol E}^{(\infty)}$ is holomorphic and non-vanishing around zero;
\item[(d)] $\widehat{\boldsymbol\Psi}^{(\infty)}$ satisfies \hyperref[rhhpsiAB]{\rhhpsiAB}(d) uniformly in $\C\setminus(-\infty;0]$ and the term $\mathcal{O}\big(\eta^{-1}\big)$ does not depend on $\tau$.
\end{itemize}
Notice that $\boldsymbol F^{(\infty)}$ has the same jumps as $\widehat{\boldsymbol\Psi}^{(\infty)}$.

\subsubsection{Local Parametrix Around $-1$}

The solution $\boldsymbol\Psi_\mathsf{Ai}:=\boldsymbol\Psi_{0,1}(\cdot;0)$ is known explicitly and is constructed with the help of the Airy function and its derivative \cite{DKMLVZ99b}. Set
\[
\widehat{\boldsymbol\Psi}^{(-1)}(\eta;\tau) := \widehat{\boldsymbol E}^{(-1)}(\eta)\boldsymbol\Psi_\mathsf{Ai}\big(s(\eta+1)\big)e^{\pm\alpha\pi\mathrm{i}\sigma_3/2}, \quad \pm\im(\eta)>0,
\]
where $ \widehat{\boldsymbol E}^{(-1)}$ is holomorphic around $-1$ and is given by
\[
 \widehat{\boldsymbol E}^{(-1)}(\eta) := \boldsymbol F^{(\infty)}(\eta) \left( \frac{\big(s(\eta+1)\big)^{-\sigma_3/4}}{\sqrt2}\left(\begin{matrix} 1 & \mathrm{i} \\ \mathrm{i} & 1 \end{matrix}\right) e^{\pm\alpha\pi\mathrm{i}\sigma_3/2} \right)^{-1}, \quad \pm\im(\eta)>0.
\]
Let $U_{-1}$ be the disk of radius $1/4$ centered at $-1$ with boundary oriented counter-clockwise. Then it is shown in \cite[Section~2.4.2]{IKOs09} that $\widehat{\boldsymbol\Psi}^{(-1)}$ satisfies
\begin{itemize}
\item[(a)] $\widehat{\boldsymbol\Psi}^{(-1)}$ is holomorphic in $U_{-1}\setminus\big(\widehat I_+\cup \widehat I_-\cup(-\infty,\infty)\big)$;
\item[(b)] $\widehat{\boldsymbol\Psi}^{(-1)}$ has continuous traces on $U_{-1}\cap\big(\widehat I_+\cup \widehat I_-\cup(-\infty,\infty)\big)$ that satisfy \hyperref[rhhpsiAB]{\rhhpsiAB}(b);
\item[(c)] it holds that
\[
\widehat{\boldsymbol\Psi}^{(-1)}(\eta;\tau) = \boldsymbol F^{(\infty)}(\eta)\left(\boldsymbol I+\mathcal{O}\left(\tau^{-1}\right)\right)e^{-\tau g(\eta)\sigma_3}
\]
as $\tau\to\infty$, uniformly for $\eta\in\partial U_{-1}\setminus\big(\widehat I_+\cup \widehat I_-\cup(-\infty,\infty)\big)$.
\end{itemize}

\subsubsection{Local Parametrix Around $0$}

Define
\[
\widehat{\boldsymbol\Psi}^{(0)}(\eta;\tau) := \widehat{\boldsymbol E}^{(0)}(\eta)\boldsymbol S_{\alpha,\beta}(\tau) \left\{
\begin{array}{ll}
\boldsymbol A_1, & \im(\eta)>0, \medskip \\
\boldsymbol A_4, & \im(\eta)<0,
\end{array}
\right.
\]
where $\boldsymbol S_{\alpha,\beta}$ and $\boldsymbol A_j$ are the same as in \hyperref[rhpsiAB]{\rhpsiAB}(c) and
\[
\widehat{\boldsymbol E}^{(0)}(\eta) := \widehat{\boldsymbol\Psi}^{(\infty)}(\eta;\tau)\eta^{-\alpha\sigma_3/2} \left(\begin{matrix}[\boldsymbol A_1]_{11}^{-1} & 0 \medskip \\ 0 & [\boldsymbol A_1]_{22}^{-1} \end{matrix}\right),
\]
which is a holomorphic function around the origin by the properties of $\widehat{\boldsymbol\Psi}^{(\infty)}$. Let $U_0$ be the disk of radius $1/4$ centered at $0$ with boundary oriented counter-clockwise. Then $\widehat{\boldsymbol\Psi}^{(0)}$ possesses the following properties:
\begin{itemize}
\item[(a)] $\widehat{\boldsymbol\Psi}^{(0)}$ is holomorphic in $U_0\setminus(-1/4,1/4)$;
\item[(b)] $\widehat{\boldsymbol\Psi}^{(0)}$ has continuous traces on $(-1/4,0)\cup(0,1/4)$ that satisfy \hyperref[rhhpsiAB]{\rhhpsiAB}(b);
\item[(c)] $\widehat{\boldsymbol\Psi}^{(0)}$ satisfies \hyperref[rhhpsiAB]{\rhhpsiAB}(c) with $\widehat{\boldsymbol E}$ replaced by $\widehat{\boldsymbol E}^{(0)}$;
\item[(d)] it holds that
\[
\widehat{\boldsymbol\Psi}^{(0)}(\eta;\tau) = \boldsymbol F^{(\infty)}(\eta)\left(\boldsymbol I+\mathcal{O}\left(e^{-c\tau}\right)\right)e^{-\tau g(\eta)\sigma_3}
\]
as $\tau\to\infty$ for some $c>0$, uniformly for $\eta\in\partial U_0\setminus\{-1/4,1/4\}$.
\end{itemize}
Indeed, properties (a,b,c) easily follow from \hyperref[rhhpsiAB]{\rhhpsiAB}(b,c) and the holomorphy of $\widehat{\boldsymbol E}^{(0)}$. To get (d), write $[\boldsymbol S_{\alpha,\beta}]_{12}(\eta)=\eta^{\alpha/2}\kappa(\eta)$, where
\[
\kappa(\eta) = 0, \quad \kappa(\eta) = \frac{1-\beta}{2\pi\mathrm{i}}\log\eta, \quad \text{or} \quad \kappa(\eta) = \frac{1+\beta}{2\pi\mathrm{i}}\log\eta
\]
depending on whether $\alpha$ is not an integer, an even integer, or an odd integer. Recall also that $\boldsymbol A_1$ and $\boldsymbol A_4$ are upper triangular matrices and $[\boldsymbol A_1]_{ii}=[\boldsymbol A_4]_{ii}$ for $i\in\{1,2\}$. Then
\begin{eqnarray}
\widehat{\boldsymbol\Psi}^{(0)}(\eta;\tau) &=& \boldsymbol F^{(\infty)}(\eta)e^{-\tau g(\eta)\sigma_3} \left(\begin{matrix} [\boldsymbol A_j]_{11}^{-1} & 0 \medskip \\ 0 & [\boldsymbol A_j]_{22}^{-1} \end{matrix}\right) \left(\begin{matrix} 1 & \kappa(\eta) \medskip \\ 0 & 1 \end{matrix}\right) \boldsymbol A_j \nonumber \\
&=& \boldsymbol F^{(\infty)}(\eta) \left(\begin{matrix} 1 & e^{-2\tau g(\eta)}\big([\boldsymbol A_j]_{22}\kappa(\eta)+[\boldsymbol A_j]_{12}\big)/[\boldsymbol A_j]_{11} \medskip \\ 0 & 1 \end{matrix}\right) e^{-\tau g(\eta)\sigma_3}, \nonumber
\end{eqnarray}
from which property (d) can be easily deduced as $\tau>0$ and $\re(g(\eta))>0$ for $\eta\in\partial U_0$.

\subsubsection{Asymptotics of \hyperref[rhpsiAB]{\rhpsiAB}}

Denote by
\[
\Sigma(\boldsymbol R_{\alpha,\beta}) := \partial U_{-1} \cup \partial U_0\cup \left(\big(\widehat I_- \cup \widehat I_+\cup (-1,\infty)\big)\cap \left(\C\setminus\big(\overline U_{-1}\cup \overline U_0\big)\right) \right),
\]
and let $\Sigma^\circ(\boldsymbol R_{\alpha,\beta})$ be $\Sigma(\boldsymbol R_{\alpha,\beta})$ with the points of self-intersection removed. Put
\[
\boldsymbol R_{\alpha,\beta}(\eta;\tau) := \widehat{\boldsymbol\Psi}_{\alpha,\beta}(\eta;\tau)\left\{
\begin{array}{ll}
\widehat{\boldsymbol\Psi}^{(-1)}(\eta;\tau)^{-1}, & \eta\in U_{-1}, \medskip \\
\widehat{\boldsymbol\Psi}^{(0)}(\eta;\tau)^{-1}, & \eta\in U_0, \medskip \\
\widehat{\boldsymbol\Psi}^{(\infty)}(\eta;\tau)^{-1}, & \eta\in\C\setminus\big(\overline U_0 \cup \overline U_{-1}\big).
\end{array}
\right.
\]
Then $\boldsymbol R_{\alpha,\beta}$ has the following properties:
\begin{itemize}
\item[(a)] $\boldsymbol R_{\alpha,\beta}$ is holomorphic in $\C\setminus\Sigma(\boldsymbol R_{\alpha,\beta})$;
\item[(b)] $\boldsymbol R_{\alpha,\beta}$ has continuous traces on $\Sigma^\circ(\boldsymbol R_{\alpha,\beta})$ that satisfy $\boldsymbol R_{\alpha,\beta+}^{(0)} := \boldsymbol R_{\alpha,\beta-}^{(0)}\left(\boldsymbol I + \mathcal{O}\left(\tau^{-1}\right) \right)$ as $\tau\to\infty$;
\item[(c)] it holds that $\boldsymbol R_{\alpha,\beta}(\eta;\tau) = \boldsymbol I + \mathcal{O}\left(\eta^{-1}\right)$ as $\eta\to\infty$ uniformly in $\C\setminus\Sigma(\boldsymbol R_{\alpha,\beta})$.
\end{itemize}

Property (a) follows from the facts that $\widehat{\boldsymbol\Psi}^{(e)}$ has the same jumps as $\widehat{\boldsymbol\Psi}_{\alpha,\beta}$ in $U_e$, $e\in\{-1,0\}$, $\widehat{\boldsymbol\Psi}^{(\infty)}$ has the same jump across $(-\infty,-1)$ as $\widehat{\boldsymbol\Psi}_{\alpha,\beta}$, and $\widehat{\boldsymbol\Psi}^{(0)}$ has the same local behavior around $0$ as $\widehat{\boldsymbol\Psi}_{\alpha,\beta}$. Property (c) follows easily from the fact that both $\widehat{\boldsymbol\Psi}^{(\infty)}$ and $\widehat{\boldsymbol\Psi}_{\alpha,\beta}$ satisfy \hyperref[rhhpsiAB]{\rhhpsiAB}(d). Property (b) on $\partial U_e$, $e\in\{-1,0\}$, is the consequence of the fact
\[
\boldsymbol R_{\alpha,\beta-}^{-1}\boldsymbol R_{\alpha,\beta+} = \widehat{\boldsymbol\Psi}^{(\infty)}\widehat{\boldsymbol\Psi}^{(e)-1} = \boldsymbol I + \boldsymbol F^{(\infty)}\mathcal{O}\left(\tau^{-1}\right)\boldsymbol F^{(\infty)-1}.
\]
Finally, on the rest of $\Sigma(\boldsymbol R_{\alpha,\beta})$ it holds that
\[
\boldsymbol R_{\alpha,\beta+} = \boldsymbol R_{\alpha,\beta-}
\left\{
\begin{array}{lll}
\boldsymbol I + \boldsymbol F_-^{(\infty)}\left(\begin{matrix} 0 & e^{-2\tau g} \\ 0 & 0 \end{matrix}\right)\boldsymbol F_+^{(\infty)-1} & \text{on} & (-3/4,-1/4), \medskip \\
\boldsymbol I + \boldsymbol F^{(\infty)}\left(\begin{matrix} 0 & \beta e^{-2\tau g} \\ 0 & 0 \end{matrix}\right)\boldsymbol F^{(\infty)-1} & \text{on} & (1/4,\infty), \medskip \\
\boldsymbol I + \boldsymbol F^{(\infty)}\left(\begin{matrix} 0 & 0 \\ e^{\pm\alpha\pi\mathrm{i}}e^{2\tau g} & 0 \end{matrix}\right)\boldsymbol F^{(\infty)-1} & \text{on} & \widehat I_\pm\setminus \overline U_{-1}.
\end{array}
\right.
\]
As $g(\eta)>0$ for $\eta\in(-1,\infty)$ and $g(\eta)<0$ for $\eta\in\widehat I_\pm$, the last part of the property (b) follows. Given (a,b,c) it is by now standard to conclude that
\[
\boldsymbol R_{\alpha,\beta}(\eta;\tau) = \boldsymbol I + \mathcal{O}\left(\frac1{\tau(1+|\eta|)}\right)
\]
as $\tau\to\infty$ uniformly for $\eta\in\C\setminus\Sigma(\boldsymbol R_{\alpha,\beta})$. Thus,
\begin{eqnarray}
\widehat{\boldsymbol\Psi}_{\alpha,\beta}(\eta;\tau) &=& \frac{\eta^{-\sigma_3/4}}{\sqrt2} \left(\boldsymbol I + \mathcal{O}\left(\frac1{\tau\sqrt{1+|\eta|}}\right) \right) \left( \boldsymbol I + \mathcal{O}\left(\eta^{-1/2}\right) \right) \left(\begin{matrix} 1 & \mathrm{i} \\ \mathrm{i} & 1 \end{matrix}\right)e^{-\tau g(\eta)\sigma_3} \nonumber \\
\label{RAB}
&=& \frac{\eta^{-\sigma_3/4}}{\sqrt2} \left( \boldsymbol I + \mathcal{O}\left(\eta^{-1/2}\right) \right) \left(\begin{matrix} 1 & \mathrm{i} \\ \mathrm{i} & 1 \end{matrix}\right)e^{-\tau g(\eta)\sigma_3} 
\end{eqnarray}
as $\eta\to\infty$ uniformly for $\eta\in\C\setminus\Sigma(\boldsymbol R_{\alpha,\beta})$ and $\tau$ large. Estimate \eqref{betterestimate1} now follows from \eqref{st1}.

\subsection{Asymptotics of \hyperref[rhpsiAB]{\rhpsiAB} for $s<0$}

In this section we assume that $\beta\neq0$ and define
\[
\log\beta=\log|\beta|+\mathrm{i}\arg(\beta), \quad \arg(\beta)\in(-\pi,\pi).
\]
Again, we only need to prove \eqref{betterestimate1} when  $s\to-\infty$.

\subsubsection{Renormalized \hyperref[rhpsiAB]{\rhpsiAB}}

Set $\widehat J_\pm$ to be two Jordan arcs connecting $0$ and $1$, oriented from $0$ to $1$, and lying in the first $(+)$ and the fourth ($-$) quadrants.  Denote further by $\Omega_\pm$ the domains delimited by $\widehat J_\pm$ and $[0,1]$. Define
\[
g(\eta) = \frac23(\eta-1)^{3/2}, \quad \eta\in\C\setminus(-\infty,1],
\]
to be the principal branch and set for convenience $\tau:=(-s)^{3/2}$. Let
\begin{equation}
\label{st2}
 \widehat{\boldsymbol\Psi}_{\alpha,\beta}(\eta;\tau) = (-s)^{\sigma_3/4}\boldsymbol\Psi_{\alpha,\beta}(-s\eta;s) \left\{
\begin{array}{lr}
\left(\begin{matrix} 1 & 0 \\ \mp 1/\beta & 1 \end{matrix}\right) & \text{in} \quad \Omega_\pm, \medskip \\
\boldsymbol I & \text{otherwise}.
\end{array}
\right.
\end{equation}
Put for brevity $\Sigma(\widehat{\boldsymbol\Psi}_{\alpha,\beta}) := I_+\cup I_-\cup(-\infty,\infty) \cup \widehat J_+ \cup \widehat J_-$. Then $\widehat{\boldsymbol\Psi}_{\alpha,\beta}$ solves the following Riemann-Hilbert problem (\rhhpsiAB):
\begin{itemize}
\label{rhhpsiAB2}
\item[(a)] $\widehat{\boldsymbol\Psi}_{\alpha,\beta}$ is holomorphic in $\C\setminus\Sigma(\widehat{\boldsymbol\Psi}_{\alpha,\beta})$;
\item[(b)] $\widehat{\boldsymbol\Psi}_{\alpha,\beta}$ has continuous traces on $\Sigma(\widehat{\boldsymbol\Psi}_{\alpha,\beta})\setminus\{0,1\}$ that satisfy
\[
\widehat{\boldsymbol\Psi}_{\alpha,\beta+} = \widehat{\boldsymbol\Psi}_{\alpha,\beta-}
\left\{
\begin{array}{rll}
\left(\begin{matrix} 0 & 1 \\ -1 & 0 \end{matrix}\right) & \text{on} & (-\infty,0), \medskip \\
\left(\begin{matrix} 0 & \beta \\ -1/\beta & 0 \end{matrix}\right) & \text{on} & (0,1), \medskip \\
\left(\begin{matrix} 1 & \beta \\ 0 & 1 \end{matrix}\right) & \text{on} & (1,\infty), \medskip \\
\end{array}
\right.
\]
and
\[
\widehat{\boldsymbol\Psi}_{\alpha,\beta+} = \widehat{\boldsymbol\Psi}_{\alpha,\beta-}
\left\{
\begin{array}{rll}
\left(\begin{matrix} 1 & 0 \\ 1/\beta & 1 \end{matrix}\right) & \text{on} & \widehat J_\pm, \medskip \\
\left(\begin{matrix} 1 & 0 \\ e^{\pm\alpha\pi\mathrm{i}} & 1 \end{matrix}\right) & \text{on} & I_\pm;
\end{array}
\right.
\]
\item[(c)] as $\eta\to0$  it holds that
\[
\widehat{\boldsymbol\Psi}_{\alpha,\beta}(\eta;\tau) = \mathcal{O}\left( \begin{matrix} |\zeta|^{\alpha/2} & |\zeta|^{\alpha/2} + |\zeta|^{-\alpha/2} \\ |\zeta|^{\alpha/2} & |\zeta|^{\alpha/2} + |\zeta|^{-\alpha/2} \end{matrix} \right) \quad \text{and} \quad \widehat{\boldsymbol\Psi}_{\alpha,\beta}(\eta;\tau) = \mathcal{O}\left( \begin{matrix} 1 & \log|\zeta| \\ 1 & \log|\zeta| \end{matrix} \right) 
\]
when $\alpha\neq 0$ and $\alpha=0$, respectively;
\item[(d)] $\widehat{\boldsymbol\Psi}_{\alpha,\beta}$ has the following behavior near $\infty$:
\[
\widehat{\boldsymbol\Psi}_{\alpha,\beta}(\eta;\tau) = \left(\boldsymbol I+\mathcal{O}\left(\eta^{-1}\right)\right) \frac{\eta^{-\sigma_3/4}}{\sqrt2} \left(\begin{matrix} 1 & \mathrm{i} \\ \mathrm{i} & 1 \end{matrix}\right)e^{-\tau g(\eta)\sigma_3}
\]
uniformly in $\C\setminus\big( I_+\cup I_-\cup(-\infty,\infty)\big)$.
\end{itemize}

\subsubsection{Global Parametrix}

Set
\[
\widehat{\boldsymbol\Psi}^{(\infty)}(\eta;\tau):=\boldsymbol F^{(\infty)}(\eta)e^{-\tau g(\eta)\sigma_3},
\]
where
\[
\boldsymbol F^{(\infty)}(\eta) := \left( \begin{matrix} 1 & 0 \\ -\frac1{\pi\mathrm{i}}\log\beta & 1 \end{matrix} \right) \frac{(\eta-1)^{-\sigma_3/4}}{\sqrt2} \left(\begin{matrix} 1 & \mathrm{i} \\ \mathrm{i} & 1 \end{matrix}\right) F_\beta^{-\sigma_3}(\eta)
\]
and the function $F_\beta$ is give by \eqref{Fbeta}. Now, it is a straightforward verification to see that
\begin{itemize}
\item[(a)] $\widehat{\boldsymbol\Psi}^{(\infty)}$ is holomorphic in $\C\setminus(-\infty,1]$;
\item[(b)] $\widehat{\boldsymbol\Psi}^{(\infty)}$ has continuous traces on $(-\infty,1)$ that satisfy
\[
\widehat{\boldsymbol\Psi}_+^{(\infty)} = \widehat{\boldsymbol\Psi}_-^{(\infty)} \left\{
\begin{array}{rll}
\left(\begin{matrix} 0 & 1 \\ -1 & 0 \end{matrix}\right) & \text{on} & (-\infty,0), \medskip \\
\left(\begin{matrix} 0 & \beta \\ -1/\beta & 0 \end{matrix}\right) & \text{on} & (0,1);
\end{array}
\right.
\]
\item[(c)] $\widehat{\boldsymbol\Psi}^{(\infty)}$ satisfies \hyperref[rhhpsiAB]{\rhhpsiAB}(d) uniformly in $\C\setminus(-\infty;1]$ and the term $\mathcal{O}\big(\eta^{-1}\big)$ does not depend on $\tau$.
\end{itemize}

Again, notice that $\widehat{\boldsymbol\Psi}^{(\infty)}$ and $\boldsymbol F^{(\infty)}$ satisfy the same jump relations.

\subsubsection{Local Parametrix Around $1$}

Denote by $U_1$ the disk centered at $1$ of radius $1/4$ with boundary oriented counter-clockwise. Choose arcs $\widehat J_\pm$ so that
\[
\big\{\eta-1:~\eta\in \widehat J_\pm \cap U_1\big\} \subset I_\pm.
\]
Let, as before, $\boldsymbol\Psi_\mathsf{Ai}=\boldsymbol\Psi_{0,1}(\cdot;0)$. Set
\[
\widehat{\boldsymbol\Psi}^{(1)}(\eta;\tau) := \widehat{\boldsymbol E}^{(1)}(\eta)\boldsymbol\Psi_\mathsf{Ai}\big(-s(\eta-1)\big)\beta^{-\sigma_3/2},
\]
where $ \widehat{\boldsymbol E}^{(1)}$ is holomorphic around $1$ and is given by
\[
\widehat{\boldsymbol E}^{(1)}(\eta) := \boldsymbol F^{(\infty)}(\eta) \left(\frac{\big(-s(\eta-1)\big)^{-\sigma_3/4}}{\sqrt2} \left(\begin{matrix} 1 & \mathrm{i} \\ \mathrm{i} & 1 \end{matrix}\right) \beta^{-\sigma_3/2}  \right)^{-1}.
\]
Then it can be checked that $\widehat{\boldsymbol\Psi}^{(1)}$ satisfies
\begin{itemize}
\item[(a)] $\widehat{\boldsymbol\Psi}^{(1)}$ is holomorphic in $U_1\setminus\Sigma(\widehat{\boldsymbol\Psi}_{\alpha,\beta})$;
\item[(b)] $\widehat{\boldsymbol\Psi}^{(1)}$ has continuous traces on $U_1\cap\Sigma(\widehat{\boldsymbol\Psi}_{\alpha,\beta})$ that satisfy \hyperref[rhhpsiAB2]{\rhhpsiAB}(b);
\item[(c)] it holds that
\[
\widehat{\boldsymbol\Psi}^{(1)}(\eta;\tau) = \boldsymbol F^{(\infty)}(\eta)\left(\boldsymbol I+\mathcal{O}\left(\tau^{-1}\right)\right)e^{-\tau g(\eta)\sigma_3}
\]
as $\tau\to\infty$, uniformly for $\eta\in\partial U_1\setminus\Sigma(\widehat{\boldsymbol\Psi}_{\alpha,\beta})$.
\end{itemize}

\subsubsection{Local Parametrix Around $0$}

Denote by $U_0$ the disk centered at $0$ of radius $1/4$ whose boundary oriented counter-clockwise. Let
\[
m(\eta) := 3 \mp 2\mathrm{i}g(\eta), \quad \pm\im(\eta)>0.
\]
Then $m$ is conformal in $U_0$, $m(0)=0$, and $m(x)>0$ for $x\in(0,1/4)$. Choose the arcs $\widehat J_\pm$ so that $m\big(\widehat J_\pm\big) \subset J_\pm$. Define
\[
\widehat{\boldsymbol\Psi}^{(0)}(\eta;\tau) := \widehat{\boldsymbol E}^{(0)}(\eta)\mathcal{D}\left( \boldsymbol\Phi_{\alpha,\beta}\big(\tau m(\eta)\big)\right),
\]
where $\boldsymbol\Phi_{\alpha,\beta}$ is the solution of \hyperref[rhphiAB]{\rhphiAB}, $\mathcal{D}\left( \boldsymbol\Phi_{\alpha,\beta}\big(\tau m\big)\right)$ is a holomorphic deformation of $\boldsymbol\Phi_{\alpha,\beta}\big(\tau m\big)$ that moves the jumps from $(\tau m)^{-1}\big(I_\pm\big)$  to $I_\pm$, and $\widehat{\boldsymbol E}^{(0)}$ is holomorphic around $0$ and is given by
\begin{equation}
\label{E02}
\widehat{\boldsymbol E}^{(0)}(\eta) := \boldsymbol F^{(\infty)}(\eta) \left( e^{-3\tau\mathrm{i}\sigma_3/2}\big(\mathrm{i}\tau m(\eta)\big)^{\log\beta\sigma_3/2\pi\mathrm{i}} \boldsymbol B_\pm\right)^{-1} 
\end{equation}
(the constant matrices $\boldsymbol B_\pm$ were also defined in \hyperref[rhphiAB]{\rhphiAB}). To see that $E^{(0)}$ is indeed holomorphic recall that
\[
\boldsymbol B_+ = \boldsymbol B_-\left(\begin{matrix} 0 & 1 \\ -1 & 0\end{matrix}\right) \quad \text{and} \quad \big(\mathrm{i}x\big)_-^{\log\beta/2\pi\mathrm{i}} = \beta\big(\mathrm{i}x\big)_+^{\log\beta/2\pi\mathrm{i}}
\]
for $x>0$, which implies that the function in parenthesis in \eqref{E02} has the same jump as $\boldsymbol F^{(\infty)}$ on $(-1/4,1/4)$. Observe further that
\[
\boldsymbol B_\pm e^{\mp\mathrm{i}\tau m(\eta)\sigma_3/2} = e^{3\tau\mathrm{i}\sigma_3/2}\boldsymbol B_\pm e^{-\tau g(\eta)\sigma_3}, \quad \pm\im(\eta)>0.
\]
Therefore, it follows from \hyperref[rhphiAB]{\rhphiAB}(d) that
\begin{multline*}
\widehat{\boldsymbol\Psi}^{(0)}(\eta;\tau) = \boldsymbol F^{(\infty)}(\eta) \left( e^{-3\tau\mathrm{i}\sigma_3/2}\big(\mathrm{i}\tau m(\eta)\big)^{\log\beta\sigma_3/2\pi\mathrm{i}} \boldsymbol B_\pm\right)^{-1} \left(\boldsymbol I + \mathcal{O}\left(\tau^{-1}\right)\right)\times \\
\times  \left( e^{-3\tau\mathrm{i}\sigma_3/2}\big(\mathrm{i}\tau m(\eta)\big)^{\log\beta\sigma_3/2\pi\mathrm{i}} \boldsymbol B_\pm\right)e^{-\tau g(\eta)\sigma_3}.
\end{multline*}
Finally, notice that
\[
\left|\tau^{\log\beta/2\pi\mathrm{i}}\right| = \tau^{\arg(\beta)/2\pi}, \quad \arg(\beta)\in(-\pi,\pi).
\]
Thus, $\widehat{\boldsymbol\Psi}^{(0)}$ has the following properties:
\begin{itemize}
\item[(a)] $\widehat{\boldsymbol\Psi}^{(0)}$ is holomorphic in $U_0\setminus\Sigma(\widehat{\boldsymbol\Psi}_{\alpha,\beta})$;
\item[(b)] $\widehat{\boldsymbol\Psi}^{(0)}$ satisfies \hyperref[rhhpsiAB2]{\rhhpsiAB}(b) on $\Sigma(\widehat{\boldsymbol\Psi}_{\alpha,\beta})\cap U_0$;
\item[(c)] $\widehat{\boldsymbol\Psi}^{(0)}$ satisfies \hyperref[rhhpsiAB2]{\rhhpsiAB}(c) within $U_0$ (by \hyperref[rhphiAB]{\rhphiAB}(c));
\item[(d)] it holds that
\[
\widehat{\boldsymbol\Psi}^{(0)}(\eta;\tau) = \boldsymbol F^{(\infty)}(\eta)\left(\boldsymbol I + \mathcal{O}\left(\tau^{\arg(\beta)/\pi-1}\right)\right)e^{-\tau g(\eta)\sigma_3}
\]
as $\tau\to\infty$ uniformly on $\partial U_0\setminus\Sigma(\widehat{\boldsymbol\Psi}_{\alpha,\beta})$.
\end{itemize}

\subsubsection{Asymptotics of \hyperref[rhpsiAB]{\rhpsiAB}}

Define
\[
\boldsymbol R_{\alpha,\beta}(\eta;\tau) := \widehat{\boldsymbol\Psi}_{\alpha,\beta}(\eta;\tau)\left\{
\begin{array}{ll}
\widehat{\boldsymbol\Psi}^{(0)}(\eta;\tau)^{-1}, & \eta\in U_0, \medskip \\
\widehat{\boldsymbol\Psi}^{(1)}(\eta;\tau)^{-1}, & \eta\in U_1, \medskip \\
\widehat{\boldsymbol\Psi}^{(\infty)}(\eta;\tau)^{-1}, & \eta\in\C\setminus\big(\overline U_0 \cup \overline U_1\big).
\end{array}
\right.
\]
Notice that the jumps of $\boldsymbol R_{\alpha,\beta}$ across $\widehat J_\pm\setminus \big(\overline U_0 \cup \overline U_1\big)$ are equal to
\[
\boldsymbol I + \boldsymbol F^{(\infty)-1}\left( \begin{matrix} 0 & 0 \\ e^{2\tau g} & 0 \end{matrix} \right)\boldsymbol F^{(\infty)}.
\]
Since $\re(g)<0$ there, we get exactly as in the case $s>0$ that
\[
\boldsymbol R_{\alpha,\beta}(\eta;\tau) = \boldsymbol I + \mathcal{O}\left(\frac1{\tau^{1-\arg(\beta)/\pi}(1+|\eta|)}\right)
\]
as $\tau\to\infty$ uniformly for $\eta\in\C\setminus\left(\partial U_0 \cup \partial U_1\cup \left(\Sigma(\widehat{\boldsymbol\Psi}_{\alpha,\beta}) \setminus \big(\overline U_0 \cup \overline U_1\big)\right)\right)$. Hence, \eqref{RAB} still holds and therefore \eqref{betterestimate1} follows from \eqref{st2}.

\subsection{Asymptotics of \hyperref[rhwpsiAB]{\rhwpsiAB}}

Below, we assume that $\beta=0$. As before, we only need to prove \eqref{betterestimate2} when  $s\to-\infty$.

\subsubsection{Renormalized \hyperref[rhwpsiAB]{\rhwpsiAB}}

Define
\[
g(\eta) = \frac23\eta^{1/2}(\eta-1), \quad \eta\in\C\setminus(-\infty,1],
\]
to be the principal branch and set for convenience $\tau:=(-s)^{3/2}$. Let
\begin{equation}
\label{st3}
 \widehat{\boldsymbol\Psi}_\alpha(\eta;\tau) = (-s)^{\sigma_3/4}\widetilde{\boldsymbol\Psi}_{\alpha,0}(-s\eta;s).
\end{equation}
Then $\widehat{\boldsymbol\Psi}_\alpha$ solves the following Riemann-Hilbert problem (\rhhpsiAB):
\begin{itemize}
\label{rhhpsiA}
\item[(a)] $\widehat{\boldsymbol\Psi}_\alpha$ is holomorphic in $\C\setminus\big(I_+\cup I_-\cup(-\infty,0]\big)$;
\item[(b)] $\widehat{\boldsymbol\Psi}_\alpha$ has continuous traces on $I_+\cup I_-\cup(-\infty,0)$ that satisfy
\[
\widehat{\boldsymbol\Psi}_{\alpha+} = \widehat{\boldsymbol\Psi}_{\alpha-}
\left\{
\begin{array}{rll}
\left(\begin{matrix} 0 & 1 \\ -1 & 0 \end{matrix}\right) & \text{on} & (-\infty,0), \medskip \\
\left(\begin{matrix} 1 & 0 \\ e^{\pm\alpha\pi\mathrm{i}} & 1 \end{matrix}\right) & \text{on} & I_\pm;
\end{array}
\right.
\]
\item[(c)] as $\eta\to0$  it holds that
\[
\widehat{\boldsymbol\Psi}_\alpha(\eta;\tau) = \mathcal{O}\left( \begin{matrix} |\zeta|^{\alpha/2} & |\zeta|^{\alpha/2} + |\zeta|^{-\alpha/2} \\ |\zeta|^{\alpha/2} & |\zeta|^{\alpha/2} + |\zeta|^{-\alpha/2} \end{matrix} \right) \quad \text{and} \quad \widehat{\boldsymbol\Psi}_\alpha(\eta;\tau) = \mathcal{O}\left( \begin{matrix} 1 & \log|\zeta| \\ 1 & \log|\zeta| \end{matrix} \right) 
\]
when $\alpha\neq 0$ and $\alpha=0$, respectively;
\item[(d)] $\widehat{\boldsymbol\Psi}_\alpha$ has the following behavior near $\infty$:
\[
\widehat{\boldsymbol\Psi}_\alpha(\eta;\tau) = \left(\boldsymbol I+\mathcal{O}\left(\eta^{-1}\right)\right) \frac{\eta^{-\sigma_3/4}}{\sqrt2} \left(\begin{matrix} 1 & \mathrm{i} \\ \mathrm{i} & 1 \end{matrix}\right)e^{-\tau g(\eta)\sigma_3}
\]
uniformly in $\C\setminus\big( I_+\cup I_-\cup(-\infty,\infty)\big)$.
\end{itemize}

\subsubsection{Global Parametrix}

Set
\[
\widehat{\boldsymbol\Psi}^{(\infty)}(\eta;\tau):= \frac{\eta^{-\sigma_3/4}}{\sqrt2} \left(\begin{matrix} 1 & \mathrm{i} \\ \mathrm{i} & 1 \end{matrix}\right)e^{-\tau g(\eta)\sigma_3} =: \boldsymbol F^{(\infty)}(\eta)e^{-\tau g(\eta)\sigma_3}.
\]
It is a straightforward verification to see that
\begin{itemize}
\item[(a)] $\widehat{\boldsymbol\Psi}^{(\infty)}$ is holomorphic in $\C\setminus(-\infty,0]$;
\item[(b)] $\widehat{\boldsymbol\Psi}^{(\infty)}$ has continuous traces on $(-\infty,0)$ that satisfy $\widehat{\boldsymbol\Psi}_+^{(\infty)} = \widehat{\boldsymbol\Psi}_-^{(\infty)}\left(\begin{matrix} 0 & 1 \\ -1 & 0 \end{matrix}\right)$;
\item[(c)] $\widehat{\boldsymbol\Psi}^{(\infty)}$ satisfies \hyperref[rhhpsiA]{\rhhpsiA}(d) with $\mathcal{O}\big(\eta^{-1}\big)\equiv0$.
\end{itemize}

\subsubsection{Local Parametrix Around $0$}

Denote by $U_0$ the disk centered at $0$ of small enough radius so that $g^2(\eta)$ is conformal in $U_0$. Notice that $g^2(x)>0$ for $\{x>0\}\cap U_0$. Define
\[
\widehat{\boldsymbol\Psi}^{(0)}(\eta;\tau) := \widehat{\boldsymbol E}^{(0)}(\eta)\mathcal{D}\left( \boldsymbol\Psi_\alpha \big( (\tau g(\eta)/2)^2\big)\right),
\]
where $\boldsymbol\Psi_\alpha$ is the solution of \hyperref[rhpsiA]{\rhpsiA}, $\mathcal{D}\left( \boldsymbol\Psi_\alpha\big( (\tau g/2)^2\big)\right)$ is a holomorphic deformation of $\boldsymbol\Psi_\alpha\big( (\tau g/2)^2\big)$ that moves the jumps from $(\tau^2 g^2/4)^{-1}\big(I_\pm\big)$  to $I_\pm$, and $\widehat{\boldsymbol E}^{(0)}$ is holomorphic around $0$ and is given by
\[
\widehat{\boldsymbol E}^{(0)}(\eta) := \boldsymbol F^{(\infty)}(\eta)\mathcal{D}\left( \boldsymbol F^{(\infty)-1}\big( (\tau g/2)^2\big) \right).
\]
Clearly, $\widehat{\boldsymbol\Psi}^{(0)}$ has the following properties:
\begin{itemize}
\item[(a)] $\widehat{\boldsymbol\Psi}^{(0)}$ is holomorphic in $U_0\setminus\big(I_+\cup I_-\cup(-\infty,\infty)\big)$;
\item[(b)] $\widehat{\boldsymbol\Psi}^{(0)}$ satisfies \hyperref[rhhpsiA]{\rhhpsiA}(b) on $\big(I_+\cup I_-\cup(-\infty,\infty)\big)\cap U_0$;
\item[(c)] $\widehat{\boldsymbol\Psi}^{(0)}$ satisfies \hyperref[rhhpsiA]{\rhhpsiA}(c) within $U_0$ (by \hyperref[rhpsiA]{\rhpsiA}(c));
\item[(d)] it holds that
\[
\widehat{\boldsymbol\Psi}^{(0)}(\eta;\tau) = \boldsymbol F^{(\infty)}(\eta)\left(\boldsymbol I + \mathcal{O}\left(\tau^{-1}\right)\right)e^{-\tau g(\eta)\sigma_3}
\]
as $\tau\to\infty$ uniformly on $\partial U_0\setminus\big(I_+\cup I_-\cup(-\infty,\infty)\big)$.
\end{itemize}

\subsubsection{Asymptotics of \hyperref[rhwpsiAB]{\rhwpsiAB}}

Define
\[
\boldsymbol R_\alpha(\eta;\tau) := \widehat{\boldsymbol\Psi}_\alpha(\eta;\tau)\left\{
\begin{array}{ll}
\widehat{\boldsymbol\Psi}^{(0)}(\eta;\tau)^{-1}, & \eta\in U_0, \medskip \\
\widehat{\boldsymbol\Psi}^{(\infty)}(\eta;\tau)^{-1}, & \eta\in\C\setminus \overline U_0.
\end{array}
\right.
\]
Exactly as before we have that
\[
\boldsymbol R_\alpha(\eta;\tau) = \boldsymbol I + \mathcal{O}\left(\frac1{\tau(1+|\eta|)}\right)
\]
as $\tau\to\infty$ uniformly for $\eta\in\C\setminus\left(\partial U_0 \cup \left( \big( I_+ \cup I_- \cup (-\infty,\infty) \big)\setminus \overline U_0 \right)\right)$. Hence, \eqref{betterestimate2} follows from \eqref{st3}.

\bibliographystyle{plain}

\begin{thebibliography}{10}

\bibitem{Ang19}
A.~Angelesco.
\newblock Sur deux extensions des fractions continues alg\'ebraiques.
\newblock {\em Comptes Rendus de l'Academie des Sciences, Paris}, 168:262--265,
  1919.

\bibitem{Ap88}
A.I. Aptekarev.
\newblock Asymptotics of simultaneously orthogonal polynomials in the
  {A}ngelesco case.
\newblock {\em Mat. Sb.}, 136(178)(1):56--84, 1988.
\newblock English transl. in {\it {M}ath. {USSR} {S}b.} 64, 1989.

\bibitem{Ap02}
A.I. Aptekarev.
\newblock Sharp constant for rational approximation of analytic functions.
\newblock {\em Mat. Sb.}, 193(1):1--72, 2002.
\newblock English transl. in {\it {M}ath. {S}b.} 193(1-2):1--72, 2002.

\bibitem{ApLy10}
A.I. Aptekarev and V.G. Lysov.
\newblock Asymptotics of {H}ermite-{P}ad\'e approximants for systems of
  {M}arkov functions generated by cyclic graphs.
\newblock {\em Mat. Sb.}, 201(2):29--78, 2010.

\bibitem{BY10}
L.~Baratchart and M.~Yattselev.
\newblock Convergent interpolation to {C}auchy integrals over analytic arcs
  with {J}acobi-type weights.
\newblock {\em Int. Math. Res. Not.}, 2010.
\newblock Art. ID rnq 026, pp. 65.

\bibitem{uBogatClI}
A.~Bogatskiy, T. Clayes, A.R. Its.
\newblock Hankel determinant and orthogonal polynomials for a {G}aussian weight with a discontinuity at the edge.
\newblock Submitted for publication. \url{http://arxiv.org/abs/1507.01710}

\bibitem{Deift}
P.~Deift.
\newblock {\em Orthogonal Polynomials and Random Matrices: a Riemann-Hilbert
  Approach}, volume~3 of {\em Courant Lectures in Mathematics}.
\newblock Amer. Math. Soc., Providence, RI, 2000.

\bibitem{DIKr11}
P.~Deift, A.R. Its, and I.~Krasovsky.
\newblock Asymptotics of {T}oeplitz, {Hankel}, and {Toeplitz+Hankel}
  determinants with {Fisher-Hartwig} singularities.
\newblock {\em Ann. Math.}, 174:1243--1299, 2011.

\bibitem{DKMLVZ99b}
P.~Deift, T.~Kriecherbauer, K.T.-R. McLaughlin, S.~Venakides, and X.~Zhou.
\newblock Strong asymptotics for polynomials orthogonal with respect to varying
  exponential weights.
\newblock {\em Comm. Pure Appl. Math.}, 52(12):1491--1552, 1999.

\bibitem{DZ93}
P.~Deift and X.~Zhou.
\newblock A steepest descent method for oscillatory {R}iemann-{H}ilbert
  problems. {A}symptotics for the m{K}d{V} equation.
\newblock {\em Ann. of Math.}, 137:295--370, 1993.

\bibitem{FIK91}
A.S. Fokas, A.R. Its, and A.V. Kitaev.
\newblock Discrete {P}anlev\'e equations and their appearance in quantum
  gravity.
\newblock {\em Comm. Math. Phys.}, 142(2):313--344, 1991.

\bibitem{FIK92}
A.S. Fokas, A.R. Its, and A.V. Kitaev.
\newblock The isomonodromy approach to matrix models in {2D} quantum
  gravitation.
\newblock {\em Comm. Math. Phys.}, 147(2):395--430, 1992.

\bibitem{Gakhov}
F.D. Gakhov.
\newblock {\em Boundary Value Problems}.
\newblock Dover Publications, Inc., New York, 1990.

\bibitem{GerKVA01}
W.~{Van Assche}, J.S. Geronimo, and A.B. Kuijlaars.
\newblock Riemann-{Hilbert} problems for multiple orthogonal polynomials.
\newblock In {\em Special functions 2000: current perspective and future
  directions}, number~30 in NATO Sci. Ser. II Math. Phys. Chem., pages 23--59,
  Dordrecht, 2001. Kluwer Acad. Publ.

\bibitem{GRakh81}
A.A. Gonchar and E.A. Rakhmanov.
\newblock On convergence of simultaneous {P}ad\'e approximants for systems of
  functions of {M}arkov type.
\newblock {\em Trudy Mat. Inst. Steklov}, 157:31--48, 1981.
\newblock English transl. in {\it {P}roc. {S}teklov {I}nst. {M}ath.} 157, 1983.

\bibitem{IKOs08}
A.R. Its, A.B.J. Kuijlaars, and J.~\"Ostensson.
\newblock Critical edge behavior in unitary random matrix ensembles and the
  thirty-fourth {P}ainlev\'e transcendent.
\newblock {\em Int. Math. Res. Not. IMRN}, page 67pp., 2008.
\newblock Art. ID rnn017.

\bibitem{IKOs09}
A.R. Its, A.B.J. Kuijlaars, and J.~\"Ostensson.
\newblock Asymptotics for a special solution of the thirty fourth {P}ainlev\'e
  equation.
\newblock {\em Nonlinearity}, 22(7):1523--1558, 2009.

\bibitem{KamvissisMcLaughlinMiller}
S.~Kamvissis, K.T.-R. McLaughlin, and P.~Miller.
\newblock {\em Semiclassical soliton ensembles for the focusing nonlinear
  Schr{\"o}dinger equation}, volume 154 of {\em Annals of Mathematics Studies}.
\newblock Princeton University Press, 2003.

\bibitem{KMcLVAV04}
A.B. Kuijlaars, K.T.-R. McLaughlin, W.~Van Assche, and M.~Vanlessen.
\newblock The {R}iemann-{H}ilbert approach to strong asymptotics for orthogonal
  polynomials on $[-1,1]$.
\newblock {\em Adv. Math.}, 188(2):337--398, 2004.

\bibitem{Mar95}
A.A. Markov.
\newblock Deux d\'emonstrations de la convergence de certaines fractions
  continues.
\newblock {\em Acta Math.}, 19:93--104, 1895.

\bibitem{FMMFSou10}
A.~Foulqui\'e Moreno, A.~Mart\'inez Finkelshtein, and V.L. Sousa.
\newblock On a conjecture of {A. Magnus} concerning the asymptotic behavior of
  the recurrence coefficients of the generalized {Jacobi} polynomilas.
\newblock {\em J. Approx. Theory}, 162:807--831, 2010.

\bibitem{FMMFSou11}
A.~Foulqui\'e Moreno, A.~Mart\'inez Finkelshtein, and V.L. Sousa.
\newblock Asymptotics of orthogonal polynomials for a weight with a jump on
  $[-1,1]$.
\newblock {\em Constr. Approx.}, 33(2):219--263, 2011.

\bibitem{Nik79}
E.M. Nikishin.
\newblock A system of {M}arkov functions.
\newblock {\em Vestnik Moskovskogo Universiteta Seriya 1, Matematika
  Mekhanika}, 34(4):60--63, 1979.
\newblock Translated in \emph{Moscow University Mathematics Bulletin} 34(4),
  63--66, 1979.

\bibitem{Nik80}
E.M. Nikishin.
\newblock Simultaneous {P}ad\'e approximants.
\newblock {\em Mat. Sb.}, 113(155)(4):499--519, 1980.

\bibitem{Nut90}
J.~Nuttall.
\newblock Pad\'e polynomial asymptotic from a singular integral equation.
\newblock {\em Constr. Approx.}, 6(2):157--166, 1990.

\bibitem{Privalov}
I.I. Privalov.
\newblock {\em Boundary Properties of Analytic Functions}.
\newblock GITTL, Moscow, 1950.
\newblock German transl., VEB Deutscher Verlag Wiss., Berlin, 1956.

\bibitem{Ransford}
T.~Ransford.
\newblock {\em Potential Theory in the Complex Plane}, volume~28 of {\em London
  Mathematical Society Student Texts}.
\newblock Cambridge University Press, Cambridge, 1995.

\bibitem{SaffTotik}
E.B. Saff and V.~Totik.
\newblock {\em Logarithmic Potentials with External Fields}, volume 316 of {\em
  Grundlehren der Math. Wissenschaften}.
\newblock Springer-Verlag, Berlin, 1997.

\bibitem{StahlTotik}
H.~Stahl and V.~Totik.
\newblock {\em General Orthogonal Polynomials}, volume~43 of {\em Encycl.
  Math.}
\newblock Cambridge University Press, Cambridge, 1992.

\bibitem{Van03}
M.~Vanlessen.
\newblock Strong asymptotics of the recurrence coefficients of orthogonal
  polynomials associated to the generalized {J}acobi weight.
\newblock {\em J. Approx. Theory}, 125:198--237, 2003.

\bibitem{Ver36}
S.~Verblunsky.
\newblock On positive harmonic functions (second paper).
\newblock {\em Proc. London Math. Soc.}, 40(2):290--320, 1936.

\bibitem{XuZh11}
S.-X. Xu and Y.-Q. Zhao
\newblock Painlev\'e {XXXIV} asymptotics of orthogonal polynomials for the {G}aussian weight with a jump at the edge.
\newblock {\em Applied Mathematics}, 127:67--105, 2011

\bibitem{Zver71}
E.I. Zverovich.
\newblock Boundary value problems in the theory of analytic functions in
  {H}\"older classes on {R}iemann surfaces.
\newblock {\em Russian Math. Surveys}, 26(1):117--192, 1971.

\end{thebibliography}

\end{document}